\documentclass[10pt,a4paper]{amsart}
\usepackage{ae,lmodern}
\usepackage[utf8]{inputenc}
\usepackage[english]{babel}
\usepackage[T1]{fontenc}
\usepackage{amsmath}
\usepackage{amsfonts}
\usepackage{amssymb}
\usepackage{amsthm}
\usepackage{enumerate}
\usepackage{graphicx}
\usepackage[all]{xy}
\usepackage{tikz-cd}
\usepackage{tikz}
\usepackage{tkz-tab}
\usepackage{pgfplots}
\pgfplotsset{compat=1.18}
\usetikzlibrary{positioning, arrows.meta, decorations.text, decorations.markings}
\usetikzlibrary{shapes}
\usetikzlibrary{graphs,calc}
\usetikzlibrary{patterns}
\usetikzlibrary{positioning}
\usepackage{tikz-feynman}
\tikzfeynmanset{compat=1.1.0}

\DeclareMathOperator{\Id}{Id}
\DeclareMathOperator{\Spec}{Spec}

\DeclareMathOperator{\Ker}{Ker}

\DeclareMathOperator{\divi}{div}

\DeclareMathOperator{\GL}{GL}

\DeclareMathOperator{\colim}{colim}

\DeclareMathOperator{\coker}{coker}
\DeclareMathOperator{\gr}{gr}%graduate
\DeclareMathOperator{\Res}{Res}%Residue morphism
\DeclareMathOperator{\Sym}{Sym}%Symmetric power
\DeclareMathOperator{\Orth}{O}%Orthogonal group
\DeclareMathOperator{\SOrth}{SO}%Special Orthogonal group
\DeclareMathOperator{\pr}{pr}% projection
\DeclareMathOperator{\rank}{rk}%rank
\DeclareMathOperator{\disc}{disc}%discriminant
\DeclareMathOperator{\Ext}{\mathsf{\Lambda}}
\DeclareMathOperator{\Vect}{Vect}
\DeclareMathOperator{\Proj}{Proj}

\newcommand{\Mot}{\mathcal{M}}
\newcommand{\PervMot}{\Mot_{\textrm{perv}}}

\newcommand{\DMot}{\mathcal{DM}}
\DeclareMathOperator{\MotLoc}{MotLoc}
\DeclareMathOperator{\mot}{mot}
\newcommand{\Gysin}{!}
\newcommand{\usual}{*}
\newcommand{\motper}{\mathcal{P}^{\mathfrak{m}}}
\newcommand{\drper}{\mathcal{P}^{\mathfrak{dr}}}
\newcommand{\per}{\mathcal{P}}
\newcommand{\rhomot}{\rho^{\mathfrak{m}}}
\newcommand{\Deltadr}{\Delta^{\mathfrak{dr}}}
\DeclareMathOperator{\ev}{ev}
\newcommand{\ConsSh}{\mathcal{D}^{\mathrm{b}}_{\mathrm{c}}}
\DeclareMathOperator{\Loc}{Loc}
\newcommand{\betti}{\textrm{B}}
\newcommand{\derham}{\textrm{dR}}
\newcommand{\fderham}{\mathfrak{dr}}
\newcommand{\fmot}{\mathfrak{m}}
\newcommand{\Rbetti}{R_{\betti}}
\newcommand{\omegabetti}{\omega_\betti}

\newcommand{\an}{\mathrm{an}}
\newcommand{\edgeset}{E_n}
\newcommand{\vertexset}{V_n}
\newcommand{\indexset}{\{1,\ldots,n\}}

\newcommand{\indexsetbar}{\{1,\ldots,n,\infty\}}
\newcommand{\merged}{\textrm{red}}
\newcommand{\momenta}{\underline{p}}
\newcommand{\masses}{\underline{m}}
\newcommand{\massesindex}{N_m}
\newcommand{\momentaindex}{N_p}
\newcommand{\momentaset}{\{1,\ldots,\momentaindex\}}
\newcommand{\massesset}{\{1,\ldots,\massesindex\}}
\newcommand{\momentai}{p_{1,i}}
\newcommand{\momentaij}{p_{i+1,j}}
\newcommand{\kinematics}{\underline{s}}
\newcommand{\propagator}{D}
\newcommand{\Kinematics}{K}
\newcommand{\PKinematics}{\mathbb{P}K}
\newcommand{\goodK}{K^{\textrm{gen}}}
\newcommand{\momentumspace}{\mathbb{A}_{\Gamma_n,\momenta}}
\newcommand{\quadric}{X}
\newcommand{\immersion}{j_{\Gamma_n,\momenta}}
\newcommand{\embeddingspace}{N}
\newcommand{\subquadric}[1]{X_{#1}}
\newcommand{\reducedquadric}{\bar{X}}
\newcommand{\spoint}{x_\infty}%singular point
\newcommand{\G}{\mathcal{G}} %Graham determinant
\newcommand{\sdisc}{\disc_{\pm}}%signed discriminant
\newcommand{\field}{k}
\newcommand{\eucl}{\textrm{eucl}}
\newcommand{\discalg}{\Delta}
\DeclareMathOperator{\PD}{PD}

\newtheorem{theorem}{Theorem}[section]
\newtheorem{lemma}[theorem]{Lemma}
\newtheorem{proposition}[theorem]{Proposition}
\newtheorem{corollary}[theorem]{Corollary}

\theoremstyle{definition}
\newtheorem{definition}[theorem]{Definition}
\newtheorem{example}[theorem]{Example}

\theoremstyle{remark}
\newtheorem{remark}[theorem]{Remark}

\tikzfeynmanset{
    warn luatex=false
}
\begin{document}
\title[Motivic one-loop Feynman integrals in momentum space]{Motivic Galois theory for one-loop Feynman integrals in momentum space}
\author{Ulysse Mounoud}
\address{Institut Montpelliérain Alexander Grothendieck, Université de Montpellier, Montpellier, France}
\email{ulysse.mounoud@umontpellier.fr}

\begin{abstract}
We develop a motivic framework for Feynman integrals of one-loop graphs in momentum space. Its advantage compared to the already existing framework in Feynman representation is that it naturally includes graphs with cuts. To each such graph, we associate a motivic local system over the space of generic  kinematics. Our construction is functorial with respect to the natural operations on graphs: edge contraction and cutting. We compute the weight-graded pieces of the motivic local systems. They are Tate twists of quadratic Artin motives associated with maximally cut quotient graphs. We also derive a formula for the (co)action of the de Rham motivic Galois group, expressed in terms of cut quotient graphs.
\end{abstract}

\maketitle
\tableofcontents
\section{Introduction}
Feynman integrals are a famous computational bottleneck in perturbative quantum field theory, where they arise as coefficients in series expansions of scattering amplitudes \cite{weinzierl_feynman_2022}. They are indexed by graphs equipped with kinematic data of masses and momenta. Their complexity increases rapidly with the first Betti number of the graph, commonly called the loop number. Integrals of one-loop graphs are of a particular importance because they yield next to leading order contributions to computations of scattering amplitudes. They can be expressed in terms of volumes of hyperbolic simplices \cite{davydychev_geometrical_1998,schnetz_geometry_2010}, which, thanks to work of Rudenko \cite{rudenko_goncharov_2022}, leads to explicit formulas in terms of multiple polylogarithms for an important class of one-loop Feynman integrals \cite{ren_one-loop_2024}. However, the structure of one-loop Feynman integrals is still not completely understood.

The simplest example of a one-loop Feynman integral is provided by the ``bubble'' below.
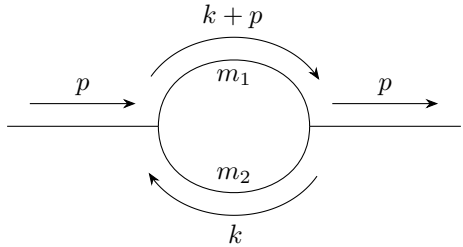
\begin{figure}[htbp]
    \centering
    \begin{tikzpicture}
    \begin{feynman}
        \vertex (i1) at (0,0);
        \vertex (v1) at (2,0);
        \vertex (v2) at (4,0);
        \vertex (o1) at (6,0);

        \diagram* {
            (i1) -- [momentum=$p$] (v1)
                 -- [half left, looseness=1.5, momentum=$k+p$, edge label'=$m_1$] (v2)
                 -- [half left, looseness=1.5, momentum=$k$, edge label'=$m_2$] (v1),
            (v2) -- [momentum=$p$] (o1),
        };
    \end{feynman}
    \end{tikzpicture}
    \caption{Bubble graph with kinematics.}
    \label{figure: bubble graph with momenta introduction}
\end{figure}

\noindent
The (external) kinematics are given by the two masses $m_1, m_2$ and the external momentum $p\in \mathbb{R}^d$, with $d$ an even integer called the space-time dimension. When $d=2$ the associated integral in \emph{momentum representation} is the integral over the internal momentum $k$
\begin{equation}
\label{equation: bubble integral momentum representation}
 I(\parbox{23pt}{
   \begin{tikzpicture}[scale=0.2]
\draw [thick] (0,0) -- (1,0) ;
\draw [thick] (2,0) ellipse (1 and 1);
\draw [thick] (3,0) -- (4,0);
\end{tikzpicture}}):=\frac{1}{\pi}\int_{\mathbb{R}^2}\frac{\mathrm{d}^2k}{(k^2+m_2^2)((k+p)^2+m_1^2)}
\end{equation}
where, in the Euclidean kinematics convention, the square of a vector denotes the square of its Euclidean norm. Each quadratic factor in the denominator is called a propagator and corresponds to an internal edge of the graph. Moreover, we integrate over distributions of momenta on internal edges of the graph such that momentum conservation holds at each vertex.

More generally, we are interested in the integrals associated with the $n$-gon Feynman graph $\Gamma_n$ of figure \ref{figure: n-gon introduction}, which have a similar shape as integral \eqref{equation: bubble integral momentum representation}.
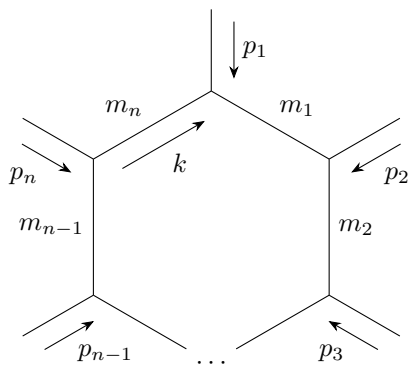
\begin{figure}[htbp]
    \centering
\begin{tikzpicture}[scale=0.9]
\begin{feynman}

  \def\r{2}
  \def\R{3.2}

  \vertex (v1)  at (90:\r);
  \vertex (v2)  at (30:\r);
  \vertex (v3)  at (-30:\r);
  \vertex (vd)  at (-90:\r) {\(\cdots\)};
  \vertex (vn1) at (-150:\r);
  \vertex (vn)  at (150:\r);

  \vertex (p1)  at (90:\R);
  \vertex (p2)  at (30:\R);
  \vertex (p3)  at (-30:\R);
  \vertex (pn1) at (-150:\R);
  \vertex (pn)  at (150:\R);

  \diagram* {
    (v1) -- [edge label=\(m_1\)] (v2)
         -- [edge label=\(m_2\)] (v3)
         -- (vd)
         -- (vn1)
         -- [edge label=\(m_{n-1}\)]  (vn)
         --  [edge label=\(m_n\),momentum'=\(k\)](v1),

    (p1) -- [momentum=\(p_1\)] (v1),
    (p2) -- [momentum=\(p_2\)] (v2),
    (p3) -- [momentum=\(p_3\)] (v3),
    (pn1) -- [momentum'=\(p_{n-1}\)] (vn1),
    (pn)  -- [momentum'=\(p_n\)] (vn),
  };

\end{feynman}
\end{tikzpicture}
\caption{Feynman $n$-gon with kinematics.}
    \label{figure: n-gon introduction}
\end{figure}

\noindent
The integrals associated to $\Gamma_n$ are the integrals:
\begin{equation}
    \label{equation: general integral momentum representation}
I(\Gamma_n,\underline{\nu},d)=\frac{1}{\pi^{\frac{d}{2}}}\int_{\mathbb{R}^{d}}\frac{\mathrm{d}^d k}{D_1^{\nu_1}\ldots D_n^{\nu_n}}
\end{equation}
where $d$ is the space-time dimension and for $1\leq i\leq n$, \[D_i:=(k+p_1+\cdots+p_i)^2+m_i^2\] is the propagator associated to the $i^{\textrm{th}}$ edge, which is raised to some integral power $\nu_i$. Remarkably, for fixed $n$, these integrals span a finite dimensional vector space over the field of rational functions. This is essentially\footnote{One has to relate the cohomology groups for different values of $d$ for the argument to apply.} a consequence of the finiteness of algebraic de Rham cohomology \cite{grothendieck_rham_1966}, once we interpret integrals \eqref{equation: general integral momentum representation} as the pairing of the algebraic de Rham cohomology class given by the differential form and the Betti homology class given by the domain of integration. Even more remarkably, this finite dimensional vector space has the natural structure of a group representation. This is a consequence of the Tannakian formalism \cite{deligne_tannakian_1982} once we promote our cohomology groups to Nori motives \cite{huber_periods_2017}, which, in some precise sense, form a universal cohomology theory for algebraic varieties.

The same arguments apply to any period (integral of a rational function) and lead to a Galois theory of periods\footnote{More precisely of motivic periods, a formal replacement of periods.} \cite{kontsevich_periods_2001} which extends the usual Galois theory of algebraic numbers \cite{andre_galois_2009}. The application of this theory to Feynman integrals can be traced back to the ``cosmic Galois group'' of Cartier \cite{cartier_mad_2001, goncharov_galois_2005}, and to the seminal paper of Bloch, Esnault and Kreimer \cite{bloch_motives_2006}, which was the first to associate motives to a particular class of Feynman integrals that do not depend on external kinematics. More general integrals such as \eqref{equation: general integral momentum representation} that depend on kinematic parameters can be described as families of periods \cite{brown_notes_2017-1}. They have mainly been studied using the \emph{Feynman representation}, notably by Brown who developed a motivic framework for convergent Feynman integrals with generic kinematics \cite{brown_feynman_2017-1}. Note that Feynman integrals enjoy several integral representations \cite{weinzierl_feynman_2022} which lead to different geometries. In the Feynman representation, the integral \eqref{equation: general integral momentum representation} is rewritten as an integral over the Feynman parameters, which are associated to each edge of the graph. Both Feynman and momentum representations have their advantages. In this paper, we attack the problem of a motivic framework for the momentum representation. We do no consider any regularisation and avoid convergence problems by restricting to generic, non-vanishing kinematics.

We define the motive $\mot(\Gamma_n)$ that is naturally associated to $\Gamma_n$ and lift natural operations on one-loop graphs (edge cutting and edge contraction, also called pinching) to morphisms of motives. This requires extending the definition of $\mot(\Gamma_n)$ to graphs with a subset of cut edges $(\Gamma_n,\gamma)$. Momentum representation turns out to be particularly adapted to do this. It is to our knowledge the first time that motives of such graphs with cuts are defined. We give a basis of the de Rham realisation $\mot_{\derham}(\Gamma_n)$ in terms of quotient graphs with all edges cut, and write explicitly the corresponding integrals as Feynman integrals of quotient graphs in a lower dimension $d$. Finally, we define de Rham periods indexed by cut graphs, and we use them to express the coaction.

One of the key protagonists of our study is the weight filtration, which is a canonical increasing filtration $W_\bullet$ on each motive. Our main access to the structure of $\mot(\Gamma_n)$ is the computation of the associated weight-graded pieces:
\begin{equation}
\label{equation: weight-graded motive definition}
\gr^W_i \mot(\Gamma_n):=W_{i}\mot(\Gamma_n)/W_{i-1}\mot(\Gamma_n).
\end{equation}
In particular, this is enough to compute the de Rham realisation $\mot_{\derham}(\Gamma_n)$ in our case. Before presenting the results in more detail, let us briefly expand on our motivic framework.

\subsection{Motivic framework}

Let $S$ be a smooth geometrically connected variety over $\mathbb{Q}$, that we think of as the space of kinematic parameters. We will work in the Tannakian category $\MotLoc(S)$ of motivic local systems over $S$. It was defined by Terenzi in \cite{terenzi_tensor_2026} as a full subcategory of the category of perverse Nori motives $\Mot_{\textrm{perv}}(S)$ of Ivorra and Morel \cite{ivorra_four_2024}.

\subsubsection{Nori motives}
We use the formalism of Nori motives, because it provides us with abelian categories of mixed motives with good functoriality properties. Indeed, the derived categories of Nori motives $\DMot(S)$ enjoy a six-functor formalism \cite{ivorra_four_2024, terenzi_tensor_2026}, compatible with the Betti realisation functor:
\[\Rbetti:\DMot(S) \to \ConsSh(S^{\an};\mathbb{Q})\]
where $\ConsSh(S^{\an})$ is the bounded derived category of constructible sheaves with rational coefficients on the complex variety $S^{\mathrm{an}}$ obtained by analytification of $S$. Compared to other triangulated categories of motives, the advantage of $\DMot(S)$ is that it also possesses perverse and constructible $\mathrm{t}$-structures \cite{tubach_nori_2025}, which are motivic lifts of the corresponding $\mathrm{t}$-structures on $\ConsSh(S^{\an};\mathbb{Q})$. By taking the hearts of these $\mathrm{t}$-structures, we obtain the abelian categories of constructible and perverse Nori motives, $\Mot_{\textrm{cons}}(S)$ and $\Mot_{\textrm{perv}}(S)$. When $S=\Spec k$ is a point, they coincide with the classical Tannakian category of Nori motives over a field $\Mot(k)$. Moreover, the perverse Nori motives whose Betti realisation is a (shifted) local system are called motivic local systems and they form a full subcategory $\MotLoc(S)\subset\PervMot(S)$ which is a Tannakian category \cite{terenzi_tensor_2026}, and behaves similarly to $\Mot(k)$.

\subsubsection{Motivic periods and Galois groups}
Let $k\subset \mathbb{C}$ be algebraic over $\mathbb{Q}$. The definition of (motivic) periods over $\Spec k$ through Nori motives is classical \cite{huber_periods_2017}. We define the periods of a Nori motive $M$ as the vector space:
\begin{equation*}
    \label{equation: periods of M}
\per(M):=\mathbb{Q}\langle \int_\sigma \omega,\, \sigma \in M_{\textrm{B}}^\lor, \, \omega \in M_{\derham} \rangle \subset \mathbb{C}.
\end{equation*}
Note that we used the notation $\int$ because the pairing between de Rham cohomology and Betti homology is defined using integration. There is a way to replace this transcendental construction by a formal, algebraic one, and to consider instead their formal pairing, usually denoted by $[M,\sigma,\omega]^\mathfrak{m}$ \cite{brown_notes_2017-1}, and that we will often shorten as $\int^{\mathfrak{m}}_\sigma \omega$. It is an element of the ring of motivic periods over $k$ denoted $\motper(k)$. This ring is defined by generators $\int^{\mathfrak{m}}_\sigma \omega$ with relations
\begin{equation}
    \label{equation: relation motivic periods}
\int^{\mathfrak{m}}_\sigma \phi_{\derham}(\omega)=\int^{\mathfrak{m}}_{\phi_B^\lor(\sigma)} \omega
\end{equation}
for all morphisms $\phi$ from $M$ to $N$ in $\Mot(k)$, and $\omega\in M_{\derham}$, $\sigma \in N_B^\lor$. Typically, $\phi_{\derham}$ is a pullback on differential forms, and $\phi_B^\lor$ is the pushforward in singular homology, in which case equation \eqref{equation: relation motivic periods} is the change of variables relation. Because equation \eqref{equation: relation motivic periods} is satisfied by the usual integration pairing, we have an evaluation morphism:
\begin{equation}
    \label{equation: evaluation morphism periods}
\motper(k)\to \mathbb{C}.
\end{equation}
The strength of this formal construction comes from the fact that $\Mot(k)$ is a Tannakian category, and the de Rham and Betti realisations are exact faithful tensor functors. Indeed, we may similarly define the ring of de Rham periods $\drper(k)$ which is generated by formal pairings $[M,\nu,\omega]^{\mathfrak{dr}}$ of de Rham homology and cohomology classes, with similar relations as \eqref{equation: relation motivic periods}. The Tannakian formalism shows that $\Spec(\drper)$ is the universal pro-algebraic group that acts naturally on the de Rham realisation. If $M$ is a Nori motive and $(e_i)_{1\leq i \leq r}$ is a basis of $M_{\derham}$, then the action, or rather the coaction on $M_{\derham}$ is given by:
\begin{equation}
    \label{equation: coaction on M_dR}
    \begin{array}{ccccc}
    \rho_{\textrm{univ}}(M) & : & M_{\derham} & \to & M_{\derham}\otimes \drper \\
         &   & \omega & \mapsto & \sum_{i=1}^r e_i \otimes [M,e_i^\lor,\omega]^{\mathfrak{dr}}
    \end{array}
\end{equation}
It induces the natural coproduct $\Deltadr$ on $\drper$ which induces the group law on $\Spec(\drper)$, as well as the natural coaction $\rhomot$ of $\drper$ on $\motper$. If $M$ is a Nori motive, we define its de Rham motivic Galois group $G_{\derham}(M)$ as the image of $\Spec(\drper)$ in $\GL(M_{\derham})$. It is the spectrum of the subring of $\drper$ generated by de Rham periods of $M$. Its action on $M_{\derham}$ is still given by formula \eqref{equation: coaction on M_dR}.
\subsubsection{Motivic local systems}
If $S$ is a smooth geometrically connected variety over $k$, then we can consider the Tannakian category of motivic local systems $\MotLoc(S)$ instead \cite{terenzi_tensor_2026}. It is endowed with a Betti realisation with values in local systems on the analytification:
\[\Rbetti : \MotLoc(S)\to \Loc(S^{\an}).\]
Any $\mathbb{C}$-point $x$ of $S$ yields a fiber functor on $\Loc(S^{\an})$. By composition with $\Rbetti$, it yields a fiber functor: \[\omegabetti:\MotLoc(S) \to \Vect_\mathbb{Q}.\]
Unfortunately, the de Rham realisation has not been constructed yet. It should take value in algebraic vector bundles with a flat connection. If we assume for simplicity that $S$ is affine, then vector bundles correspond to projective modules on the ring of functions and we should get an exact tensor functor:
\begin{equation}
\label{equation: de Rham fiber functor MotLoc}
\omega_{\derham}:\MotLoc(S) \to \Proj_{\mathcal{O}_S(S)}.
\end{equation}
From there, the Tannakian formalism \cite{deligne_tannakian_1982} would produce $\motper(S)$ and $\drper(S)$ as before, as well as the motivic coproduct and coaction. Moreover, there should be a comparison isomorphism yielding a ring morphism to holomorphic functions on the universal cover:
\[\ev:\motper(S)\to \mathcal{O}^\an(\tilde{S}^{\an}).\]
We will define motivic local systems, but we will bypass the construction of the de Rham realisation as follows. A $\mathbb{C}$-point $x$ of $S$ can be viewed as a point of $S$ with residue field $k_x$ embedded in $\mathbb{C}$. There is a pullback morphism
\[x^*:\MotLoc(S)\to \Mot(k_x)\]
which takes us to the category of Nori motives over a subfield of $\mathbb{C}$, where the de Rham realisation is defined. For each value of the parameters $x\in S(\mathbb{C})$, we will define motivic periods and de Rham periods in $\motper(k_x)$ and $\drper(k_x)$ and compute the motivic coaction. Hence, even though our motives are defined globally on the base, our Tannakian formalism is only pointwise.

\subsection{Results}
Recall that $\Gamma_n$ is the $n$-gon graph, and $d$ is an even integer called the space-time dimension. We are interested in the integrals \eqref{equation: general integral momentum representation}. By invariance under orthogonal transformation, they only depend on the kinematic invariants $((p_i\cdot p_j)_{1\leq i,j\leq n},\masses^2)$ which belong to the space of kinematics $\Kinematics_{n}$. We assume that these kinematic invariants belong to the subvariety of $d$ dimensional generic kinematics $\Kinematics_{n,d}^{{\textrm{gen}}}$ which is a smooth affine variety over $\mathbb{Q}$ (see definition \ref{definition: space of generic kinematics}). If the sum of exponents $\nu$ is greater than $d/2$ and the kinematics are furthermore Euclidean (a positivity condition, see definition \ref{definition: space of euclidean kinematics}), then the integral converges. Our first result is the construction of motivic local systems underlying this integral. Because the integrals we consider are homogeneous in the kinematics we work over the projectivisation of our spaces of kinematics, which we denote as $\PKinematics$.
\begin{theorem}
\label{theorem: first result introduction}
We define motivic local systems $\mot'(\Gamma_n)\subset \mot(\Gamma_n)$ over $\PKinematics_{n}^{\textrm{gen}}$ such that for all even $d$ and for all Euclidean generic kinematics $(\kinematics,\masses)\in \Kinematics_{n,d}^{\eucl}$:
\begin{enumerate}
    \item $W_d\mot'(\Gamma_n)$ and $W_d\mot(\Gamma_n)$ extend to $\PKinematics_{n,d}^{\textrm{gen}}$; 
    \item $I(\Gamma_n,d,\nu,\kinematics,\masses)$ is a period of $W_d\mot(\Gamma_n)_{\kinematics,\masses}$ if $\nu>d/2$;
    \item $I(\Gamma_n,d,\nu,\kinematics,\masses)$ is a period of $W_d\mot'(\Gamma_n)_{\kinematics,\masses}$ if $\nu\geq d$.
\end{enumerate}
\end{theorem}
We call $\mot(\Gamma_n)$ and $\mot'(\Gamma_n)$ the full and the reduced motive respectively. The main ingredient of the construction is a compactification of the affine space into a smooth projective quadric $X$. This compactification is well-known in the physics literature as a particular case of the embedding formalism, and was already used in \cite{abreu_diagrammatic_2017-1}. Each propagator $\propagator_i$ has a linear expression in these new projective coordinates, and defines a hyperplane section $X_i$ of $X$. Then, up to some twist, $\mot'(\Gamma_n)$ is
\[H^d\Big(X\setminus \bigcup_{1\leq i \leq n} X_i\Big)\]
which is independent of $d$ for $d\geq n$. The full motive $\mot(\Gamma_n)$ is a slight variation (see definition \ref{definition: graph motives pointwise}) that takes into account the pole at the boundary of the compactification.

We give a graphical description of the structure of these motives using the operations of pinching and cutting edges. If $\gamma$ is a subset of edges of $\Gamma_n$ then we may pinch edges of $\gamma$ to produce the quotient graph $\Gamma_n/\gamma$. This quotient graph can be identified with a $k$-gon graph for $k\leq n$ to which Theorem~\ref{theorem: first result introduction} applies. To such an operation we associate morphisms:
\[\mot(\Gamma_n/\gamma)\to \mot(\Gamma_n)  \ \text{ and } \mot'(\Gamma_n/\gamma)\to \mot'(\Gamma_n).\]
Alternatively, we may also consider $\gamma$ as additional data on the graph called "cuts", a terminology from the physics literature. We will denote the resulting cut graph as $(\Gamma_n,\gamma)$. There is a natural way to define $\mot(\Gamma_n,\gamma)$ and $\mot'(\Gamma_n,\gamma)$ as well as residue morphisms:
\[\mot(\Gamma_n) \to \mot(\Gamma_n,\gamma) \ \text{ and } \mot'(\Gamma_n) \to \mot'(\Gamma_n,\gamma).\]
These two operations of pinching and cutting commute with each other. Combining them gives us access to all weight-graded pieces (see defining equation \eqref{equation: weight-graded motive definition}) of $\mot(\Gamma_n)$ and $\mot'(\Gamma_n)$.
We get the following theorem for the reduced motive (a similar result holds for the full motive, see Theorem~\ref{theorem: diagrammatic weight-graded computation}).
\begin{theorem}
\label{theorem: weight graded introduction}
The cutting and pinching morphisms induce an isomorphism:
    \[\gr^W\mot'(\Gamma_n)\simeq \bigoplus_\gamma \mot'(\Gamma_n/\gamma^c,\gamma) \]
    where the sum runs over even subsets of edges $\gamma$ of $\Gamma_n$. Moreover, for each such $\gamma$ there is an isomorphism:
    \[\mot'(\Gamma_n/\gamma^c,\gamma)\simeq \chi_\gamma\left(-\#\gamma/2\right)\]
    where $\chi_\gamma$ is an Artin motive attached to a quadratic character.
\end{theorem}

For example, for $n=d=4$ we get the following formula for the ``box'':
\begin{equation*}
    \gr^W \mot'(\parbox{18pt}{
    \begin{tikzpicture}[scale=0.2]
    \foreach \a in {0,90,180,270}{
    \draw [rotate=\a,thick] (-1,-1)--(1,-1);
    \draw [rotate=\a,thick] (1,-1)--(1.5,-1.5);
    }
\end{tikzpicture}
    })=\mot'(\parbox{18pt}{
    \begin{tikzpicture}[scale=0.2]
    \foreach \a in {0,90,180,270}{
    \draw [rotate=\a,thick] (-1,-1)--(1,-1);
    \draw [rotate=\a] (0,-1.5)--(0,-0.5);
    \draw [rotate=\a,thick] (1,-1)--(1.5,-1.5);
    }
\end{tikzpicture}
    }) \oplus \bigoplus_{1\leq i<j \leq 4} \mot'(\parbox{23pt}{
   \begin{tikzpicture}[scale=0.2]
\draw [thick] (0,0) -- (1,0) ;
\draw [thick] (2,0) ellipse (1 and 1);
\draw [thick] (2,1) node {$|$};
\draw [thick] (2,-1) node {$|$};
\draw [thick] (3,0) -- (4,0);
\end{tikzpicture}}) \oplus \mot'(\cdot)
\end{equation*}
where by convention $\mot'(\cdot)$ is the trivial motive $\mathbb{Q}(0)$. In particular, the theorem determines the weight of the motives.
\begin{corollary}
\label{corollary: weight}
The full motive has weight $2\lfloor(n+1)/2\rfloor$ and the reduced motive has weight $2\lfloor n/2 \rfloor$. Their bottom weight part have rank $1$. The top weight part of the full motive (resp. of the reduced motive) has rank $1$ if $n$ is odd (resp. if $n$ is even).
\end{corollary}
The theorem also implies that one can build the weight $d$ part of the motive from quotient graphs with bounded number of edges. Together with Theorem~\ref{theorem: first result introduction}, the corollary below says that if one is interested only in dimension $d$ integrals, then one can reduce to graphs with at most $d$ edges, as expected. Actually, this is used to prove that $W_d\mot'(\Gamma_n)$ and $W_d\mot(\Gamma_n)$ extend to $\Kinematics_{n,d}^{\textrm{gen}}$ in Theorem~\ref{theorem: first result introduction}.
\begin{corollary}
\label{corollary: fixed d introduction}
The pinching morphisms induce isomorphisms:
\begin{align}
    \label{equation: weight d as colimit of quotient graphs introduction}
\colim_{\#\gamma\leq d}(\mot'(\Gamma_n/\gamma^c))&\to W_d\mot'(\Gamma_n)\\
\colim_{\#\gamma\leq d}(\mot(\Gamma_n/\gamma^c))&\to W_d\mot(\Gamma_n)
\end{align}
where the colimit is over the poset of subsets of edges $\gamma$ with cardinal at most $d$, and if $\gamma_0\subset \gamma_1$, the associated morphism is the pinching morphism.
\end{corollary}
Finally, the theorem implies that the associated pure weight parts are simply Tate twists on some finite étale cover of the base $\Kinematics_{n}^{\textrm{gen}}$. One says that the motives are of mixed Artin—Tate type. In this case, the Hodge filtration on the de Rham realisation yields an isomorphism:
\[\mot_{\derham}(\Gamma_n)\simeq \gr^W_\bullet \mot_{\derham}(\Gamma_n).\]
Hence, Theorem~\ref{theorem: weight graded introduction} also yields bases of the de Rham realisation of our motives. Using these bases, we associate to $\Gamma_n$ and to Euclidean kinematics $(\kinematics,\masses)$ a motivic period $I^{\fmot}(\Gamma_n,\kinematics,\masses)$, which is a motivic period of the full motive for any $n\geq 2$. For $n$ even, it is actually a motivic period of the reduced motive. Up to a prefactor, it evaluates to
\[I(\Gamma_n,d=2\lceil n/2 \rceil,\underline{\nu}=\underline{1},\kinematics,\masses)=\frac{1}{\pi^{\lceil n/2 \rceil}}\int_{\mathbb{R}^{2\lceil n/2 \rceil}}\frac{\mathrm{d}^{2\lceil n/2 \rceil} k}{D_1\cdots D_n}.\]
This also applies to any quotient of $\Gamma_n$ with at least $2$ edges. For the `tadpole'' graph with one edge, this does not work, but we can naturally associate a motivic period to linear combinations of one edge quotient graphs if the sum of coefficients is zero. We also define de Rham periods associated with one-loop cut graphs $(\Gamma_n,\gamma)$ and generic kinematics $(\kinematics,\masses)$, that we denote $I^{\mathfrak{dr}}(\Gamma_n,\gamma,\kinematics,\masses)$. We can then express the motivic coaction.
\begin{theorem}
\label{theorem: motivic coaction introduction}
The de Rham motivic coaction is given by the formula:
\[\rhomot I^{\mathfrak{m}}(\Gamma_n,\kinematics,\masses)= \sum_{\gamma\subset E_{\Gamma_n}} I^{\mathfrak{m}}(\Gamma_n/\gamma^c,\kinematics,\masses) \otimes I^{\mathfrak{dr}}(\Gamma_n,\gamma,\kinematics,\masses).\]
If $n$ is even, then the terms with $\#\gamma$ odd cancel. If $n$ is odd, then we use convention \eqref{equation: zero sum convention} for the sum of terms with $\#\gamma=1$.
\end{theorem}
\begin{remark}
We have formulated the theorem as a pointwise coaction. The only missing ingredient for upgrading it to a global coaction is the construction of the de Rham realisation. It should be possible to extend to the de Rham setting the methods used in \cite{tubach_nori_2025} to define the Hodge realisation. Alternatively, as suggested by Tubach\footnote{Private communication.}, one may obtain the result by formulating the appropriate universal property of Nori motives.\end{remark}
\subsection{Relation to other work}
\subsubsection{Dimension four}
If one is only interested in the case $d=4$, then one can reduce to the study of the bubble, the triangle and the box diagrams by Theorem~\ref{theorem: first result introduction} and Corollary~\ref{corollary: fixed d introduction}. In that case, the relevant coaction formula was established by Tapušković \cite{tapuskovic_motivic_2021} in the framework developped by Brown \cite{brown_feynman_2017-1}. Tapušković writes the de Rham periods in the coaction as de Rham logarithms, rather than de Rham periods associated to cut graphs. Moreover, he also considers cases where some masses vanish, in which case it is not clear that one can write the coaction formula graphically without regularisation.

\subsubsection{Dimensional regularisation}
Recall that dimensional regularisation makes it possible to deal with divergent integrals by allowing the space-time dimension $d$ to vary near an integer value by an $\epsilon$ parameter. Feynman integrals in dimensional regularisation are then Laurent series in $\epsilon$ whose coefficients are (families of) periods. In the physics literature, coaction formulas in the dimensional regularisation case were given in \cite{abreu_diagrammatic_2017-1}. These are formulas at the level of series of periods, that encapsulates the coaction formulas of each coefficient. By setting $\epsilon$ to $0$ one can retrieve the coaction formulas of the convergent Feynman integrals that we considered in this paper. The formulas in \cite{abreu_diagrammatic_2017-1} are based on robust evidence: numerous computations up to high order of $\epsilon$. Unfortunately, the appropriate motivic framework to establish them on solid mathematical grounds is still missing. As hinted at in \cite{brown_lauricella_2023}, there should exist a suitable Tannakian category that extends the usual category of Nori motives, and which naturally produces these series in $\epsilon$ as "generalised" periods. This is similar to how exponential periods can be studied via exponential motives.

On another note, recent work of Durh and Mork \cite{duhr_analytic_2025} computes higher order terms in the $\epsilon$ expansion using computations of integral dimension Feynman integrals \cite{ren_one-loop_2024}. Remarkably, even if one is interested only in results in dimension $d=4$, results in every even dimension are relevant because of dimensional shift-identities.
\subsubsection{Mixed-Tate motives for volumes of hyperbolic simplices}
There is a rich literature on the interpretation of one-loop Feynman integrals as volumes of hyperbolic simplices \cite{davydychev_geometrical_1998,schnetz_geometry_2010,bourjaily_all-mass_2020-1,ren_one-loop_2024}. In his work on volumes of hyperbolic manifolds \cite{goncharov_volumes_1999}, Goncharov defined mixed Tate motives over $\overline{\mathbb{Q}}$ associated to hyperbolic simplices defined over $\overline{\mathbb{Q}}$. They are almost the same as our motives $\mot'(\Gamma)$. The only difference is that we are working over a basis of parameters, which gives us access to variational information and to the Artin motives associated to quadratic characters. This comes at the cost of working in the category of Nori motives instead of the category of mixed Tate motives\footnote{or rather, mixed Artin—Tate motives \cite{scholbach_mixed_2011}, which are the motives that become mixed-Tate on a finite étale cover}, which is not well-defined in our context. Moreover, instead of motivic periods, Goncharov works with framed motives, which are essentially de Rham periods of mixed-Tate motives.
\subsection{Organisation of the paper}
The paper is organised as follows. In section \ref{section: one-loop feynman integrals in momentum space} we define precisely the Feynman diagrams and integrals that we will be looking at, as well as edge pinching and cutting. In section \ref{section: compactification of the momentum space} we define a convenient compactification of the integration domain, as well as spaces of generic and Euclidean kinematics. In section \ref{section: motivic periods} we lift our integrals to motivic periods, and exhibit the relevant cohomology groups. In section \ref{section: motive of a generic hyperplane arrangement complement in a smooth quadric} we compute the cohomological Nori motive of a generic hyperplane arrangement complement in a smooth projective quadric thanks to a spectral sequence. This will give the reduced motives. We show how to reduce the dimension of the quadric when there are few hyperplanes, and we carefully track the Betti classes, which will be useful for the computation of the motivic coaction. In section \ref{section: a mild variation}, we adapt the arguments of the previous section to the case where one hyperplane section is singular, which corresponds to the case of the full motive. Finally, in section \ref{section: motives for one loop graphs} we define the full and the reduced motives, and transcribe the results of our computations to a diagrammatic description of their structure. We give an explicit basis of their de Rham realisation, and a formula for the de Rham motivic coaction.
\subsection{Acknowledgements}
This article was written during my PhD at the IMAG in Montpellier with the
support of the ANR Cyclades (Projet ANR-23-CE40-0011). I owe many thanks to my advisor, Clément Dupont, for his constant guidance and support throughout the writing process. Many thanks also to Swann Tubach, who answered my numerous questions about Nori motives and their realisations with great care, and to Sofian Tur-Dorvault and Nikola Tomic, for precious mathematical conversations.
\section{One-loop Feynman integrals in momentum space}
\label{section: one-loop feynman integrals in momentum space}
We begin by recalling the expression in momentum space of the Feynamn integral associated with a one-loop Feynman graph.
\subsection{Kinematic configurations}
We introduce the kinematic parameters before the graphs, because we want the same space of parameters for different graphs. These kinematics consist of momenta and masses. To take into account homogeneity (in the physics sense) we will view them as elements of some abstract vector spaces, rather than elements of $\mathbb{R}^d$ and $\mathbb{R}$. Moreover, we work over some subfield $\field$ of $\mathbb{C}$. Hence, we let $L$ be a one dimensional vector space over $\field$, whose elements are masses. We also let $M$ be a vector space over $\field$ of even dimension $d>0$, whose elements are momenta. Typically, a physically meaningful example would be to let $M$ be the Minkowski space. Similarly, we consider $q$ a non-degenerate quadratic form on $M$ with value in $L^{\otimes 2}$. However, we do not assume anything about its signature. If $p_1$ and $p_2$ are two momenta, then we let $p_1^2:=q(p_1)$, and we denote the bilinear form associated to $q$ as $p_1\cdot p_2$.

Our preferred choice for $(M,L,q)$ will be the standard Euclidean quadratic space of dimension $d$ over $\field$, $(k^d,k,q_{\textrm{eucl}})$ where:
\begin{equation}
\label{equation: euclidean quadratic form}
  q_{\textrm{eucl}}(x_1,\ldots,x_d)=\sum_{i=1}^d x_i^2
\end{equation}

Moreover, we consider positive integers $\massesindex$ and $\momentaindex$, masses $\masses=(m_1,\ldots,m_{\massesindex})$ in $L^{\massesindex}$ and momenta $\momenta=(p_1,\ldots,p_{\momentaindex})$ in $M^{\momentaindex}$ satisfying momentum conservation:
\begin{equation}
    \label{equation: momentum conservation}
    \sum_{i=1}^{\momentaindex} p_i=0.
\end{equation}
We call $(M,L,q,\momenta,\masses)$ a kinematic configuration, and often shorten it as $(\momenta,\masses)$. If $(M,L,q)$ is the standard Euclidean quadratic space, we call it a Euclidean kinematic configuration.

\subsection{Momentum space of a Feynman graph}
\begin{definition}
A massive oriented Feynman graph $\Gamma$ is the datum of a connected graph, encoded in a finite edge set $E_\Gamma$, a finite vertex set $V_\Gamma$ and a map \[E_\Gamma\to V_\Gamma\times V_\Gamma\] that assigns to each edge its initial and terminal vertex, together with attachment maps
\[a_{\Gamma,p} : \momentaset \to V_\Gamma 
\text{ and }
a_{\Gamma,m} : E_\Gamma \to \massesset\]
that assign vertices to the momenta, and masses to the edges.
\end{definition}

\begin{remark}
Be careful that our definition is not completely standard. More generally, one would distinguish between massive and massless edges, and impose that some of the external momenta have zero norm ($p_i^2=0$). Moreover, fixing an orientation of the edges is not necessary but simplifies the exposition.
\end{remark}

The graph below is an instance of a Feynman graph, with $\momentaindex=\massesindex=4$. We say that the bottom-right vertex receives momenta $p_2$ and $p_3$, and that the bottom edge has mass $m_3$.
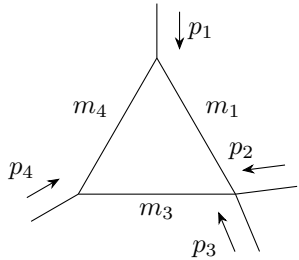
\begin{figure}[htbp]
    \centering
\begin{tikzpicture}[scale=0.6]
\begin{feynman}

  \def\r{2}
  \def\R{3.2}

  \vertex (v1) at (90:\r);
  \vertex (v2) at (-30:\r);
  \vertex (v3) at (210:\r);

  \vertex (p1) at (90:\R);

  \vertex (p2) at (-15:\R);
  \vertex (p3) at (-45:\R);

  \vertex (p4) at (210:\R);

  \diagram* {
    (v1) -- [edge label=\(m_1\)] (v2)
         -- [edge label=\(m_3\)] (v3)
         -- [edge label=\(m_4\)] (v1),

    (p1) -- [momentum=\(p_1\)] (v1),

    (p2) -- [momentum'=\(p_2\)] (v2),
    (p3) -- [momentum=\(p_3\)] (v2),

    (p4) -- [momentum=\(p_4\)] (v3),
  };

\end{feynman}
\end{tikzpicture}
\caption{Triangle graph with four external legs}
\label{figure: triangle graph example}
\end{figure}

We define the affine momentum space of $\Gamma$ as the affine subspace $\mathbb{A}_{\Gamma,\momenta}$ of elements in $M^{E_\Gamma}$ such that momentum conservation holds at each vertex (i.e. the sum of momenta from edges and half-edges is zero). More precisely, we let
\begin{equation}
\label{equation: boundary operator oriented graph}
  \partial_\Gamma : M^{E_\Gamma} \to M^{V_\Gamma}
\end{equation}
be the linear map that computes the sum of incoming momenta minus the sum of outgoing momenta at each vertex. Its kernel and cokernel are respectively the first and zeroth homology groups of the graph. We assumed that $\Gamma$ is connected, hence we get the exact sequence
\begin{equation}
\label{equation: exact sequence oriented graph}
    \begin{tikzcd}
0 \ar[r] & H_1\left(\Gamma,M\right)  \ar[r] & M^{E_\Gamma} \ar[r,"\partial_\Gamma"] & M^{V_\Gamma} \ar[r] & M \ar[r] & 0
\end{tikzcd}.
\end{equation}
The last map is just the sum of all momenta. For each vertex $v\in V_\Gamma$, we define \[p_v:=\sum_{i\in a_{\Gamma,p}^{-1}(v)} p_i.\]
Then equation \eqref{equation: momentum conservation} implies \[\sum_{v\in V_\Gamma}p_v=0\] which, by exactness of \eqref{equation: exact sequence oriented graph} implies that $(p_v)_{v\in V_\Gamma}$ is in the image of $\partial_\Gamma$.

\begin{definition}
    The affine momentum space of $\Gamma$ is the affine space
    \begin{equation}
        \label{equation: definition momentum space graph}
        \mathbb{A}_{\Gamma,\momenta}:=\partial_{n,M}^{-1}((-p_v)_{v\in V_\Gamma})
    \end{equation}
    viewed as an algebraic variety over $k$.
\end{definition}

\subsection{Massive $n$-gon}
We will study the family of massive $n$-gons for $n$ a positive integer. We assume that $\momentaindex=\massesindex=n$ . We define the $n$-gon of figures \ref{figure: n-gon introduction} and \ref{figure: n-gon momenta} as follows.
\begin{definition}
    We let $\Gamma_n$ be the directed graph whose vertex set and edge set are \[\vertexset:=\{v_i,\,i\in \indexset\} \text{ and } \edgeset:=\{e_i,\, i\in \indexset\},\] and where for all $i\in \indexset$ the edge $e_i$ goes from $v_i$ to $v_{i+1}$ (modulo $n$). Moreover, for $1\leq i \leq n$ the vertex $v_i$ receives the momentum $p_i$, and the edge $e_i$ has mass $m_i$.
\end{definition}
In this case, the linear map \eqref{equation: boundary operator oriented graph} is simply:
\begin{equation}
\begin{array}{ccccc}
       \partial_{n,M} & : &  M^{\edgeset} & \to & M^{\vertexset}  \\
     &  & (k_i)_{i\in \indexset} & \mapsto & (k_{i-1}-k_{i})_{i\in \indexset}
\end{array}
\end{equation}
Moreover, the long exact sequence \eqref{equation: exact sequence oriented graph} becomes:
\begin{equation}
\label{equation: exact sequence n-gon}
    \begin{tikzcd}
0 \ar[r] & M \ar[r] & M^{\edgeset} \ar[r,"\partial_{n,M}"] & M^{\vertexset} \ar[r] & M \ar[r] & 0
\end{tikzcd}
\end{equation}
where the leftmost map is the diagonal, and the rightmost one is the sum of components. In particular, the vector space underlying $\mathbb{A}_{\Gamma_n,\momenta}$ is simply $M$. For all $i\in \indexset$ we let \[k_i: \mathbb{A}_{\Gamma_n,\momenta} \to M\] be the projection that picks out the momentum along the edge $e_i$. The $k_i$ are all isomorphisms and correspond to different choices of origin of the affine space. It will be convenient to pick one of them in particular. We let $k:=k_n$. Then:
\begin{equation}
\label{equation: k_i in function of k}
    \forall i\in \indexset,\, k_i=k+p_1+\cdots+p_i.
\end{equation}
We set
\begin{equation}
    \label{equation: definition momentai}
    \momentai:=\sum_{l=1}^i p_l
\end{equation}
so that $k_i=k+p_{1,i}$. More generally, for $1\leq i \leq j \leq n$ we set:
\begin{equation}
    \label{equation: definition momentaij}
    p_{i,j}:=\sum_{l=i}^j p_l.
\end{equation}
It is the sum of incoming momenta at vertices between edges $i$ and $j$.

\begin{figure}[htbp]
    \centering
\begin{tikzpicture}[scale=0.6]
\begin{feynman}

  \def\r{2}
  \def\R{3.2}

  \vertex (v1)  at (90:\r);
  \vertex (v2)  at (30:\r);
  \vertex (v3)  at (-30:\r);
  \vertex (vd)  at (-90:\r) {\(\cdots\)};
  \vertex (vn1) at (-150:\r);
  \vertex (vn)  at (150:\r);

  \vertex (p1)  at (90:\R);
  \vertex (p2)  at (30:\R);
  \vertex (p3)  at (-30:\R);
  \vertex (pn1) at (-150:\R);
  \vertex (pn)  at (150:\R);

  \diagram* {
    (v1) -- (v2)
         -- (v3)
         -- (vd)
         -- (vn1)
         -- (vn)
         --  [momentum'=\(k\)](v1),

    (p1) -- [momentum=\(p_1\)] (v1),
    (p2) -- [momentum=\(p_2\)] (v2),
    (p3) -- [momentum=\(p_3\)] (v3),
    (pn1) -- [momentum'=\(p_{n-1}\)] (vn1),
    (pn)  -- [momentum'=\(p_n\)] (vn),
  };

\end{feynman}
\end{tikzpicture}
\caption{Feynman $n$-gon with momenta.}
    \label{figure: n-gon momenta}
\end{figure}
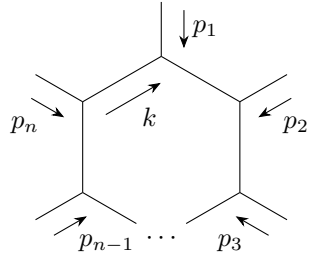
\subsection{Propagators}
To each edge $e_i$ of the graph $\Gamma$ we will now attach a function $\propagator_i$ on the affine momentum space $\momentumspace$, called the propagator associated with $e_i$. Recall that $L$ is a one-dimensional vector space and $q$ is a non-degenerate quadratic form on $M$ with value in $L^{\otimes 2}$. Equivalently, $q$ can be viewed as a quadratic form on $M\otimes L^{-1}$.
\begin{definition}
    The propagator associated with the edge $e_i\in \edgeset$ is the $L^{\otimes 2}$ valued quadratic function:
    \begin{equation}
        \label{equation: definition propagator}
        \propagator_i:=k_i^2+m_i^2
    \end{equation}
    where by convention:
    \begin{equation}
    \label{equation: definition massless propagator}
    k_i^2:=q\circ k_i.
\end{equation}
\end{definition}
By equation \eqref{equation: k_i in function of k}, we can also write it as $\propagator_i=(k+\momentai)^2+m_i^2$.
\begin{remark}
    Changing the orientation of an edge $e_i$ changes the linear form $k_i$ by a sign but leaves $\propagator_i$ invariant. Hence, the orientation of edges is not relevant to our story.
\end{remark}

Recall that the vector space associated with $\momentumspace$ is $M$. Hence, if we denote by $\det$ the top exterior product, any element in $\det M^\lor$ yields a top-differential form on $\mathbb{A}_{\Gamma,p}$. We fix such a non-zero element $\eta$. We also consider integers $\underline{\nu}\in\mathbb{Z}^n$, and set $\nu:=\sum_{i=1}^n \nu_i$.

\begin{definition}
    The differential form associated to $\Gamma_n$, $\eta$ and $\nu$ is the section of $\Omega^d_{\momentumspace}\otimes_\field L^{\otimes -2\nu}$:
    \begin{equation}
    \label{equation: definition Feynman differential form}
        \omega_{\Gamma_n,\underline{\nu},\momenta,\masses,\eta}:=\frac{\eta}{\prod_{i=1}^n\propagator_i^{\nu_i}}.
    \end{equation}
\end{definition}

\subsection{Feynman integral}
To define the integral, we assume that $\field$ is a subfield of $\mathbb{R}$, and that $q$ is positive definite. We denote tensorisation with $\mathbb{R}$ over $\field$ with a subscript $\mathbb{R}$. Then $q$ induces a norm on $\det M_\mathbb{R}^\lor$ that we denote by $|\eta|\in L^{\otimes -d}_{\mathbb{R}}$ for $\eta\in \det M^\lor_\mathbb{R}$. Note that if $\eta\in \det M^\lor$, then $|\eta|^2\in L^{\otimes -2d}$.
\begin{definition}
The Feynman integral of $\Gamma_n$ with kinematics $(\momenta,\masses)$ and exponents $\underline{\nu}$ is
    \begin{equation}
        \label{equation: definition Feynman integral}
        I(\Gamma_n,\underline{\nu},\momenta,\masses):=\frac{1}{|\eta|\pi^{\frac{d}{2}}}\int_{\mathbb{A}_{\Gamma_n,p}(\mathbb{R})} \omega_{\Gamma_n,\underline{\nu},\momenta,\masses,\eta} \in L_\mathbb{R}^{\otimes (d-2\nu)}
    \end{equation}
where $\mathbb{A}_{\Gamma_n,\momenta}(\mathbb{R})$ is oriented by $\eta$.
\end{definition}
Note that $\omega_{\Gamma_n,\underline{\nu},\momenta,\masses,\eta}$ depends on $\eta$, contrary to $I(\Gamma_n,\underline{\nu},\momenta,\masses)$. Moreover, for non-zero masses, the integral converges if and only if $\nu>\frac{d}{2}$. From now on we assume that this condition is met. For Euclidean kinematic configurations, there is a standard element $\mathrm{d}^dk$ in $\det^\lor k^d$ of norm $1$. We will also denote it as $\eta_{\eucl}$. Hence:
 \begin{equation}
        \label{equation: euclidean Feynman integral}
        I(\Gamma_n,\underline{\nu},\momenta,\masses)=\frac{1}{\pi^{\frac{d}{2}}}\int_{\mathbb{R}^d} \frac{\mathrm{d}^dk}{\prod_{i=1}^n((k+\momentai)^2+m_i^2)^{\nu_i}}\in \mathbb{R}.
\end{equation}
\begin{example}
    \label{example: bubble graph integral}
    The Feynman integral of the bubble graph $\Gamma_2$ with Euclidean $2$ dimensional kinematics is:
    \[I_{\Gamma_2,\underline{1},p_1,m_1,m_2}:=\frac{1}{\pi}\iint_{\mathbb{R}^2}\frac{d^2k}{(k^2+m_2^2)((k+p_1)^2+m_1^2)}.\]
The integral converges if and only if $m_0$ and $m_1$ are non-zero.
\end{example}
The next lemma justifies the prefactor in the definition and gives another possible definition of the integral. We let $d\delta_i\in \mathbb{Z}^n$ be the element with $i^{\textrm{th}}$ entry equal to $d$ and other entries equal to zero.
\begin{lemma}
For all $i\in \indexset$:
\begin{equation}
    \label{equation: trivial integral computation}
 I(\Gamma_n,d\delta_i,\momenta,\masses)=m_i^{-d}.
\end{equation}
In particular, for all $\eta\in \det M^\lor\setminus \{0\}$, and for all orientation of $\mathbb{A}_{\Gamma_n,\momenta,\masses}(\mathbb{R})$ we have:
\begin{equation}
    \label{equation: second definition of integral}
I(\Gamma_n,\underline{\nu},\momenta,\masses)=m_i^{-d}\left(\int_{\mathbb{A}_{\Gamma_n,\momenta}(\mathbb{R})} \omega_{\Gamma_n,\underline{\nu},\momenta,\masses,\eta}\right)\cdot\left(\int_{\mathbb{A}_{\Gamma_n,\momenta}(\mathbb{R})} \omega_{\Gamma_n,d\delta_i,\momenta,\masses,\eta}\right)^{-1}.
\end{equation}
\end{lemma}
\begin{proof}
We can assume that $i=n$, $L=\mathbb{R}$, and $m_n=1$. Equation \eqref{equation: trivial integral computation} becomes:
    \[\int_{\mathbb{R}^{d}}\frac{d\xi_1\wedge \cdots d\xi_{d}}{\left(\sum_{i=1}^{d} \xi_i^2+1\right)^d}=\pi^{\frac{d}{2}}\]
which can be proved by induction by integrating two coordinates at a time. Then, equation \eqref{equation: second definition of integral} follows from the definition of the integrals.
\end{proof}
The advantage of equation \eqref{equation: second definition of integral} is that it makes explicit how to define a Betti and a de Rham class that do not depend on an orientation of $M$.
\subsection{Kinematic invariants}
We would like to let $(M,L,q,\momenta,\masses)$ vary as parameters. Note that our constructions are functorial in $(M,L,q)$. Hence, it is natural to introduce spaces of kinematic invariants.
\begin{definition}
The space of momenta invariants in dimension $d$ is:
\begin{equation}
    \label{equation: definition momenta invariants}
    S_{n,d}=\{G=(s_{i,j})_{1\leq i,j\leq n}\in \Sym_n(\mathbb{Q}),\, \rank G \leq d \textrm{ and } \forall i \in \indexset,\, \sum_{j=1}^n s_{i,j}=0 \}
\end{equation}
The space of kinematics in dimension $d$ is:
\begin{equation}
\label{equation: definition kinematic space}    \Kinematics_{n,d}=\{(\underline{s},\underline{m}^2)\in S_{n,d}\times \mathbb{A}^n\}
\end{equation}
The space of projective kinematics in dimension $d$ is the projectivisation $\PKinematics_{n,d}$ of $\Kinematics_{n,d}$.
\end{definition}
Remark that for $d\geq n-1$, the rank condition is empty. In that case, we remove the $d$ subscript and use the notations $S_{n}$, $\Kinematics_n$ and $\PKinematics_{n}$. Moreover:
\[\Kinematics_{n}\simeq \mathbb{A}^{\binom{n-1}{2}} \times \mathbb{A}^n.\]
To kinematic data $(M,L,q,\momenta,\masses)$ we can associate the point $[p_i\cdot p_j:m_k^2]$ of $\PKinematics_{n,d}$. One should think about $\PKinematics_{n,d}$ as a coarse moduli space, or a GIT quotient for the action of $\Orth_d\times \mathbb{G}_m$ on $(\momenta,\masses^2)$ for fixed $(\mathbb{Q}^d,\mathbb{Q},q_{\textrm{eucl}})$.
\subsection{Quotient and cut graphs}
\subsubsection{Quotient graphs}
We introduce the relevant operations on graphs. The first is contraction of edges. For example, figure \ref{figure: pinching example} describes what happens when we pinch an edge of the box graph (the $4$-gon).

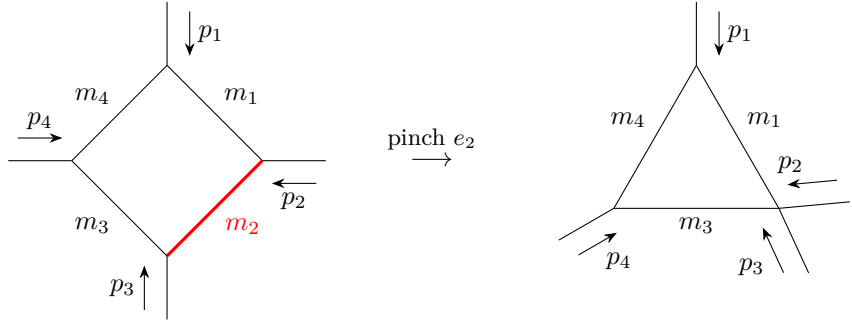
\begin{figure}[htbp]
    \centering

\begin{tikzpicture}[scale=0.7]
\begin{feynman}
  \def\r{1.8}
  \def\R{3}

  \vertex (v1) at (90:\r);
  \vertex (v2) at (0:\r);
  \vertex (v3) at (-90:\r);
  \vertex (v4) at (180:\r);

  \vertex (p1) at (90:\R);
  \vertex (p2) at (0:\R);
  \vertex (p3) at (-90:\R);
  \vertex (p4) at (180:\R);

  \diagram* {
    (v1) -- [edge label=\(m_1\)] (v2)
         -- [edge label=\(m_2\), very thick, red] (v3)
         -- [edge label=\(m_3\)] (v4)
         -- [edge label=\(m_4\)] (v1),

    (p1) -- [momentum=\(p_1\)] (v1),
    (p2) -- [momentum=\(p_2\)] (v2),
    (p3) -- [momentum=\(p_3\)] (v3),
    (p4) -- [momentum=\(p_4\)] (v4),
  };
\end{feynman}

\node at (5,0) {\(\longrightarrow\)};
\node at (5,0.4) {\small pinch \(e_2\)};

\begin{scope}[shift={(10,0)}]
\begin{feynman}

  \def\r{1.8}
  \def\R{3}

  \vertex (v1) at (90:\r);
  \vertex (v23) at (-30:\r);
  \vertex (v4) at (210:\r);

  \vertex (p1) at (90:\R);
  \vertex (p2) at (-15:\R);
  \vertex (p3) at (-45:\R);
  \vertex (p4) at (210:\R);

  \diagram* {
    (v1) -- [edge label=\(m_1\)] (v23)
         -- [edge label=\(m_3\)] (v4)
         -- [edge label=\(m_4\)] (v1),

    (p1) -- [momentum=\(p_1\)] (v1),
    (p2) -- [momentum'=\(p_2\)] (v23),
    (p3) -- [momentum=\(p_3\)] (v23),
    (p4) -- [momentum'=\(p_4\)] (v4),
  };

\end{feynman}
\end{scope}

\end{tikzpicture}    \caption{Edge pinching from the box to the triangle}
    \label{figure: pinching example}
\end{figure}

More generally, let $\gamma\subset {E_\Gamma}$ be a subset of the edges of $\Gamma$. We define the equivalence relation $\sim_\gamma$ on $V_\Gamma$ by $v\sim_\gamma v'$ if and only if there is a path of (unoriented) edges belonging to $\gamma$ between $v$ and $v'$. We also let $\gamma^c:=E_\Gamma\setminus \gamma$.

\begin{definition}
   The quotient of $\Gamma$ by $\gamma$ is the graph $\Gamma/\gamma$ with vertex set and edge set \[V_{\Gamma/\gamma}:=V_\Gamma/\sim_\gamma \text{ and }  E_{\Gamma/\gamma}:=\gamma^c\] and natural attachment maps given by composition from those of $\Gamma$.
\end{definition}
\begin{lemma}
    \label{lemma: quotient graph}
If $\gamma$ is a strict subset of ${E_{\Gamma_n}}$ there is a canonical isomorphism:
\begin{equation}
    \label{equation: canonical affine map quotient graph}
   \pr_{\gamma} : \mathbb{A}_{\Gamma_n,\momenta} \to \mathbb{A}_{\Gamma_n/\gamma,\momenta}.
\end{equation}
For all $e\in {E_\Gamma}\setminus\gamma$ we have \[\propagator_{e,\Gamma_n}=\pr_{\gamma}^*\propagator_{e,\Gamma_n/\gamma}.\]
\end{lemma}
In summary, in our case considering the quotient graph $\Gamma_n/\gamma$ really amounts to getting rid of the propagators $\propagator_e$ for $e\in \gamma$.

\subsubsection{Merging external edges}
When several external edges arrive at the same vertex, only the sum of their incoming momenta is relevant (see the definition of the affine momentum space of the graph \eqref{equation: definition momentum space graph}). For instance the triangle with four external edges of figure~\ref{fig:triangle_merge} can be reduced to three external edges.
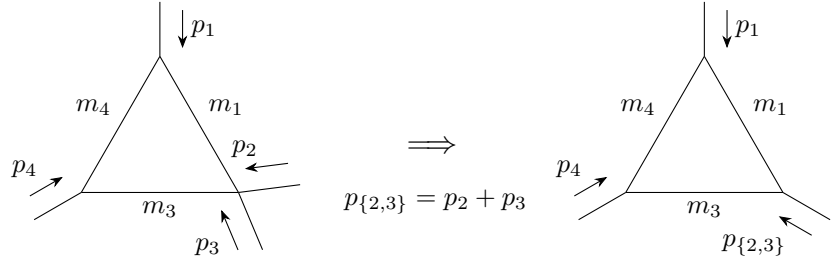
\begin{figure}[htbp]
\centering
\begin{tikzpicture}[scale=0.6]
\begin{feynman}

  \def\r{2}
  \def\R{3.2}

  \begin{scope}[xshift=-6cm]

  \vertex (v1) at (90:\r);
  \vertex (v2) at (-30:\r);
  \vertex (v3) at (210:\r);

  \vertex (p1) at (90:\R);
  \vertex (p2) at (-15:\R);
  \vertex (p3) at (-45:\R);
  \vertex (p4) at (210:\R);

  \diagram* {
    (v1) -- [edge label=\(m_1\)] (v2)
         -- [edge label=\(m_3\)] (v3)
         -- [edge label=\(m_4\)] (v1),

    (p1) -- [momentum=\(p_1\)] (v1),

    (p2) -- [momentum'=\(p_2\)] (v2),
    (p3) -- [momentum=\(p_3\)] (v2),

    (p4) -- [momentum=\(p_4\)] (v3),
  };

  \end{scope}

  \node at (0,0) {\Large $\Longrightarrow$};
  \node at (0,-1.2) {\(\;p_{\{2,3\}} = p_2 + p_3\)};

  \begin{scope}[xshift=6cm]

  \vertex (w1) at (90:\r);
  \vertex (w2) at (-30:\r);
  \vertex (w3) at (210:\r);

  \vertex (q1) at (90:\R);
  \vertex (q23) at (-30:\R);
  \vertex (q4) at (210:\R);

  \diagram* {
    (w1) -- [edge label=\(m_1\)] (w2)
         -- [edge label=\(m_3\)] (w3)
         -- [edge label=\(m_4\)] (w1),

    (q1) -- [momentum=\(p_1\)] (w1),

    (q23) -- [momentum=\(p_{\{2,3\}}\)] (w2),

    (q4) -- [momentum=\(p_4\)] (w3),
  };

  \end{scope}

\end{feynman}
\end{tikzpicture}
\caption{Merging two external edges into a single one.}
\label{fig:triangle_merge}
\end{figure}

If $\Gamma=\Gamma_n/\gamma$ is a quotient of the $n$-gon graph, we let $\Gamma_{\merged}$ be the graph obtained by merging together external legs. Then $\Gamma_{\merged}$ can be identified with the $k$-gon graph where $k$ is the number of internal edges of $\Gamma$. Moreover, there is a natural map on the spaces of kinematics:
\begin{equation}
    \label{equation: edge merging map}
f_{\Gamma}:\Kinematics_{n,d} \to \Kinematics_{k,d}
\end{equation}
that corresponds to summing momenta.
\subsubsection{Cut graphs}
We give for now a rather formal definition of graphs with cuts. 
\begin{definition}
    A cut graph is a couple $(\Gamma,\gamma)$ where $\Gamma$ is a Feynman graph, and $\gamma\subset E_\Gamma$ is the set of cut edges or cuts.
\end{definition}
For us $\Gamma$ will always be $\Gamma_n$ or one of its quotient graphs.
\begin{example}
The figure \ref{figure: box with cut} depicts the box graph with the second edge cut.
\begin{figure}[htbp]
    \centering
\begin{tikzpicture}[scale=0.6]
\begin{feynman}

  \def\r{1.8}
  \def\R{3}

  \vertex (v1) at (90:\r);
  \vertex (v2) at (0:\r);
  \vertex (v3) at (-90:\r);
  \vertex (v4) at (180:\r);

  \vertex (p1) at (90:\R);
  \vertex (p2) at (0:\R);
  \vertex (p3) at (-90:\R);
  \vertex (p4) at (180:\R);

  \tikzset{
    cut/.style={
      postaction={
        decorate,
        decoration={
          markings,
          mark=at position 0.5 with {
          \draw[thick] (-0.15,0.15) -- (0.15,-0.15);
          }
        }
      }
    }
  }

  \diagram* {
    (v1) -- [edge label=\(m_1\)] (v2)
         -- [edge label=\(m_2\), cut] (v3)
         -- [edge label=\(m_3\)] (v4)
         -- [edge label=\(m_4\)] (v1),

    (p1) -- [momentum=\(p_1\)] (v1),
    (p2) -- [momentum=\(p_2\)] (v2),
    (p3) -- [momentum=\(p_3\)] (v3),
    (p4) -- [momentum=\(p_4\)] (v4),
  };

\end{feynman}
\end{tikzpicture}
    \caption{Box with second edge cut}
    \label{figure: box with cut}
\end{figure}
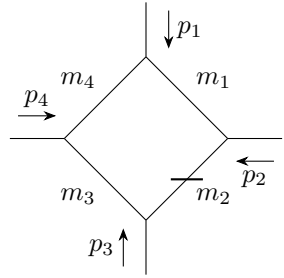
\end{example}
\section{Compactification of the momentum space}
\label{section: compactification of the momentum space}
We need to choose a compactification of the affine space $\mathbb{A}_{\Gamma_n,\momenta}$. The most obvious choice is to take the projective space compactification. However, for one-loop graphs there is an alternative, which can be found in \cite{abreu_cuts_2017}. Thanks to the embedding formalism we will view $\momentumspace$ as an open subset of a smooth quadric hypersurface $\quadric$ such that the subvarieties $\propagator_i=0$ are now hyperplane sections of $\quadric$. This is sometimes referred to as linearizing the propagators.
\subsection{Definition}
Consider the following map, which is written in projective coordinates notation:
\[[k:k^2:1]:\momentumspace \to \mathbb{P}(M\times L^{\otimes 2}\times k).\]
It is an immersion, which lands in the smooth projective quadric $\quadric$ defined by the non-degenerate quadratic form (with value in $L^{\otimes 2}$):
\begin{equation}
\label{equation: extended quadratic form}
    \tilde{q}(K,K_+,K_-):=K^2-K_+K_-
\end{equation}
In this way, we get the open immersion
\begin{equation}
    \label{equation: compactification affine momentum space}
    \immersion: \momentumspace \to \quadric
\end{equation}
which is the compactification that we will be using. We let $l_\infty\in \mathcal{O}_{X_{\Gamma_n,\momenta}}(1)$ be the linear form defining the boundary of this compactification:
\[l_\infty(K,K_+,K_-):=K_-.\]
A key property of the compactification \eqref{equation: compactification affine momentum space} is that the homogenisation of the propagators is linear. Indeed, the homogenisation of
\[\propagator_i= k^2 +2 k\cdot \momentai +\momentai^2+m_i^2\]
can be written using the cancellation of expression \eqref{equation: extended quadratic form} as the linear form
\begin{equation}
    \label{equation: expression of l_1}
    l_i(K,K_+,K_-):=K_+ + 2 K\cdot \momentai +(\momentai^2+m_i^2)K_-.
\end{equation}
It is such that:
\[\propagator_i=l_i/l_\infty.\]
We gave a definition of the compactification that relied on the choice of the coordinate $k$, but it is actually independent of this choice. Indeed, remark that the vector space of functions spanned by a propagator $\propagator_i$ and the affine functions on $\momentumspace$ is independent of $i$. Denote its dual by $\embeddingspace$. There is a natural immersion:
\[\momentumspace \to \mathbb{P}(\embeddingspace)\]
which is the choice-free expression of \eqref{equation: compactification affine momentum space}. The sections $l_\infty$ and $l_i$ correspond to the constant function $1$ and to the propagators $\propagator_i$. Moreover, the quadratic form $\tilde{q}$ on $\embeddingspace$ with values in $L^{\otimes 2}$ which defines $\quadric$ is also well-defined.
\subsection{Hyperplane sections}

We define the symmetric bilinear form $\langle \cdot,\cdot\rangle$ by:
\begin{equation}
    \label{equation: definition bilinear form}
\langle u,v \rangle:=\tilde{q}(u+v)-\tilde{q}(u)-\tilde{q}(v)
\end{equation}
for all $u,v \in \embeddingspace$. Be careful that $\langle u,u\rangle=2\tilde{q}(u)$. For all $i\in\indexset$ we define $u_i\in \embeddingspace$ and $u_\infty\in \embeddingspace\otimes_\field L^{\otimes -2}$ as the elements such that:
\[l_\infty(\cdot)=\langle u_\infty, \cdot \rangle \ ; \ l_i(\cdot)=\langle u_i , \cdot \rangle.\]
For all $i\in\indexsetbar$ we let $N_i$ be the kernel of $l_i$ and $H_i$ be the hyperplane $\mathbb{P}(N_i)$.
\begin{lemma}
    \label{lemma: dot product u}
We have the following expresssion:
\begin{multline}
    \tilde{q}(u_\infty)=0 \ ; \ \tilde{q}(u_i)=-m_i^2 \ ; \ \tilde{q}(u_i+u_j)=-\momentaij^2 \ ; \\ \langle u_\infty,u_i\rangle=1 \ ; \  \langle u_i,u_j \rangle=-(m_i^2+m_j^2+\momentaij^2).
\end{multline}
    Moreover, there is a natural identification of $u_\infty ^{\perp}/\langle u_\infty\rangle$ and the quadratic form induced on it by $\tilde{q}$ with $(M,q)$. 
\end{lemma}
\begin{proof}
Compute in coordinates:
\[\langle(K,K_+,K_-),(K',K'_+,K'_-) \rangle=2K\cdot K' - (K_+K'_- + K_-K'_+)\]
Then:
\begin{equation}
\label{equation: u in coordinates}
u_\infty=(0,-1,0) \ ; \  u_i=(\momentai,-(\momentai^2+m_i^2),-1).
\end{equation}
\end{proof}
Define $\subquadric{i}$ and $\subquadric{\infty}$ as the subquadrics defined by $l_i=0$ and $l_\infty=0$ respectively (they depend on $\momenta,\masses$). Let also $\spoint:=\langle u_\infty \rangle\in \quadric$. The geometric implication of the lemma is that $\subquadric{\infty}$ is a singular hyperplane section of $\quadric$. It is a projective quadric of corank one, with one singular point $\spoint$ , and it is a $\mathbb{P}^1$-bundle over the smooth projective quadric defined by $q$.
\subsection{Generic kinematics}
For any $I\subset \indexsetbar$ let
\[\G_I:=\det (\langle u_i,u_j\rangle)_{i,j\in I}\]
be the Gram determinant. Note that it is invariant under permutations. By lemma \ref{lemma: dot product u}, for any $I\subset \indexsetbar$, the function $\G_I$ can be viewed as a (homogeneous) regular function on $\Kinematics_{n}$.

\begin{definition}
\label{definition: space of generic kinematics}
The space of generic $d$ dimensional kinematics for $\Gamma_n$ is the open subscheme $\Kinematics_{n,d}^{\textrm{gen}}$ of the smooth locus of $\Kinematics_{n,d}$ where for all subsets $I\subset \indexsetbar$ of cardinal at most $d+1$ and different than $\{\infty\}$ the function $\G_I$ is invertible.
\end{definition}

\begin{proposition}
\label{proposition: normal crossing good parameters}
The kinematic configuration $(M,L,q_M,\momenta,\masses)$ is generic if and only if the divisor $\bigcup_{i=1}^{n} X_i$ is normal crossing and intersects $\subquadric{\infty}$ outside of its singular point and transversally.
\end{proposition}

\begin{remark}
If the kinematics are generic but $\G_I$ vanishes for $I$ of cardinal $d+2$ then there is an isolated point outside of the quadric where the hyperplane arrangement $(H_i)_{i\in\indexsetbar}$ is not normal crossing. This is of no incidence for us.
\end{remark}

As the next lemma shows, for Euclidean kinematics it is easier to check the genericity condition.
\begin{lemma}
    \label{lemma: euclidean parameters are very good parameters}
    Let $(\mathbb{R}^d,\mathbb{R},q_{\textrm{eucl}},\momenta,\masses)$ be a Euclidean kinematic configuration. Then it is generic if and only if all masses are non zero and for all $I\subset\indexsetbar$ of cardinal at most $d+1$, the vectors $u_i$ are linearly independent.
\end{lemma}
\begin{proof}
The conditions are clearly necessary. Indeed $\tilde{q}(u_i)=-m_i^2$, and invertibility of the Gram determinant implies linear independence.
Conversely, if the conditions are met, then the restriction of the quadratic form to the space of vectors orthogonal to $u_i$ is positive definite for all $i$. Hence, invertibility of the Gram determinant $\G_I$ is equivalent to linear independance of the family of $u_i$ for all subsets $I$ that intersect $\indexset$.
\end{proof}

\begin{definition}
\label{definition: space of euclidean kinematics}
The space $\Kinematics_{n,d}^{\eucl}$ of generic Euclidean kinematics in dimension $d$ is the image of generic Euclidean kinematic configurations $(\mathbb{R}^d,\mathbb{R},q_{\textrm{eucl}},\momenta,\masses)$ in $\Kinematics_{n,d}^{\textrm{gen}}(\mathbb{R})$.
\end{definition}

Finally, recall that if $\Gamma$ is a quotient graph of $\Gamma_n$ with $k$ edges, then merging the external edges that arrive to the same vertex in $\Gamma$ produces a graph $\Gamma_{\merged}$ that can be identified with the $k$-gon $\Gamma_{k}$. By Proposition~\eqref{proposition: normal crossing good parameters} the morphism \eqref{equation: edge merging map} induces a morphism between spaces of generic kinematics:
\begin{equation}
    \label{equation: edge merging map generic}
f_\Gamma:\Kinematics_{n,d}^{\textrm{gen}}\to \Kinematics_{k,d}^{\textrm{gen}}.
\end{equation}
\subsection{Quotients and cuts}
\label{subsection: Compactification_Quotients and cuts}
If $\gamma$ is a subset of $\indexsetbar\simeq E_{\Gamma_n}\cup \infty$, then we set:
\begin{equation}
    \label{equation: definition N_gamma X_gamma}
N_{\gamma}:= \bigcap_{e_i\in \gamma} \Ker(l_i) \ ; \ X_{\gamma}:= \bigcap_{e_i\in \gamma} X_i.
\end{equation}
If $\gamma$ is empty, we omit it from the notation. For generic kinematics, $X_{\gamma}$ is a projective quadric inside $\mathbb{P}(N_\gamma)$, which is singular if and only if $\gamma=\infty$. Then, to a cut quotient graph $(\Gamma,\gamma)$ of $\Gamma_n$ and kinematics $(M,L,q_M,\momenta,\masses)$ we assign the following data :
\begin{itemize}
    \item the vector space $\embeddingspace_\gamma$ of dimension $d+2-\#\gamma$;
    \item the non-degenerate quadratic form $q_{\embeddingspace_{\gamma}}$ on $\embeddingspace_\gamma$ with values in $L^{\otimes 2}$ that defines a smooth projective quadric $X_{\gamma}$;
    \item the linear form $l_\infty$ on $\embeddingspace_\gamma$ that defines a subquadric $X_{\gamma\cup\{\infty\}}$;
    \item the linear forms $l_i$ on $\embeddingspace_\gamma$ with values in $L^{\otimes 2}$ that define subquadrics $X_{\gamma\cup {e_i}}$ for $e_i\in {E_\Gamma\setminus \gamma}$.
\end{itemize}
This data will suffice to define the motives associated to the quotient cut graph. To actually specify an associated period we need a bit more. First from the equation \eqref{equation: definition N_gamma X_gamma} and the isomorphism \[N_\infty/u_\infty\simeq M\] of Lemma~\ref{lemma: dot product u} we get a natural isomorphism
\begin{equation}
\label{equation: det N isomorphic det H}
\begin{split}    
    \det \embeddingspace_\gamma &\simeq \det M \otimes_\field L^{\otimes (2-2\#\gamma)}.
\end{split}
\end{equation}
that depends up to a sign on the ordering of $\gamma$ (see remark \ref{remark: sign convention} in next section). 
Moreover, for Euclidean kinematics, it follows from Lemma~\ref{lemma: dot product u} that $q_{N}$ is of signature $(d+1,1)$, and $q_{N_\gamma}$ is of signature $(d+2-\#\gamma,0)$ if $1\leq\#\gamma\leq d$.
\subsection{Differential forms in embedding space}
Let $(\Gamma,\gamma)$ be a quotient cut graph of $\Gamma_n$, with $\#\gamma=r$ between $0$ and $d$. Recall that the canonical sheaf on $\mathbb{P}(N_\gamma)$ is
\begin{equation}
    \label{equation: sheaf top differential projective space}
    \Omega^{d-r}_{\mathbb{P}(N_\gamma)/k}\simeq \mathcal{O}_{\mathbb{P}(N_\gamma)}(-d+r+2)\otimes_{\field} \det N_\gamma^\lor.
\end{equation}
Any $\eta \in \det M^\lor$ yields an element of $\det N_\gamma^\lor\otimes_{\field}L^{\otimes (2r-2)}$ by equation \eqref{equation: det N isomorphic det H}, and we let $\Omega_{\eta,\mathbb{P}(N_\gamma)}$ be the corresponding section of $\Omega^{d-r}_{X_\gamma/k}(d-r)\otimes_{\field}L^{\otimes 2r-2}$. We define:
\begin{equation}
    \label{equation: definition Omega quadric}
    \Omega_{\eta,X_\gamma}:=\Res_{X_\gamma}\frac{\Omega_{\eta,\mathbb{P}(N_\gamma)}}{q}.
\end{equation}
If $\eta$ is nonzero, it is a basis of $\Omega^{d-r}_{X_\gamma/k}(d-r)\otimes_{\field}L^{\otimes 2r}$.
For $\underline{\nu}\in \mathbb{Z}^{E_\Gamma\setminus \gamma}$, we define the section of $\Omega^{d-r}_{X_\gamma/k}\otimes_k L^{\otimes 2(r-\nu)}$:
\begin{equation}
    \label{equation: new definition Feynman differential form}
\omega_{\Gamma,\gamma,\eta,\underline{\nu}}:=\frac{l_\infty^{\nu-d}\Omega_{\eta,X_\gamma}}{\prod_{e_i\in \gamma} l_i^{\nu_i}}.
\end{equation}
We can check that $\omega_{\Gamma,\eta,\underline{0}}$ restricted to $\momentumspace$ is simply the differential form $\eta$, hence the definition is consistent with definition \eqref{equation: definition Feynman differential form}. Differential forms associated with $(\Gamma,\gamma)$ are simply those differential forms on $(\Gamma_n,\gamma)$ that do not involve the propagators of contracted edges. It imposes the condition $\nu_i=0$ if the edge $e_i$ is contracted. Moreover, the lemma below shows that differential forms of cut graphs are residues of differential forms of graphs with no cuts, with $\nu_i=1$ if the edge $e_i$ belongs to the set of cuts $\gamma$.
\begin{lemma}
    \label{lemma: keeping track of differential forms}
Let $i_1<\cdots<i_r$ be such that $\gamma=\{e_{i_1},\ldots,e_{i_r}\}$. Then:
\begin{equation}
    \label{equation: residue on differential forms}
    \Res_{X_{i_1,\ldots,i_r}}\circ \cdots \circ \Res_{X_{i_1}}(\omega_{\Gamma,\eta,\underline{1}})=\omega_{\Gamma,\gamma,\eta,\underline{1}}
\end{equation}
\end{lemma}
\begin{proof}
Use equation \eqref{equation: definition Omega quadric} and check that:
\[\Res_{\mathbb{P}(N_{i_1,\ldots,i_r})}\circ \cdots \Res_{\mathbb{P}(N_{i_1})}(\Omega_{\eta,\mathbb{P}(N)}/\prod_{e_i\in\gamma}l_i)=\Omega_{\eta,\mathbb{P}(N_\gamma)}.\]
Then use the anticommutation of residues. See the remark below for the signs.
\end{proof}
\begin{remark}
\label{remark: sign convention}
We fix the choice of sign in isomorphism \eqref{equation: det N isomorphic det H} so that the lemma holds.
\end{remark}
For Euclidean kinematics, we can rewrite \eqref{equation: definition Feynman integral} as:
\begin{equation}
        \label{equation: new definition Feynman integral}
        I(\Gamma_n,\underline{\nu},\momenta,\masses)=\frac{1}{|\eta|\pi^{\frac{d}{2}}}\int_{X(\mathbb{R})} \omega_{\Gamma_n,\eta,\underline{\nu}}.
\end{equation}
\section{Motivic periods}
\label{section: motivic periods}

We consider a generic Euclidean kinematic configuration $(\field^d,\field,q_{\textrm{eucl}},\momenta,\masses)$ over $\field \subset \mathbb{R}$ and we will consider motivic periods in $\motper(\field)$, and Nori motives in $\Mot(\field)$.
\subsection{Case $\nu\geq d$}
We set:
\begin{equation}
    \label{equation: definition U}
    U:=X\setminus \bigcup_{i=0}^{n-1} X_i.
\end{equation}

\begin{proposition}
    \label{proposition:definition of motivic period}
    If $\nu\geq d$ then the differential form $\omega_{\eta_{\textrm{eucl}},\underline{\nu}}$ and the submanifold $X(\mathbb{R})$ oriented by $\omega_{\eta_{\textrm{eucl}},\underline{\nu}}$ define a de Rham cohomology class and a Betti homology class of $H^d(U)$. We define the motivic period
\[I^{\mathfrak{m}}(\Gamma_n,\underline{\nu},d,\momenta,\masses):=\left(\pi^{\fmot}\right)^{-d/2}[H^d(U),[X(\mathbb{R})],[\omega_{\eta_{\eucl},\underline{\nu}}]].\]
\end{proposition}
\begin{proof}
    We have
    \begin{equation}
        \label{equation: divisor of omega}
        \divi(\omega)=(\nu-d)[X_\infty] -\sum_{i=0}^{n-1}\nu_i[X_i] 
    \end{equation}
    hence $\omega$ belongs to $\Omega^d_{X/k}\left(U\right)$. 
    For the second claim note that $X(\mathbb{R})$ is a smooth compact submanifold of $X(\mathbb{C})$ of dimension $d$. 
    Moreover for all $i\in \{1,\ldots,n\}$ check that $X_i(\mathbb{R})=\emptyset$ and $X_\infty(\mathbb{R})=\{\spoint\}$. 
    Finally $X(\mathbb{R})$ is oriented by the restriction of $\omega_{\eta_{\textrm{eucl}},\underline{\nu}}$ which has no zero and no pole on $X(\mathbb{R})\setminus \{\spoint\}$. 
    Because $d\geq 2$, $\{\spoint\}$ is of codimension at least $2$ so $\omega$ defines an orientation of $X(\mathbb{R})$. 
\end{proof}

\subsection{Case $\frac{d}{2}<\nu<d$}

If $\nu<d$ then $\omega_{\eta_{\textrm{eucl}},\underline{\nu}}$ has a pole along $X_\infty$. But $X_\infty$ has one real point $\spoint$. We let $\widetilde{X}$ be the blow-up of $X$ along $\spoint$, and $\widetilde{X}_\infty$ (resp. $X_i$) be the strict transform of $X_\infty$ (resp. of $X_i$). We let $E$ be the exceptional divisor and $E_\infty:=E\cap \widetilde{X}_\infty$. We let $\widetilde{\omega}_{\eta_{\textrm{eucl}},\underline{\nu}}$ be the pull-back of $\omega_{\eta_{\textrm{eucl}},\underline{\nu}}$ on the blow-up.

\begin{proposition}
    \label{proposition:definition of motivic period second case}
    If $\nu>\frac{d}{2}$ then the differential form $\widetilde{\omega}_{\eta_{\textrm{eucl}},\underline{\nu}}$ and the submanifold $\widetilde{X}(\mathbb{R})$ oriented by $\widetilde{\omega}_{\eta_{\textrm{eucl}},\underline{\nu}}$ outside of $E(\mathbb{R})$ define a de Rham cohomology class and a Betti homology class of the motive
    \[H^d\left(\widetilde{X}\setminus \left(\bigcup_{i=0}^{n-1} \widetilde{X}_i \cup \widetilde{X}_\infty \right), E\setminus E_\infty\right).\]
 In this way we define the motivic period
 \[I^{\mathfrak{m}}(\Gamma_n,\underline{\nu},d,\momenta,\masses):=\left(\pi^{\fmot}\right)^{-d/2}\int^{\fmot}_{\widetilde{X}(\mathbb{R})}\widetilde{\omega}_{\eta_{\textrm{eucl}},\underline{\nu}}.\]
 For $\nu\geq d$ this definition agrees with the previous one.
\end{proposition}
\begin{proof}
    Locally at $\spoint$, the differential form $\omega$ is of the form :

\[\omega=f\frac{dz_1\ldots dz_d}{\left(\sum_i z_i^2\right)^{d-\nu}} \]

where $f$ is a unit. Hence after blowing-up we have coordinates of the form $(u,v_2,\ldots,v_{d})=(z_1,\frac{z_2}{z_1},\ldots,\frac{z_d}{z_1})$ and we find
\begin{align*}
\widetilde{\omega}_{\eta_{\textrm{eucl}},\underline{\nu}} &= f\frac{u^{d-1}du dv_1\ldots dv_{d-1}}{\left(u^2+\sum_i u^2v_i^2\right)^{d-\nu}} \\
 &= f\frac{u^{2\nu-d-1}du dv_1\ldots dv_{d-1}}{\left(1+\sum_i v_i^2\right)^{d-\nu}}.
\end{align*}
where $u=0$ is an equation for the exceptional divisor. This computation implies :
\begin{equation}
    \label{equation: divisor omega after blow-up}
\divi\left(\widetilde{\omega}_{\eta_{\textrm{eucl}},\underline{\nu}}\right)=\left(2\nu-d-1\right)\left[E\right]-\left(d-\nu\right)\left[\widetilde{X}_\infty\right]-\sum_{i=0}^{n-1} \nu_i\left[\widetilde{X_i}\right].
\end{equation}
Moreover, as $1+\sum_{i=1}^{d-1} v_i^2=0$ does not admit real solutions we see that $\widetilde{X}_\infty$ does not have any real points. If $\nu>\frac{d}{2}$ then $\widetilde{\omega}_{\eta_{\textrm{eucl}},\underline{\nu}}$ does not have a pole along the exceptional divisor. Hence $\widetilde{\omega}_{\eta_{\textrm{eucl}},\underline{\nu}}$ and $X(\mathbb{R})$ define de Rham and Betti classes.

Finally the natural morphism
\begin{multline*}
   H^d(U)\simeq H^d(U,\spoint) \simeq  H^d\left(\widetilde{X}\setminus \bigcup_{i\in \mathbb{Z}/n\mathbb{Z}} \widetilde{X_i},E\right) \\ \to H^d\left(\widetilde{X}\setminus \left(\bigcup \widetilde{X}_i \cup \widetilde{X}_\infty \right), E\setminus E_\infty \right) 
\end{multline*}
maps one to another the de Rham and Betti classes we have defined for $\nu\geq d$. This shows consistency of the two definitions.
\end{proof}

\begin{remark}
    Note that the cycle associated with $\widetilde{X}\left(\mathbb{R}\right)$ has a boundary along $E\left(\mathbb{R}\right)$ because $d$ is even.
\end{remark}

\begin{proposition}
\label{proposition: normal crossing very good parameters}
For generic kinematics, the divisor $\bigcup_{i\in \indexsetbar} \widetilde{X}_i \cup E$ is normal crossing.
\end{proposition}
\begin{proof}
By proposition \ref{proposition: normal crossing good parameters} it suffices to check that $E$ and $\widetilde{X}_\infty$ are smooth and intersect transversally. To do so take étale coordinates centered at $\spoint$ such that the subvariety $X_\infty$ has equation \[z_1^2+\cdots+z_d^2=0.\]
\end{proof}

\section{Motive of a generic hyperplane arrangement complement in a smooth quadric}
\label{section: motive of a generic hyperplane arrangement complement in a smooth quadric}
We investigate the cohomological motive of a generic hyperplane arrangement complement in a smooth quadric defined by data $(N,L,q,u_1,\ldots,u_n)$. To simplify notations, we assume that $L=\field$, hence $(N,q)$ is just a quadratic space. We begin by recalling some facts about the cohomology of a smooth quadric (see also \cite{nagel_cohomology_2025}), then proceed to compute the weight-graded pieces of our motives.
\subsection{Notations}
If $f$ is a map between $k$-varieties, we let $f^\usual$ (resp. $f_\usual$) be the usual pullback (resp. pushforward) in cohomology (resp. homology). If $f$ is a proper map between smooth varieties, then we can take the pullback (resp. the pushforward) for compactly supported cohomology (resp. Borel-Moore homology), which yields a pullback $f^\Gysin$ (resp. a pushforward $f_\Gysin$) in homology (resp. in cohomology) under Poincaré duality.

If $X$ and $Y$ are smooth projective varieties over $k$, equidimensional of dimension $d_X$ and $d_Y$, and $C\subset X\times Y$ is a smooth closed subvariety of dimension $d_C$, then $C$ defines a correspondence of degree $r=d_Y-d_C$ from $X$ to $Y$. Hence it defines a pushforward morphism:
\[C_*: H^i(X)(-r) \to H^{i+2r}(Y).\]
Then $C_*=p_{Y\Gysin}p_X^\usual$ where $p_X$ and $p_Y$ are the projections to $X$ and $Y$.

We let $\mu_2\subset \mathbb{G}_m$ be the group $\{\pm 1\}$. If $\mu_2$ acts linearly on a $k$-vector space $H$ (or an object of a $k$-linear abelian category), we denote the decomposition in eigenspaces for the eigenvalue $1$ and $-1$ as follows:
\[H\simeq H^+ \oplus H^-.\]

Finally, we will sometimes omit the zeros in short-exact sequences to save some space.

\subsection{Cohomology of a smooth projective quadric}
Let $(N,q)$ be a quadratic space of rank and dimension $d+2$ over $k$, and let $\quadric:=V(q)\subset \mathbb{P}(N)$ be the corresponding smooth projective quadric of dimension $d$. 
Recall that: \[
     \dim H^d(\quadric) = \left\{
    \begin{array}{ll}
      2  & \textrm{if $d$ is even} \\
     0 & \textrm{if $d$ is odd.}
    \end{array}
    \right.
    \]
From now on we suppose that $d=2m$ is even, and we describe $H^d(\quadric)$. Recall that $q$ induces a morphism $N\to N^\lor$, and hence an element $\disc(q)$ of $\det N ^{\otimes -2}$ called the discriminant of $q$. We define the signed discriminant of $q$ as
\[\sdisc(q):=(-1)^{m+1}\disc(q).\]
Both signed and unsigned discriminants map direct orthogonal sums of (even rank) quadratic spaces to (tensor) products.
We define the discriminant algebra of $q$ as:
\[
\discalg(q):= k\oplus \det N
\]
where the product of elements in $\det N$ is given by the linear map \[\sdisc(q):\det N ^{\otimes 2} \to k.\]
We then define the quadratic character attached to $q$ as:
\[
\chi_q:=\widetilde{H}^0(\Spec(\discalg(q))).
\]
Let $F$ be the maximal isotropic Grassmannian of $(N,q)$, i.e. points of $F$ correspond to Lagrangians of $N$.
\begin{lemma}
    \label{lemma: function ring F}
There is a natural isomorphism:
    \[\mathcal{O}_F(F)\simeq \discalg(q).\]
\end{lemma}
\begin{proof}
It is well-known that $F$ has two geometric connected components, hence $\mathcal{O}_F(F)$ is $2$-dimensional.
To construct the desired isomorphism, notice that if $\Lambda$ is a Lagrangian defined over $\bar{k}$, then we get natural isomorphisms
\[\det N_{\bar{k}} \simeq \det \Lambda \otimes \det N_{\bar{k}}/\Lambda \simeq \det \Lambda \otimes \det \Lambda ^\lor \simeq \bar{k}.\]
Unpacking this, we see that a Lagrangian picks out the basis element of $\det N_{\bar{k}}$
\[\eta_\Lambda:= e_1\wedge \cdots \wedge e_{2m+2}\] where $e_1,\ldots,e_{m+1}$ is a basis of $\Lambda$ and for $1\leq i,j \leq m+1$ :\[\langle e_i,e_{j+m+1} \rangle=\delta_{i,j}.\]
Moreover, the signed discriminant of $q$ computed in such a basis is $1$. In other words,
\[\sdisc q= \eta_\Lambda^{-2}.\]
Finally, over $\bar{k}$ there exist Lagrangians $\Lambda$ and $\Lambda'$ such that their intersection has dimension $m$. For such Lagrangians, one can check that \[\eta_\Lambda=-\eta_{\Lambda'}.\]
Now use theses statements and the fact that $\mathcal{O}(F)$ has dimension $2$ to conclude that the injection below has all required properties:
\[
\begin{array}{ccc}
    \det N & \to & \mathcal{O}(F) \\
     \eta & \mapsto & (\Lambda \mapsto \frac{\eta}{\eta_\Lambda}).
\end{array}
\]
\end{proof}
Let also $\mathcal{C}\subset \quadric\times F$ be the incidence variety:
\[\mathcal{C}:=\{(x,\Lambda), x\subset \Lambda \}.\]
Note that the projection morphism
\[p_F:\mathcal{C}\to F\]
is smooth proper of relative dimension $m$. Hence, $\mathcal{C}$ defines a correspondence from $\quadric$ to $F$ of degree $-m$ that induces a morphism :
\[\mathcal{C}_* : H^{2m}(\quadric)\to H^0(F)(-m).\]
\begin{proposition}
\label{proposition: middle cohomology smooth quadric even dimension}
    The morphism $\mathcal{C}_*$ is an isomorphism, and is equivariant for the natural action of $\Orth(q)/\SOrth(q)\simeq \mu_2$. In particular, there is a natural isomorphism:
    \[H^{2m}(\quadric)^-\simeq \chi_q(-m).\]
\end{proposition}
\begin{proof}
It follows from the well-known fact that the middle cohomology of $\quadric$ is spanned by classes $a$ and $b$ of Lagrangians belonging to the two connex components of $F$ ($a$ and $b$ are defined over the separable closure).
Moreover, $H^0(F)\simeq H^0(\Spec(\mathcal{O}_F(F)))$ and the action of $\Orth(q)$ on $\mathcal{O}_F(F)$ is given by its natural action on $\det (\embeddingspace)$.
\end{proof}
Let $\quadric'\subset \quadric \subset \quadric''$ be a chain of inclusion of smooth subquadrics of codimension $1$ (i.e. transverse hyperplane sections), with $\dim \quadric=2m$.
\begin{lemma}
\label{lemma: factorisation of Gysin morphisms subquadrics}
The morphisms \[H_{2m}(\quadric')\to H_{2m}(\quadric) \to H_{2m}(\quadric'')\]
factor through the direct summand $H_{2m}(\quadric)^+\simeq \mathbb{Q}(m)$.
\end{lemma}
\begin{proof}
Use the $\mu_2\simeq \Orth(q')/\SOrth(q')$ action. 
\end{proof}
\begin{lemma}
\label{lemma: + part quadric cohomology}
For all integers $k\leq 2d$, the pullback in cohomology induces an isomorphism:
\[H^{k}(\mathbb{P}(N))\simeq H^{k}(X)^+.\]
\end{lemma}
\begin{proof}
Lefschetz theorem as well as the computation of the middle cohomology of $X$.
\end{proof}
\subsection{Weight-graded computation}
\label{subsection: weight-graded computation}
We also consider vectors $u_1,\ldots,u_{n}$ in $N$ such that for all $J\subset \{1,\ldots,n\}$ of cardinal at most $d+2$, the Gram determinant of the family $(u_i)_{i\in J}$ is invertible. For all $1\leq i\leq n$ we let $(N_i,q_i)$ be the hyperplane orthogonal to $u_i$ with quadratic form obtained by restricting $q$, we set $H_i=\mathbb{P}(N_i)$ and we let $X_i=H_i\cap X$. For all subsets $I\subset \{1,\ldots,n\}$ we let
\[(N_I,q_I)=(\bigcap_{i\in I}N_i,q_{|N_I}) \ ;\ X_I=\bigcap_{i\in I}X_i.\]
Invertibility of the Gram determinants implies that $\bigcup_{i=1}^n H_i\cup X$ is a simple normal crossing divisor (Proposition~\ref{proposition: normal crossing good parameters}). 

We let $U=X\setminus \bigcup_{i=1}^{k} X_i$ be the complement of this normal crossing divisor. Then we may compute its cohomology thanks to Deligne's spectral sequence. Its first page is:
\[E_1^{-p,q}=\bigoplus_{I\subset \{1,\ldots,n\},|I|=p}H^{q-2p}\left(X_I\right)\left(-p\right)\]
with differentials given by the alternate sums of the Gysin morphisms associated to the inclusions. The spectral sequence degenerates on the $E_2$ page and converges to the weight graded pieces of $H^{p+q}\left(U\right)$.
\begin{lemma}
    \label{lemma: spectral sequence splits as direct sum first case}    
The splitting of the middle cohomology of a smooth even dimensional quadric as a direct sum induces a splitting of the spectral sequence
\[E_1^{-p,q}\simeq E_{+,1}^{-p,q}\oplus E_{-,1}^{-p,q}\]
where $E_{-}$ degenerates on the page $E_{-,1}$, and $E_{+,1}$ is the truncation at $q\leq 2d$ of the spectral sequence computing the cohomology of $\mathbb{P}(N)\setminus \bigcup_{1\leq i\leq n} H_i$.
\end{lemma}
\begin{proof}
    If we forget about the differential, then the splitting is given by Proposition~\ref{proposition: middle cohomology smooth quadric even dimension}. By Lemma~\ref{lemma: factorisation of Gysin morphisms subquadrics}, the differential $d_1$ factors through $E_{+,1}$, which implies that the spectral sequence can be written as a direct sum, where $E_{-,1}$ has zero differential. Finally, the identification of $E_{+,1}$ comes from Lemma~\ref{lemma: + part quadric cohomology}.
\end{proof}
In particular, the splitting of the spectral sequence induces a splitting of the weight-graded parts of the cohomology:
\begin{equation}
    \label{equation: weight-graded splitting}
    \gr^W_\bullet H^\bullet(U)\simeq \left(\gr^W_\bullet H^\bullet(U)\right)^+ \oplus \left(\gr^W_\bullet H^\bullet(U)\right)^-.
\end{equation}
As the next proposition shows, it lifts to a splitting of the cohomology.
\begin{proposition}
\label{proposition: computation weight-graded first case}
There is a unique splitting:
\begin{equation}
\label{equation: splitting cohomology U first case}
   H^\bullet(U)\simeq H^\bullet(U)^+ \oplus H^\bullet(U)^-
\end{equation}
that induces the splitting $\eqref{equation: weight-graded splitting}$ and such that for all $k\leq d$, the pullback in cohomology of the inclusion in projective space is an isomorphism:
\begin{align}
H^k(U)^+ &\simeq H^k(\mathbb{P}(N)\setminus \bigcup_{i=1}^n H_i)\\
 &\simeq \Ext^k \left(\mathbb{Q}(-1)^{\oplus (n-1)}\right).
\end{align}
Moreover, for all $0\leq m \leq \frac{d}{2}$:
\begin{equation}
\label{equation: computation weight graded first case}
    \gr^W_{d+2m} H^d(U)^- \simeq \bigoplus_{I\subset\indexset,|I|=2m}\chi_{q_I}\left(-\left(d/2+m\right)\right).
\end{equation}
\end{proposition}
\begin{proof}
It is a corollary of the previous lemma. It shows that the morphism
\[H^k(\mathbb{P}(N)\setminus \bigcup_{i\in\indexsetbar} H_i)\to H^k(U)\]
is an isomorphism if $k<d$. If $k=d$ it is injective and induces an isomorphism on the weight-graded with the $+$ part. Moreover, we can compute directly:
\begin{equation}
\label{equation: computation cohomology projective space minus hyperplanes}
    H^k(\mathbb{P}(N)\setminus \bigcup_{i=1}^n H_i)
\simeq \Ext^k \left(\mathbb{Q}(-1)^{\oplus (n-1)}\right).
\end{equation}
For $k\neq d$ the conditions impose \[H^k(U)=H^k(U)^+.\]
For $k=d$, the computation \eqref{equation: computation cohomology projective space minus hyperplanes} implies that \[\gr^W_{m}H^d(U)^+=0\] if $m\neq 2d$. Together with the conditions of the proposition, it implies that $H^d(U)^-$ is the preimage of $\gr^W_{2d}H^d(U)^-$ under the projection map to the top weight part. We can check that this i
\end{proof}

When $U=X$ the notation is consistent. More generally when $n\leq d+1$ it is also consistent with the action of isometries.

\begin{lemma}
    \label{lemma: splitting Hd for n=d+1}
    Assume $n\leq d+1$. Let $s$ be a reflection across a hyperplane containing $u_1,\ldots u_n$. Then the decomposition into eigenspaces of the cohomology of $U$ induced by $s$ is given by equation \eqref{equation: splitting cohomology U first case}.
\end{lemma}
\begin{proof}
The decomposition induced by $s$ satisfies the conditions of Proposition~\ref{proposition: computation weight-graded first case}.
\end{proof}
\subsection{Normalised motive}
To the data $(N,L,q_N,\underline{u})$ we associate the following motive:
\begin{equation}
\label{equation: motive associated to quadric arrangement}
\mot(N,L,q_N,\underline{u}):=H^d(U)^-\otimes (H^d(X)^-)^\lor.
\end{equation}
Proposition~\ref{proposition: computation weight-graded first case} implies that:
\begin{equation}
    \label{equation: bottom weight normalized motive}
W_0 \mot(N,L,q_N,\underline{u})\simeq \mathbb{Q}(0).
\end{equation}
More generally, we make the following definition, in order to restate Proposition~\ref{proposition: computation weight-graded first case}.
\begin{definition}
    \label{definition: chi_I first case}
    For all $I\subset\indexset$ such that $\rank(q_I)$ is even positive we define:
    \begin{equation}
    \label{equation: definition chi first case}
\chi_{I}:=\chi_{q_I}\otimes \chi_{q_N}^{\lor}
\end{equation}
\end{definition}

\begin{corollary}
\label{corollary: weight graded normalized motive first case}
For all $0\leq m \leq \frac{d}{2}$:
\begin{equation}
\label{equation: computation normalized weight graded first case}
    \gr^W_{2m} \mot(N,L,q,\underline{u})\simeq \bigoplus_{I\subset\indexset,|I|=2m}\chi_{I}\left(-m\right)
\end{equation}
Moreover, for all $I\subset\indexset$ of cardinal at most $d$:
\begin{equation}
\label{equation: computation chi I}
    \chi_{I}\simeq \widetilde{H}^0(\Spec(k\oplus kt/(t^2+(-1)^{m+1}\G_{I}))).
\end{equation}
\end{corollary}
\begin{proof}
The first statement is an immediate consequence of Proposition~\ref{proposition: computation weight-graded first case}. For the second statement, we use the fact that the signed discriminant takes direct orthogonal sums to tensor products (see equation \eqref{equation: functoriality chi orth sums}). We consider the decomposition:
\begin{equation*}
N\simeq\langle (u_i)_{i\in I}\rangle \oplus^\perp \langle (u_i)_{i\in I}\rangle^\perp
\end{equation*}
which yields an isomorphism:
\begin{equation*}
\chi_I\simeq \chi_{q_{|\langle (u_i)_{i\in I}\rangle}}
\end{equation*}
and the latter can be computed via the signed discriminant.
\end{proof}
\begin{remark}
The element $\G_{I}$ lives in $L^{\otimes 4m}$, which is an even tensor power, hence it makes sense to take its square roots.
\end{remark}
The advantage of definition \eqref{equation: motive associated to quadric arrangement} is its invariance under automorphisms.
\begin{lemma}
    \label{lemma: invariance under isometries}
Any automorphism of $(N,L,q_N)$ that preserves $\underline{u}$ acts as the identity on $\mot(N,L,q_N,\underline{u})$.
\end{lemma}
\begin{proof}
It acts as the identity on $\det(N)^\lor\otimes\det(N_I)$ for all subset $I$ of $\indexset$. Hence it acts as the identity on the weight graded parts. The difference with the identity morphism is then zero on the weight-graded parts, hence it is zero. Indeed morphisms of Nori motives are strict for the weight filtration.
\end{proof}
This lemma will be used to show that the motive only depends on the Gram matrix of the $u_i$.
\subsection{Orthogonal sum}
We want to show that $\mot(N,L,q_N,\underline{u})$ does not depend on the dimension $d$ of the quadric as long as it is large enough. Hence, we assume that there is an orthogonal decomposition \[\embeddingspace=\embeddingspace'\oplus^\perp \embeddingspace''\]
and that $\embeddingspace,\embeddingspace',\embeddingspace''$ are of dimensions $d+2,d'+2,d''+2$ with $d,d',d''$ non-negative and $d''$ even.
Let $X,X',X''$ be the smooth projective quadrics defined inside $\mathbb{P}(\embeddingspace),\mathbb{P}(\embeddingspace')$ and $\mathbb{P}(\embeddingspace'')$ respectively by the quadratic form $q$ and its restrictions. Note that there is a natural morphism
\begin{equation}
    \label{equation: algebra morphism orthogonal sum}
\field [\sqrt{\sdisc q}] \to k[\sqrt{\sdisc q'}]\otimes k[\sqrt{\sdisc q''}]
\end{equation}
that induces an isomorphism on the reduced cohomology
\begin{equation}
    \label{equation: functoriality chi orth sums}
\chi_q\to\chi_{q'}\otimes \chi_{q''} 
\end{equation}
as well as its Poincaré dual
\begin{equation}
    \label{equation: PD dual functoriality chi orth sums}
  \chi_{q'}\otimes \chi_{q''} \to \chi_q  
\end{equation}
which is twice the inverse of \eqref{equation: functoriality chi orth sums}. By proposition \ref{proposition: middle cohomology smooth quadric even dimension}, isomorphism \eqref{equation: PD dual functoriality chi orth sums} induces an isomorphism:
    \begin{equation}
    \label{equation: isomorphism cohomology orth sum}
    H^{d'}(X')^-\otimes H^{d''}(X'')^- \to H^d(X)^-(1).
    \end{equation}
We also consider vectors $u_1,\ldots,u_{n}$ in $\embeddingspace'$ such that for all $J\subset \{1,\ldots,n\}$, the Gram determinant of the family $(u_i)_{i\in J}$ is invertible. In particular, this implies that $n\leq d'+2$. The hyperplanes orthogonal to these vectors define subquadrics $X_i$ (resp. $X_i'$) of the quadric $X$ (resp. $X'$) for $1\leq i \leq n$. Together they form a simple normal crossing divisor. We let 
\[U:=X\setminus \bigcup_{i=1}^{n} X_i\textrm{ and } U':=(X'\setminus \bigcup_{i=1}^{n} X_i').\] Then Corollary~\ref{corollary: weight graded normalized motive first case} yields an isomorphism:
\begin{equation}
 \label{equation: isomorphism weight-graded cohomology open orth sum}
\gr^W_\bullet H^{d'}(U')^-\otimes \left(H^{d'}(X')^-\right)^{\lor} \simeq \gr^W_\bullet H^d(U)^-\otimes\left(H^{d}(X)^-\right)^{\lor}.
\end{equation}
We would like to lift it to an isomorphism:
\begin{equation}
 \label{equation: isomorphism cohomology open orth sum}
H^{d'}(U')^-\otimes \left(H^{d'}(X')^-\right)^{\lor} \simeq H^d(U)^-\otimes \left(H^{d'}(X')^-\right)^{\lor}.
\end{equation}
By isomorphism \eqref{equation: isomorphism cohomology orth sum} it amounts to an isomorphism:
\begin{equation}
 \label{equation: isomorphism cohomology open orth sum bis}
H^{d'}(U')^-\otimes H^{d''}(X'')^-\simeq H^d(U)^-(1).
\end{equation}
Remark that $H^{d'}(U')^+$ and $H^d(U)^+$ are zero if $n\geq 1$, so some of the ``$-$'' superscripts are unnecessary. To construct the isomorphism, we first define a correspondence that induces isomorphism \eqref{equation: isomorphism cohomology orth sum}. We define the incidence variety:
\[I:=\{(x,x',x'')\in X\times X'\times X'',\, x\subset x'+x''\}.\]
We let $p$ and $q$ be the natural projections:
\[
\begin{tikzcd}
 & I \ar[ld,"p"] \ar[rd,"q"] & \\
 X' \times X''& & X
\end{tikzcd}
\]
Then $p$ is projective smooth, and $q$ is projective. In particular, $I$ is projective smooth. We then set $I_U:= I\times_X U$, and denote by $q_U$ the corresponding projection. Remark that $p$ restrict to a map $p_U:I_U\to U'\times X''$. We can now state the proposition below.
\begin{proposition}
    \label{proposition: orthogonal decompositions quadrics cohomology}
    The morphism of cohomology groups $q_!p^*$ restricts via the Künneth decomposition and the decomposition of the middle cohomology of a smooth quadric to an isomorphism
    \begin{equation*}
    %\label{equation: isomorphism cohomology orth sum}
    H^{d'}(X')^-\otimes H^{d''}(X'')^- \to H^d(X)^-(1)
    \end{equation*}
which equals isomorphism \eqref{equation: isomorphism cohomology orth sum}.
Similarly, the morphism $(q_U)_!p_U^*$ induces an isomorphism
\begin{equation*}
%\label{equation: isomorphism cohomology open orth sum}
H^{d'}(U')\otimes H^{d''}(X'')^- \to H^d(U)(1)
\end{equation*}
which induces isomorphism \eqref{equation: isomorphism weight-graded cohomology open orth sum} on the weight-graded parts.
\end{proposition}

\begin{proof}
    For the first statement, we will use proposition \ref{proposition: middle cohomology smooth quadric even dimension}. Because $d=d'+d''+2$ and $d''$ is even, $d$ and $d'$ have same parity. If $d$ and $d'$ are odd, then morphism \eqref{equation: isomorphism cohomology orth sum} goes from $0$ to $0$, so it is an isomorphism. Assume now that $d$ and $d'$ are even. Let $F$ (resp. $F'$, $F''$) be the maximal isotropic Grassmannian of $X$ (resp. $X'$, $X''$). Then there is a natural closed immersion given by taking the sum of Lagrangians: \[F'\times F'' \to F.\]
The induced morphism on rings of global functions is the natural $k$-algebra morphism \eqref{equation: algebra morphism orthogonal sum}
\begin{equation*}
%    \label{equation: morphism H0 lagrangian varieties orthogonal decomposition}
    k[\sqrt{\disc q}] \to k[\sqrt{\disc q'}]\otimes_k k[\sqrt{\disc q''}]
\end{equation*}
which induces isomorphism \eqref{equation: PD dual functoriality chi orth sums}:
\[\chi_{q'}\otimes\chi_{q''}\to \chi_q.\]
Hence it suffices to prove that the square below commutes, where the arrows stand for correspondences and their composition is defined as usual.
\[
\begin{tikzcd}
    F'\times F'' \ar[r] \ar[d] & F \ar[d]\\
    X'\times X'' \ar[r] &  X
\end{tikzcd}
\]
Indeed both compositions yield the correspondence $\{(\Lambda',\Lambda'',x),\, x\subset \Lambda'+\Lambda''\}$. Be careful that we need to take the transpose of correspondences in the diagram to get the arrows in the good direction. 

For the second part of the lemma, it suffices to prove that the morphism $(q_U)_!p_U^*$ induces isomorphism \eqref{equation: isomorphism weight-graded cohomology open orth sum} on the weight-graded parts. For $J\subset \{1,\ldots,n\}$ we let \[X_J:=\bigcap_{i\in J} X_i \ ; \ U_J:=X_J\setminus \bigcup_{J\subsetneq J'\subset \{1,\ldots,k\}}X_{J'}.\] There are residue morphisms and restriction to an open subset morphisms:
\begin{equation}
    H^d(U) \to H^{d-k}(U_J)(-k) \leftarrow H^{d-k}(X_J)(-k).
\end{equation}
Together they induce the isomorphism of Proposition~\ref{proposition: computation weight-graded first case}:
\begin{equation}
    \gr^W_\bullet H^d(U) \simeq \bigoplus_{J\subset \{1,\ldots,d'\},\# J\equiv d [2] } H^{d-k}(X_J)^-(-k).
\end{equation}
The same holds if we replace the letters $X$ and $U$ by $X'$ and $U'$. Hence it suffices to check that the squares below are commutative, and then an easy induction argument on $d'$ finishes the proof.

The first square concerns commutativity with respect to restriction to an open:
\begin{equation}
\label{equation: cd functoriality correspondence restriction to an open}
\begin{tikzcd}
H^{d'+d''}(X'\times X'') \ar[r] \ar[d] & H^d(X)(1) \ar[d] \\
H^{d'+d''}(U'\times X'') \ar[r]  & H^d(U)(1).
\end{tikzcd}
\end{equation}
It may be obtained from the horizontal composition of two commuting squares. The first one is
\[
\begin{tikzcd}
H^{d'+d''}(X'\times X'') \ar[r] \ar[d] & H^{d'+d''}(I) \ar[d] \\
H^{d'+d''}(U'\times X'') \ar[r]  & H^{d'+d''}(I_U)
\end{tikzcd}
\]
which commutes because cohomology pullbacks are functorial.
The second square is:
\[
\begin{tikzcd}
H^{d'+d''}(I) \ar[r] \ar[d] & H^d(X)(1) \ar[d] \\
H^{d'+d''}(I_U) \ar[r]  & H^d(U)(1).
\end{tikzcd}
\]
It is obtained by applying Lemma~\ref{lemma: commutativity square pullback pushforward} (see below) to the cartesian square
\[
\begin{tikzcd}
    I_{U} \ar[r] \ar[d] \ar[dr, phantom, "\square"] & I  \ar[d] \\
    U \ar[r] & X.
\end{tikzcd}
\]

The next square concerns commutativity with respect to residues and is slightly more subtle.
We have to show that the following square commutes:
\begin{equation}
    \label{equation: commutative square residues proof}
\begin{tikzcd}
H^{d'+d''}\left(U' \times X''\right) \ar[r] \ar[d] 
  & H^{d'-1+d''}\left(U_i' \times X''\right)(-1) \ar[d] \\
 H^{d}(U)(1)\ar[r] 
  & H^{d-1}(U_i).
\end{tikzcd}
\end{equation}
The residue morphisms arise as the connecting morphisms associated with the distinguished triangle (or fiber sequence)
\[
\pi_*i_!i^!\pi^* \longrightarrow \pi_*\pi^* \longrightarrow \pi_*j_*j^*\pi^* ,
\]
where \(i\) and \(j\) denote complementary closed and open immersions, respectively, and \(\pi\) is the structural morphism. See also the appendix for the relation between 6 functors formalism and cohomology.
In our setting, let $1\leq i \leq k$, and let \[U_{(i)}:=X\setminus \bigcup_{j\neq i} X_j\] and similarly for \(U_{(i)}'\). Then there are two such fiber sequences: one corresponding to \(U_{(i)}\) and one to \(U_{(i)}'\times X''\). We define:
\[j:U\to U_{(i)} \ ; \ i:U_i \to U_{(i)} \ ; \ \pi:U_{(i)}\to \Spec(k)\]
and similarly $j'$, $i'$ and $\pi'$ for \(U_{(i)}'\times X''\).

To check commutativity of the square \eqref{equation: commutative square residues proof}, it therefore suffices to check that its vertical morphisms are induced by a morphism between these two fiber sequences.
This requires introducing some notation. We let \[I_{(i)}:=(I\times_X' U_{(i)}')\times_X U_{(i)} \ ; I_{i}:=(I\times_X' U_{i}')\times_X U_{i}. \]
Then we have the commutative diagram below in which the rectangles denote cartesian square, and in which we introduce some additional notation:
\begin{equation}
\label{equation: cd geometry for morphism of fiber sequences}
\begin{tikzcd}
    U_i'\times X'' \ar[r,"i'"] \ar[dr, phantom, "\square"]  & U_{(i)}'\times X'' & U'\times X'' \ar[l,"j'"] \\
    I_i  \ar[r,"i''"] \ar[u,"p_i"] \ar[d,"q_i"] & I_{(i)} \ar[u,"p_{(i)}"] \ar[d,"q_{(i)}"] \ar[dr, phantom, "\square"]& I_U \ar[l,"j''"] \ar[u,"p_U"] \ar[d,"q_U"] \\
    U_i \ar[r,"i"] & U_{(i)} &\ar[l,"j"] U
\end{tikzcd}
\end{equation}
From this commutative diagram, we may define a morphism of fiber sequences thanks to the usual adjunctions. We let $\pi ''$ be the structural morphism for $I_{(i)}$.
\begin{equation}
    \label{equation: cd morphism fiber sequences}
    \begin{tikzcd}
    \pi'_*i'_!i'^!\pi'^* \ar[r] \ar[d] & \pi'_*\pi'^* \ar[r] \ar[d] & \pi'_*j'_*j'^*\pi'^* \ar[d] \\
    \pi''_*i''_!i''^!\pi''^* \ar[r] \ar[d] & \pi_*''\pi''^* \ar[r] \ar[d] & \pi''_*j''_*j''^*\pi''^* \ar[d] \\
    \pi_*i_!i^!\pi^* [2](1) \ar[r] & \pi_*\pi^* [2](1) \ar[r] & \pi_*j_*j^*\pi^* [2](1)
    \end{tikzcd}
\end{equation}
More precisely, the upper vertical morphisms are defined using the adjunctions $\Id \to p_*p^*$, and the bottom ones are defined using the adjunctions $q_!q^! \to \Id$. Be careful that the middle line is not a fiber sequence. Now, the upper left square and the bottom right square commute because of Lemma~\ref{lemma: commutativity square pullback pushforward}. The commutation of the other two squares is standard, it is just the functoriality of pullback and Gysin morphisms in cohomology.
\end{proof}

\begin{lemma}
\label{lemma: commutativity square pullback pushforward}
Let
\[
\begin{tikzcd}
    W \ar[r,"g'"] \ar[d,"f'"] \ar[dr, phantom, "\square"] & Y \ar[d,"f"] \\
    V \ar[r,"g"] & X
\end{tikzcd}
\]
be a cartesian square of smooth irreducible varieties, where \(g\) is smooth and \(f\) is proper.  
Let \(\pi : X \to \Spec(k)\) denote the structural morphism, and set \(h := f g' = f' g\).  
Define
\[
c := \dim X - \dim Y = \dim V - \dim W.
\]
Then there is a natural commutative square:
\[
\begin{tikzcd}
    (\pi f)_*(\pi f)^* \ar[r] \ar[d] 
      & (\pi h)_*(\pi h)^* \ar[d] \\
    \pi_*\pi^*[2c](c) \ar[r] 
      & (\pi g)_*(\pi g)^*[2c](c).
\end{tikzcd}
\]
\end{lemma}

\begin{proof}
By Poincaré duality, it reduces to compatibility of pullback and pushforward in locally finite homology. We can treat locally finite homology as a particular case of relative homology by introducing compactifications, and then apply Lemma~\ref{lemma: pullback/pushforward relative cohomology} in the appendix. See also Remark~\ref{remark: unnecessary compactifications}.
\end{proof}

\subsection{Homology class of real points}
In this section we assume that $k=\mathbb{R}$ and we are only interested in singular homology with rational coefficients. Hence, we will omit the $\betti$ superscripts and the Tate twists most of the time. If $X$ is a variety defined over $\mathbb{R}$, we will also often write $X$ instead of $X(\mathbb{C})$ to alleviate the notations.

\subsubsection{Class of the sphere}
We now consider quadratic spaces over $\mathbb{R}$. For $\mathbb{R}^{2m+1,1}$ we get the $2m$ dimensional quadric $X$ defined by the equation:

\[\sum_{k=1}^{m} X_k^2+Y_k^2=AB.\]
Then $X(\mathbb{R})$ is a $2m$-dimensional sphere contained in the affine open $B\neq 0$. Hence, once oriented it defines a class in $H_{2m}(X)$. We would like to identify this class.
We set $U_k:=X_k+iY_k$ and $V_k:=X_k-iY_k$. The equation becomes
\[\sum_{k=1}^{m} U_kV_k=AB\]
which we will abbreviate as
\[U\cdot V=AB.\]
In these new coordinates, complex conjugation acts as \[ [U:V:A:B]\mapsto [\bar{V}:\bar{U}:\bar{A}:\bar{B}].\]
Moreover, we can easily define two Lagrangian subspaces:
\[\Lambda_1:=\{[U:0:A:0]\}\]
\[\Lambda_2:=\{[U:0:0:B]\}\]
Their intersection is $m-1$ dimensional, which implies that they define different homology classes.
\begin{lemma}
\label{lemma: class real point sphere}
The equality below holds in $H_{2m}(X)$:
    \begin{equation}
    \label{equation: class real points sphere}
        [X(\mathbb{R})]=[\Lambda_2]-[\Lambda_1]
    \end{equation}
In particular, there is a natural isomorphism:
\begin{equation}
    \label{equation: natural isomorphism homology sphere}
    H_{2m}(X(\mathbb{R}))\to H_{2m}(X)^-.
\end{equation}
\end{lemma}
\begin{proof}
For a short proof, compute the intersection number of $[X(\mathbb{R})]$ with $[\Lambda_2]$ and $[\Lambda_1]$. It is $\pm 1$. Check also that reflections act by $-1$ on $[X(\mathbb{R})]$. Hence, it must be $\pm([\Lambda_2]-[\Lambda_1])$.

It is also possible to make the proof explicit. We first define a real analytic parametrisation of $X(\mathbb{R})$:
    \[
    \begin{array}{ccccc}
    f_1 & : & \mathbb{P}^m(\mathbb{C}) & \to  & X(\mathbb{R})  \\
     & & [Z:C] & \mapsto     & [\bar{C}Z:C\bar{Z}:Z\bar{Z}:C\bar{C}] 
    \end{array}
    \]
    Note that this function maps the hyperplane $C=0$ to the point $[0:0:1:0]$ and is a diffeomorphism on the open complements.
    We orient $X(\mathbb{R})$ so that this map preserves the natural orientation of $ \mathbb{P}^m(\mathbb{C})$.
    We now define for all $t\in ]0;1]$:
  \[
    \begin{array}{ccccc}
    f_t & : & \mathbb{P}^m(\mathbb{C}) & \to  & X(\mathbb{C})  \\
     & & [Z:C] & \mapsto     & [\bar{C}Z:tC\bar{Z}:tZ\bar{Z}:C\bar{C}] 
    \end{array}
    \]
    and for $t=0$:
  \[
    \begin{array}{ccccc}
    f_0 & : & \mathbb{P}^m(\mathbb{C}) & \to  & X(\mathbb{C})  \\
     & & [Z:C] & \mapsto     & [Z:0:0:C] 
    \end{array}
    \]
    In particular, $f_0$ is a natural parametrisation of $\Lambda_2$.
    However, $(f_t)_{t\in [0;1]}$ does not define a homotopy because there is a continuity problem when $t$ and $C$ both vanish. Hence, we restrict $f_t$ to 
    \[D_t:=\{[Z:C], tZ\bar{Z}\leq C \bar{C}\}.\]
    We can check that this solves the continuity problem.
    The problem is that $f_1$ is now restricted to $D_1$, so it is no longer a parametrisation of $X(\mathbb{R})$. Hence, in order to complete the parametrisation, we introduce for all $t\in ]0;1]$
    \[
    \begin{array}{ccccc}
    g_t & : & D_t & \to  & X(\mathbb{C})  \\
     & & [Z:C] & \mapsto     & [\bar{C}Z:tC\bar{Z}:C\bar{C}:tZ\bar{Z}]
    \end{array}
    \]
    and
     \[
    \begin{array}{ccccc}
    g_0 & : & \mathbb{P}^m(\mathbb{C}) & \to  & X(\mathbb{C}).  \\
     & & [Z:C] & \mapsto     & [Z:0:C:0] 
    \end{array}
    \]
    Remark that $f_t$ and $g_t$ coincide when $C\bar{C}=tZ\bar{Z}$, which is the frontier of $D_t$ for $t>0$. Hence they glue together.
    For $t=1$, glueing $f_1$ and $g_1$ gives another parametrisation of $X(\mathbb{R})$. We just have to be careful that $g_1$ reverses the orientation, which explains why there is a minus sign in equation \eqref{equation: class real points sphere}.
    For $t=0$, we get parametrisations of $\Lambda_1$ and $\Lambda_2$.
\end{proof}
We now consider the subquadrics $X_i$ as well, under the assumption that they have no real points, which is equivalent to $q(u_i)<0$.
\begin{proposition}
\label{proposition: real points belong to minus part}
The Betti class $\sigma \in H_d^\betti(U)$ of the real locus $U(\mathbb{R})$ lies in $H^{\betti}_d(U)^-$, i.e. in $(H^d_{\betti}\left(U\right)^+)^\perp$.
\end{proposition}
\begin{proof}
The forms $d\log (l_i/l_j)$ generate $H^\bullet_{\derham}(U)^+$, and their restriction to the compact submanifold $U(\mathbb{R})$ are exact:
\[{\omega_{i,j}}_{|U(\mathbb{R})}=d(\log |\frac{l_j}{l_i}|).\]
In particular, the class of $U(\mathbb{R})$ vanish on $H^d_{\derham}(U)^+$.

An alternative proof if $n\leq d+1$ is to notice that reflections act on $\sigma$ by $-1$ because they reverse orientation, and to use Lemma~\ref{lemma: splitting Hd for n=d+1}.
\end{proof}

\subsubsection{Positive definite case}
We now consider $\mathbb{R}^{2m+2,0}$. Then $X$ is defined by the equation:
\[\sum_{k=1}^{m+1} X_k^2+Y_k^2=0.\]
In this case, $X(\mathbb{R})$ is empty. Thanks to the fact that $X(\mathbb{C})$ does not intersect $\mathbb{P}^{2m+1}(\mathbb{R})$, we can consider the Gysin refined morphism in singular homology with rational coefficients :
\begin{equation}
    \label{equation: Gysin positive definite case}
 H_{2m+2}(\mathbb{P}^{2m+1}(\mathbb{C}),\mathbb{P}^{2m+1}(\mathbb{R}))\to H_{2m}(X).
\end{equation}
Recall that we denote by $\pm$ superscripts the eigenspaces for a linear $\mathbb{Z}/2\mathbb{Z}$ action.
\begin{lemma}
\label{lemma: no real points case}
The morphism \eqref{equation: Gysin positive definite case} is an isomorphism, and it is equivariant for the natural action of $\Orth_{2m+2}(\mathbb{R})/\SOrth_{2m+2}(\mathbb{R})\simeq \mathbb{Z}/2\mathbb{Z}$. In particular, the splitting given by this group action induces a natural isomorphism:
\[H_{2m+1}(\mathbb{P}^{2m+1}(\mathbb{R}))\simeq H_{2m}(X)^-.\]
\end{lemma}
\begin{proof}
    There is a natural short exact sequence:
    \begin{equation}
    \label{equation: SES (complex projective space,real projective space)}
       H_{2m+2}(\mathbb{P}^{2m+1}(\mathbb{C}))\to H_{2m+2}(\mathbb{P}^{2m+1}(\mathbb{C}),\mathbb{P}^{2m+1}(\mathbb{R})) \to H_{2m+1}(\mathbb{P}^{2m+1}(\mathbb{R})).
    \end{equation}
    This short exact sequence is equivariant for the natural action of $\mathbb{Z}/2\mathbb{Z}$, which acts trivially on the first term of the sequence, and by multiplication by $-1$ on the last term. Hence the short-exact sequence splits naturally.
    
    Note that the Gysin morphism \eqref{equation: Gysin positive definite case} maps $H_{2m+2}(\mathbb{P}^{2m+1}(\mathbb{C}))$ isomorphically to the summand $H_{2m}(X)^+$ by Lemma~\ref{lemma: + part quadric cohomology}.
    
    For the other summand, we can make an explicit computation to relate it to the previous case. Consider the complex Hopf fibration of the unit sphere $S^{4m+3}$ in $\mathbb{C}^{2m+2}$ over $\mathbb{P}^{2m+1}(\mathbb{C})$ with fiber $S^1$. It contains the real Hopf fibration of $S^{2m+1}\subset \mathbb{R}^{2m+2}$ over $\mathbb{P}^{2m+1}(\mathbb{R}))$, with fiber $S^0$. An easy check shows that the homology pushforward of this fibration is an isomorphism (both terms fit into natural short exact sequences):
    \[H_{2m+2}(S^{4m+3},S^{2m+1})\simeq H_{2m+2}(\mathbb{P}^{2m+1}(\mathbb{C}),\mathbb{P}^{2m+1}(\mathbb{R}))^-.\]
    There are obvious explicit generators of the left-hand side. For example, take the intersection of $S^{4m+3}$ with half-subspace
    \[\{(x_1+iy,x_2,\ldots,x_{2m+2}), y\geq 0, \, \forall 1\leq i \leq 2m+2 \, x_i\in \mathbb{R}\}.\]
    It is a half-sphere with boundary $S^{2m+1}$, and we denote it $S^{2m+2}_+$.
    
    Now we can check that $X(\mathbb{C})$ and $S^{2m+2}_+$ are transversal one to another, hence we can compute the image of the class of $S^{2m+2}_+$ by the Gysin morphism by taking their fiber product over $\mathbb{P}^{2m+1}(\mathbb{C})$. It is given by :
    \[\{(iy,x_2,\ldots,x_{2m+2}), y=\frac{1}{\sqrt{2}}, \, x_2^2\cdots + x_{2m+2}^2=1/2 \}.\]
    This embedds in $X(\mathbb{C})$ as the closed subvariety:
    \[\{[iy:x_2:\cdots:x_{2m+2}], y^2=x_2^2\cdots + x_{2m+2}^2\}.\]
    Actually, if we modify the real structure on $X$ by performing a kind of Wick rotation, i.e. by replacing $\mathbb{R}^{2m+2}$ by $i\mathbb{R}\times \mathbb{R}^{2m+1}$, then this is the real points locus of $X$ for this new Wick rotated real structure. Denote it by $X_{Wick}(\mathbb{R})$, where $X_{Wick}$ is canonically isomorphic to $X$ over $\mathbb{C}$ but not over $\mathbb{R}$. We can now apply lemma \ref{lemma: class real point sphere} to $X_{Wick}$.
    
    Note that if we put everything back together, the fundamental class of $\mathbb{P}^{2m+1}(\mathbb{R})$ is mapped to a class of the form $\frac{1}{2}([\Lambda_2]-[\Lambda_1])$, in the notations of the previous case. The factor $\frac{1}{2}$ comes from the fact that the real Hopf fibration is two to one.
\end{proof}
We let $\tau$ be the class in $H_{2m+2}(\mathbb{P}^{2m+1}(\mathbb{C}),\mathbb{P}^{2m+1}(\mathbb{R}))^-$ whose boundary is
\begin{equation}
    \label{equation: tau definition}
\partial \tau = 2[\mathbb{P}^{2m+1}(\mathbb{R}))].
\end{equation}
It corresponds to a class of the form $[\Lambda_2]-[\Lambda_1]$ in $H_{2m}(X)^-$.
\subsubsection{Real points and orthogonal sum}
We go back to the situation of an orthogonal sum of two quadratic spaces of lemma \ref{proposition: orthogonal decompositions quadrics cohomology}, with the same notations except that we omit the $U$ subscripts for $I_U$, $p_U$ and $q_U$. Moreover, we assume that the base field is $k=\mathbb{R}$, that $(\embeddingspace',q')$ has signature $(d'+1,1)$ and $(\embeddingspace'',q'')$ has signature $(d''+2,0)$. Hence $(\embeddingspace,q)$ has signature $(d+1,1)$. We also assume that $q(u_i)<0$ for all $1\leq i\leq n$ so that $X_i$ has no real point. In this case $X(\mathbb{R})$ and $X'(\mathbb{R})$ define homology classes (up to orientation). We would like to show that isomorphism \eqref{equation: isomorphism cohomology open orth sum} maps one to the other.

We adopt the following convention. In commutative diagrams we will often not name the arrows if it is obvious which map should be considered, and we will sometimes put a "$\usual$" or a "$\Gysin$" label to indicate that we are considering a pushforward in homology, or a Gysin pullback in homology, associated to a natural morphism.
\begin{lemma}
    \label{lemma: orthogonal decomposition real points fundamental class}
    There exists a unique dashed map such that the diagram commutes:
    \begin{equation}
\label{equation: cd orthogonal decomposition real points fundamental class}
\begin{tikzcd}
H_d(X(\mathbb{R})) \ar[d] \ar[r,dashed,"\sim"] & H_{d'}(X'(\mathbb{R}))\otimes H_{d''}(X'')^- \ar[d] \\
H_d(U) \ar[r,"p_*q^!"] & H_{d'+d''}(U'\times X'')
\end{tikzcd}
    \end{equation}
  It is given by the composition:
  \[H_d(X(\mathbb{R}))\simeq H_{d}(X)^- \simeq H_{d'}(X')^- \otimes H_{d''}(X'')^- \simeq H_{d'}(X'(\mathbb{R}))\otimes H_{d''}(X'')^-.\]
    Moreover it maps the fundamental class $[X(\mathbb{R})]$ to $[X'(\mathbb{R})] \otimes ([\Lambda_2'']-[\Lambda_1''])$, where $\Lambda_1''$ and $\Lambda_2''$ are Lagrangians of $X''$ in different connex components.
\end{lemma}

\begin{proof}
We want to compute the image of the class of real points $[X(\mathbb{R})]$ by the map:
\[p_* q^!: H_{d}(U) \to H_{d'+d''}(U'\times X'')(-1).\]
The first case is when no hyperplane section is removed so $U=X$. Then we can describe everything in terms of Lagrangians using lemma \ref{lemma: class real point sphere}, and the result is clear.

Assume now that we are removing at least one hyperplane section. We can factor the map $q$ through the blow-up $\pi: \tilde{U}\to U$ of $U$ along $U'$, using the universal property of the blow-up. It yields a closed immersion $i:I\to \tilde{U}$:

\[
\begin{tikzcd}
I \ar[r,"i"] \ar[rd,"q"] & \tilde{U} \ar[d,"\pi"] \\
 & U
\end{tikzcd}
\]
Indeed, the pull-back of $U'$ by $q$ is the Cartier divisor $U'\times X''$. Moreover, since $U'$ is given by a transverse intersection with a linear subspace, the usual description of a blow-up of the projective space along a linear subspace shows that $\tilde{U}$ is naturally a closed subvariety of $U\times \mathbb{P}(\embeddingspace'')$:
\[
\tilde{U}=\{(x,x''), x+\embeddingspace'\subset x''+\embeddingspace'\}.
\]
Then, we can check explicitly that $i$ is a closed immersion. Here, we use that at least one hyperplane section was removed, otherwise we would have to blow-up $X$ along $X''$ too to get an immersion.

We can compute the Gysin morphism $q^!$ in two steps: first apply $\pi ^!$, then $i^!$.
We treat the computation of $\pi^!([X(\mathbb{R})])$ separately in Lemma~\ref{lemma: strict transform real points blow-up} (see next subsection). It relies on the observation that the normal bundle of $U'$ in $U$ is trivial, and that the exceptional divisor is just \[E= U'\times \mathbb{P}(\embeddingspace'').\]
The content of the lemma is that the class $\pi^!([X(\mathbb{R})])$ may be written as the sum of the fundamental class of $\tilde{U}(\mathbb{R})$ together with a class supported on $E$ whose boundary cancels the boundary of the fundamental class of $\tilde{U}(\mathbb{R})$ (which is only oriented outside of the exceptional divisor). Because we want to apply $i^!$ after, and $I$ does not intersect $\tilde{U}(\mathbb{R})$, we only care about the component of $\pi^!([X(\mathbb{R})])$ that is supported on $E$. More precisely, we may factorise $i^!$ as follows:
\[
\begin{tikzcd}
H_d(\tilde{U}) \ar[r,"*"] \ar[rd,"i^!"] & H_d(\tilde{U},\tilde{U}(\mathbb{R})) \ar[d,"!"] \\
 & H_{d'+d''}(I).
\end{tikzcd}
\]
Hence, we translate the results of the lemma into the following commutative diagram:
\begin{equation}
    \label{equation: cd (1) proof}
    \begin{tikzcd}
    H_d(X(\mathbb{R})) \ar[rr] \ar[dd,"*"] &  & H_{d'}(X'(\mathbb{R})) \otimes H_{d+1}(\mathbb{P}(\embeddingspace''_{\mathbb{C}}),\mathbb{P}(\embeddingspace''_{\mathbb{R}}))^{-} \ar[d] \\
     & & H_d(E,E(\mathbb{R})) \ar[d,"*"] \\
     H_d(U) \ar[r,"!"] & H_d(\tilde{U}) \ar[r,"*"] & H_d(\tilde{U},\tilde{U}(\mathbb{R}))
    \end{tikzcd}
\end{equation}
where the upper right vertical morphism is given by the Künneth theorem, and the upper horizontal morphism is an isomorphism between one-dimensional $\mathbb{Q}$-vector spaces. It maps $[X(\mathbb{R})]$ to $[X'(\mathbb{R})]\otimes \tau$ where $\tau$ has boundary twice the fundamental class of $\mathbb{P}(N''_{\mathbb{R}})$ (equation \eqref{equation: tau definition}). By Lemma~\ref{lemma: no real points case}, it corresponds to a class of the form $[\Lambda''_2]-[\Lambda''_1]$ in $H_{d''}(X'')^-$. More precisely, Lemma~\ref{lemma: no real points case} and the Künneth formula give the commutative diagram:
\begin{equation}
    \label{equation: cd (4) proof}
    \begin{tikzcd}
H_{d'}(X'(\mathbb{R})) \otimes H_{d+1}(\mathbb{P}(\embeddingspace''_{\mathbb{C}}),\mathbb{P}(\embeddingspace''_{\mathbb{R}})) \ar[d] \ar[r,"!"] & H_{d'}(X'(\mathbb{R}))\otimes H_{d''}(X'')^- \ar[d,"*"] \\
H_d(E,E(\mathbb{R})) \ar[r,"!"] & H_{d'+d''}(U'\times X'')
    \end{tikzcd}
\end{equation}
We must now apply the second Gysin morphism $i^!$ to $H_d(\tilde{U},\tilde{U}(\mathbb{R}))$. Note that $I$ intersects transversally the exceptional divisor $E$, and that there is a natural isomorphism:
\[ I\times_{\tilde{U}} E \simeq U'\times X''.\]
Hence, we may write the commutative square below.
\begin{equation}
    \label{equation: cd (2) proof}
\begin{tikzcd}
H_d(E,E(\mathbb{R})) \ar[d,"*"] \ar[r,"!"] & H_{d'+d''}(U'\times X'') \ar[d,"*"] \\
H_d(\tilde{U}, \tilde{U}(\mathbb{R})) \ar[r,"i^!"] & H_{d'+d''}(I)
\end{tikzcd}
\end{equation}
Applying $p_*$ is now easy, because the composition $U'\times X'' \to I \to U'\times X''$ is the identity, which gives the commutation of:
\begin{equation}
    \label{equation: cd (3) proof}
    \begin{tikzcd}
    H_{d'+d''}(U'\times X'')\ar[d,"*"] \ar[rd,"="] & \\
    H_{d'+d''}(I)\ar[r,"p_*"] & H_{d'+d''}(U'\times X'')
    \end{tikzcd}
\end{equation}
Finally, we put together commutative diagrams \eqref{equation: cd (1) proof}, \eqref{equation: cd (2) proof}, \eqref{equation: cd (3) proof} and \eqref{equation: cd (4) proof} to get the resulting diagram \eqref{equation: cd orthogonal decomposition real points fundamental class}.
To compute the dashed arrow of the lemma, we can use the commutative squares \eqref{equation: cd functoriality correspondence restriction to an open} and \eqref{equation: cd orthogonal decomposition real points fundamental class} to get the commutative square:
\begin{equation}
\begin{tikzcd}
    H_d(X(\mathbb{R})) \ar[d] \ar[r,dashed,"\sim"] & H_{d'}(X'(\mathbb{R}))\otimes H_{d''}(X'')^- \ar[d] \\
 H_d(X)^-  \ar[r] & H_{d'}(X')^-\otimes H_{d''}(X'')^-
\end{tikzcd}
\end{equation}
\end{proof}
\subsubsection{Strict transform of real points}
We prove the lemma that was used in the proof of Proposition~\ref{lemma: orthogonal decomposition real points fundamental class} to compute the Gysin pullback of the class of real points along a blow-up. Assume that $X$ is a smooth variety of dimension $n$ defined over $\mathbb{R}$, not necessarily proper. Assume also that $Z$ is a smooth closed subvariety of $X$ over $\mathbb{R}$ of codimension $c$, $p:\tilde{X}\to X$ is the blow-up of $X$ in $Z$, and $E$ is the exceptional divisor.

Because $p$ is proper, there is a Gysin morphism in homology:
\[p^\Gysin:H_n(X) \to H_n(\tilde{X}).\]
We also assume that $X(\mathbb{R})$ is compact and oriented, so that it defines a homology class in $H_n(X)$. We are interested in identifying its image by $p^!$. Note that $\tilde{X}(\mathbb{R})$ is closed and contained in the compact set $p^{-1}(X(\mathbb{R}))$, hence it is compact. We call it the strict transform of $X(\mathbb{R})$ because it is the closure of $p^{-1}((X\setminus Z)(\mathbb{R}))$. However, for orientation we must distinguish two cases:
\begin{itemize}
    \item If the codimension $c$ is odd, or if $Z(\mathbb{R})$ is empty, then $\tilde{X}(\mathbb{R})$ inherits an orientation from $X(\mathbb{R})$.
    \item If $c$ is even and $Z$ has real points, then we only get an orientation of $\tilde{X}(\mathbb{R})\setminus E(\mathbb{R})$.
\end{itemize}
We focus on the case where $c$ is even. Then $p^{-1}(X(\mathbb{R}))$ is the union of two closed components, $\tilde{X}(\mathbb{R})$ and $E\times_{Z}Z(\mathbb{R})$, which intersect on the $n-1$ dimensional manifold $E(\mathbb{R})$. In particular, there is a short exact sequence:
\begin{multline}
    \label{equation: SES blow-up preimage real points}
 0\to H_n(p^{-1}(X(\mathbb{R}))) \to H_n(\tilde{X}(\mathbb{R}),E(\mathbb{R})) \oplus H_n(E\times_{Z}Z(\mathbb{R}),E(\mathbb{R}))\\  \to H_{n-1}(E(\mathbb{R})) \to 0
\end{multline}
where the last non-trivial morphism is the sum of the two boundary morphisms. It is surjective because it sends $([X(\mathbb{R})],0)$ to $2[E(\mathbb{R}]$.

We now make the simplifying assumption that the normal bundle of $Z$ in $X$ is trivial, which is the only case that we need. Then, because $E$ is the projectivisation of this bundle, there is an isomorphism:
\[E\simeq Z\times \mathbb{P}^{c-1}.\]
The natural splitting of the short exact sequence \eqref{equation: SES (complex projective space,real projective space)} and the Künneth decomposition define a splitting:
\[s:H_{n-1}(Z(\mathbb{R})\times \mathbb{P}^{c-1}(\mathbb{R}))\to H_n(Z(\mathbb{R})\times \mathbb{P}^{c-1}(\mathbb{C}),Z(\mathbb{R})\times \mathbb{P}^{c-1}(\mathbb{R}))\]
of the short-exact sequence \eqref{equation: SES blow-up preimage real points}. Equivalently, we may write the splitting of the short exact sequence as a morphism:
\[H_n(\tilde{X}(\mathbb{R}),E(\mathbb{R})) \to  H_n(p^{-1}(X(\mathbb{R}))).\]
By composition, we may hence define a morphism:
\[\phi : H_n(X(\mathbb{R}))\to H_n(\tilde{X}(\mathbb{R}),E(\mathbb{R})) \to H_n(p^{-1}(X(\mathbb{R}))) \to H_n(\tilde{X}).\]
What this morphism does is map the fundamental class of $X(\mathbb{R})$ to its strict transform $\tilde{X}(\mathbb{R})$ plus some singular chain supported on the exceptional divisor, that compensates the boundary of $\tilde{X}(\mathbb{R})$. This second contribution may be seen as a class in $H_n(Z(\mathbb{R})\times \mathbb{P}^{c-1}(\mathbb{C}),E(\mathbb{R}))$ of the form
\[[Z(\mathbb{R})]\otimes \tau \in  H_{n-c}(Z(\mathbb{R}))\otimes H_{c}(\mathbb{P}^{c-1}(\mathbb{C}),\mathbb{P}^{c-1}(\mathbb{R}))^-\]
where $\tau$ is defined by \eqref{equation: tau definition}:
\[\partial \tau=2[\mathbb{P}^{c-1}(\mathbb{R})].\]
\begin{lemma}
\label{lemma: strict transform real points blow-up}
The morphisms $\phi$ and $p^!\circ i_*$ are equal, where $i$ is the inclusion $X(\mathbb{R})\to X(\mathbb{C})$, and $i_*$ is the pushforward in homology. In particular, there is a commutative diagram:
\begin{equation}
\label{equation: commutative diagram exceptional part Gysin real class blow up bis}
\begin{tikzcd}
H_n(X(\mathbb{R})) \ar[r] \ar[d,"i_*"] & H_{n-c}(Z(\mathbb{R}))\otimes H_{c}(\mathbb{P}^{c-1}(\mathbb{C}),\mathbb{P}^{c-1}(\mathbb{R}))^- \ar[d] \\
H_n(X) \ar[d,"p^!"] & H_n(E,E(\mathbb{R})) \ar[d,"*"] \\
H_n(\tilde{X})  \ar[r,"*"] & H_n(\tilde{X},\tilde{X}(\mathbb{R}))
\end{tikzcd}
\end{equation}
where the upper horizontal map maps the fundamental class $[X(\mathbb{R})]$ to $[Z(\mathbb{R})]\otimes \tau$.
\end{lemma}
\begin{proof}
We will use the real oriented blow-up of $X$ along $Z$, that we will denote $X'$ and view as a manifold with boundary. It replaces $Z(\mathbb{C})$ by the sphere bundle associated to its normal bundle. There is a natural projection $p'$ to $X(\mathbb{C})$, which is a $S^{2c-1}$ fiber bundle over $Z(\mathbb{C})$. We set \[E':=p'^{-1}(Z(\mathbb{C}))\] and we denote by $X'(\mathbb{R})$ the closure of \[p'^{-1}((X\setminus Z) (\mathbb{R})).\]
Moreover, there is a natural map $f:X'\to  \tilde{X}(\mathbb{C})$, that is an $S^1$ bundle over the exceptional divisor, and a diffeomorphism outside of it. 

Note that there is a short exact sequence for $X'$ similar to the sequence $\eqref{equation: SES (complex projective space,real projective space)}$ for $\tilde{X}$. The triviality of the normal bundle makes it possible to split it, and to define a linear map
\[\phi': H_n(X(\mathbb{R}))\to H_n(X')\]
the same way $\phi$ was defined. Observe now that the pushforward in homology $f_*$ defines a morphism from the short-exact sequence for $X'$ to the one for $\tilde{X}$, and that both splittings are compatible with $f_*$. This implies that the diagram below commutes:
\[
\begin{tikzcd}
 H_n(X(\mathbb{R})) \ar[r,"\phi'"] \ar[rd,"\phi"] & H_n(X') \ar[d,"f_*"] \\
  & H_n(\tilde{X})
\end{tikzcd}
\]
We now use the fact that the inclusion of $X\setminus Z$ in $X'$ is a homotopy equivalence to show that the diagram below commutes, where the morphisms that are not named are just pushforward in homology for the natural inclusion:
\[
\begin{tikzcd}
 & H_{\bullet}(X) \ar[ddr,"p^!"] & \\
 & H_{\bullet}(X\setminus Z) \ar[u] \ar[rd] \ar[ld,"\sim"] & \\
 H_{\bullet}(X') \ar[ruu,"p'_*"] \ar[rr,"f_*"] & & H_{\bullet}(\tilde{X})
\end{tikzcd}
\]
Indeed, we can use the commutation of the three inner triangles to show the commutation of the outer one. So far we have shown:
\begin{align*}
\phi&=f_* \circ \phi'\\
 &=p^! \circ p'_* \circ \phi'
\end{align*}
To conclude, it suffices to show that $p'_* \circ \phi'=i_*$. It follows from the fact that the morphism of short-exact sequences below induced by $p':(X',E')\to (X,Z)$ is compatible with the splittings:
\[\begin{tikzcd}
H_n(p'^{-1}(X(\mathbb{R}))) \ar[r]\ar[d] & H_n(X'(\mathbb{R}),E'(\mathbb{R})) \oplus H_n(E'\times_{Z}Z(\mathbb{R}),E'(\mathbb{R})) \ar[r] \ar[d] & H_{n-1}(E'(\mathbb{R}))  \ar[d] \\
H_n(X(\mathbb{R})) \ar[r] &  H_n(X(\mathbb{R}),Z(\mathbb{R}))  \ar[r] & 0
\end{tikzcd}
\]
The splitting of the upper short-exact sequence comes from a morphism:
\[H_{n-1}(E'(\mathbb{R}))\to H_n(E'\times_{Z}Z(\mathbb{R}),E'(\mathbb{R}))\]
and $p'_*$ maps $H_n(E'\times_{Z}Z(\mathbb{R}),E'(\mathbb{R}))$ to $0$, which concludes the proof.
\end{proof}

\section{Case where one hyperplane section is singular}
\label{section: a mild variation}
\subsection{Motive of interest}
To treat the case of the full motive, we must look at a mild degeneration of the above case. Namely, let $(\embeddingspace,q)$ be a non-degenerate quadratic space of dimension $d+2$ and let $X$ be the associated smooth projective quadric. Let $u_1,\ldots,u_n,u_\infty\in \embeddingspace$ be such that $q(u_\infty)=0$, and in generic position with respect to $q$. Under taking orthogonal subspaces, these vectors define hyperplane sections $X_1,\ldots,X_n,X_\infty$ of $X$. Then for all $I\subset \indexsetbar$, the hypotheses imply that $X_I$ is smooth, except for $X_\infty$ which has exactly one singular point $\spoint:=\langle \spoint \rangle$.

Instead of looking directly at the cohomology of the open subvariety \[U:=X\setminus \bigcup_{\indexsetbar} X_i\] we make a small adjustement. Let $\tilde{X}_I$ be the blow-up of $X_I$ along $\spoint$ for any $I\subset \indexsetbar$, and let $E_I$ be the exceptional divisor (we will write $E\subset X$ for $E_\emptyset\subset X_\emptyset$). We are interested in:
\begin{equation}
\label{equation: cohomology group blowup}
H^d(\tilde{X}\setminus \bigcup_{\indexsetbar} \tilde{X}_i,E)
\end{equation}
Actually, blow-ups will rapidly become cumbersome, so we make the following definition, using 6 functors formalism.
\begin{definition}
For any morphism $f:Y\to Z$ of $\field$-variety, we define the cohomology of $Y$ with compact support in $Z$ as: \[H_c^\bullet(Y/Z)=H^\bullet(p_{Z\,!}f_*f^*p_Z^*\mathbb{Q}(0))\]
where $p_Z$ is the structural morphism from $Z$ to $\Spec \field$, and $H^\bullet$ is the cohomology functor on $\DMot(\field)$.
\end{definition}
More details can be found in the annex about this general construction.
Most importantly, if $Z$ is proper, then \[H_c^\bullet(Y/Z)=H^\bullet(Y).\]
If $f:Y\to Z$ is proper (typically $Y=Z$) then: \[H_c^\bullet(Y/Z)=H^\bullet_c(Y).\]
If $Y$ is proper, then:
\[H_c^\bullet(Y/Z)\simeq H^\bullet(Y)\simeq H_c^\bullet(Y).\]
We let:  \[\overset{\circ}{X}:=X\setminus \{\spoint\}.\] 
\begin{lemma}
    \label{lemma: isomorphism blow-up cohomology with compact support}
The cohomology group \eqref{equation: cohomology group blowup} is naturally isomorphic to the cohomology of $U$ with compact support in $\overset{\circ}{X}$:
\begin{equation}
\label{equation: cohomology group blowup coincide compact support in}
    H^d_{c}(U/\overset{\circ}{X})\simeq H^d(\tilde{X}\setminus \bigcup_{\indexsetbar} \tilde{X}_i,E).
\end{equation}
\end{lemma}
\begin{proof}
We set the following convention. For any variety $V$ over $S$ and open subvariety $U\subset V$, we let $j_U^V$ denote the open immersion. We also let $p_V$ denote the structural morphism form $V$ to $S$. Then:
\begin{align*}
    \left(p_{\overset{\circ}{X}}\right)_{!}\left(j_{U}^{\overset{\circ}{X}}\right)_ {*}p_U^*&\simeq \left(p_{\tilde{X}}\right)_!\left(j_{\overset{\circ}{X}}^{\tilde{X}}\right)_!\left(j_{U}^{\overset{\circ}{X}}\right)_ {*}p_U^* \\
                &\simeq \left(p_{\tilde{X}}\right)_*\left(j_{\tilde{X}\setminus E}^{\tilde{X}}\right)_!\left(j_{U}^{\tilde{X}\setminus E}\right)_ {*}p_U^* \\
                &\simeq \left(p_{\tilde{X}}\right)_*\left(j^{\tilde{X}}_{\tilde{X}\setminus \cup \tilde{X}_i}\right)_*\left(j_{U}^{\tilde{X}\setminus \cup\tilde{X}_i}\right)_!p_U^*
\end{align*}
where the second isomorphism comes from the fact that $p_{\tilde{X}}$ is proper and $\tilde{X}\setminus E=\overset{\circ}{X}$, and the last isomorphism is the natural exchange structure morphism, which is an isomorphism because $E\cup \bigcup_{\indexsetbar}\tilde{X}_i$ is a normal crossing divisor in $\tilde{X}$.
The last expression is exactly how relative cohomology is defined using 6 functors.
\end{proof}
Remark that using Verdier duality we might also view the cohomology group \eqref{equation: cohomology group blowup} as the (twisted) dual of:
\[H^d(\overset{\circ}{X},\bigcup_{\indexsetbar} \overset{\circ}{X_i}).\]
\subsection{Basic case}
For any $I\subset \indexsetbar$ we let:  \[\overset{\circ}{X}_I:=X_I\setminus \{\spoint\}.\] For $I$ different from $\emptyset$ and $\{\infty\}$, we have that $\overset{\circ}{X}_I=X_I$ and:
\[H^\bullet(\overset{\circ}{X}_I/\overset{\circ}{X})\simeq H^\bullet(X_I).\]
Hence, nothing changes in that case. Moreover, we see that:
\begin{align*}
 H^\bullet(\overset{\circ}{X}/\overset{\circ}{X})&\simeq H^\bullet_c(\overset{\circ}{X})\\
 &\simeq H^\bullet(X,\spoint) \\
 &\simeq \widetilde{H}^\bullet(X).
\end{align*}
Only the case of $X_\infty$ remains. It is a projective quadric of corank $1$ with non-singular locus $\overset{\circ}{X}_\infty$. We denote by $\reducedquadric$ the corresponding smooth projective quadric obtained by quotienting by the kernel of the quadratic form. Because the projection from $\overset{\circ}{X}_\infty$ to $\overset{\circ}{X}$ is proper, we have:
\[H^\bullet(\overset{\circ}{X}_\infty/\overset{\circ}{X})\simeq H^\bullet_c(\overset{\circ}{X}_\infty).\]
\begin{lemma}
    \label{lemma: basic case}
For all $i\in\indexset$, the Gysin morphism induced by the inclusion of $X_{i\infty}$ in $\overset{\circ}{X}_\infty$ is an isomorphism:
\[H^{d-2}(X_{i\infty})(-1)\to H^{d}_c(\overset{\circ}{X}_\infty).\]
Moreover, the pullback morphism induced by the inclusion of $\overset{\circ}{X}_\infty$ in $\overset{\circ}{X}$ is an isomorphism:
\[H^{d}_c(\overset{\circ}{X}_\infty)\to H^{d}_c(\overset{\circ}{X})\simeq H^d(X).\]
\end{lemma}
\begin{proof}
The Gysin morphism is defined by duality from the pushforward in homology, and there is a commutative diagram
\[
\begin{tikzcd}
H_{d-2}(X_{i\infty}) \ar[r] \ar[rd,"\sim"] & H_{d-2}(\overset{\circ}{X}_\infty) \ar[d,"\sim"] \\
 & H_{d-2}(\bar{X})
\end{tikzcd}
\]
where two of the three morphisms are isomorphisms, because the projection from $X_{i\infty}$ to $\bar{X}$ is an isomorphism, and the projection from $\overset{\circ}{X}_\infty$ to $\bar{X}$ is an $\mathbb{A}^1$-bundle.

To show that the pullback morphism is an isomorphism, remark that the open complement $X\setminus X_\infty$ is an affine space and use the long exact sequence.
\end{proof}
Lemma~\ref{lemma: basic case} implies in particular that there are natural isomorphisms:
\begin{equation}
\label{equation: isomorphisms chi infty}
    \chi_{q_{i\infty}}\simeq \chi_{q}.
\end{equation}
\subsection{Computation}
We show how to adapt the computations of section \ref{subsection: weight-graded computation} to this new setting.
The invertibility of the Gram determinants implies that the $\overset{\circ}{X}_i$ for $i\in \indexsetbar$ form a normal crossing divisor in $\overset{\circ}{X}$. Hence, there is a version of Deligne's spectral sequence that computes the cohomology of $U$ with compact support in $\overset{\circ}{X}$. Its first page is:
\[E_1^{-p,q}=\bigoplus_{I\subset \{1,\ldots,n\},|I|=p}H^{q-2p}\left(\overset{\circ}{X}_I/\overset{\circ}{X}\right)\left(-p\right)\]
with differentials given by the alternate sums of the Gysin morphisms associated to the inclusions. The results of the previous subsection imply that $E_1^{-p,q}$ has weight $q$. Hence, the spectral sequence degenerates on the $E_2$ page and converges to the weight graded pieces of $H^{p+q}\left(U\right)$.
\begin{lemma}
    \label{lemma: spectral sequence splits as direct sum second case}    
The spectral sequence splits as
\[E_1^{-p,q}\simeq E_{-,1}^{-p,q}\oplus E_{+,1}^{-p,q}\]
where $E_{+,1}$ is the truncation at $q\leq 2d$ of the spectral sequence computing the cohomology of $\mathbb{P}(N)\setminus \bigcup_{1\leq i\leq n} H_i$ with compact support in $\overset{\circ}{\mathbb{P}(N)}:=\mathbb{P}(N)\setminus \spoint$ and $E_{-,1}$ only has one non-zero differential \[d_{-,1}^{-2,2}:E_{-,1}^{-2,2} \to E_{-,1}^{-1,2}\]
which is the sum of isomorphisms \eqref{equation: isomorphisms chi infty}:
\begin{equation}
    d_{-,1}^{-2,2}:\bigoplus_{i\in \indexset}\chi_{q_{i\infty}}\left(-d/2+1\right)\to \chi_{q}\left(-d/2+1\right).
\end{equation}
\end{lemma}
\begin{proof}
To adapt the proof of Lemma~\ref{lemma: spectral sequence splits as direct sum first case} note that for all $I\subset \indexsetbar$ and $k\leq 2\dim X_I$ the natural pullback morphism:
\[H^k(\overset{\circ}{H}_I/\overset{\circ}{\mathbb{P}(N)})\to H^k(\overset{\circ}{X}_I/\overset{\circ}{\mathbb{P}(N)})^+\]
is an isomorphism. Moreover, $\overset{\circ}{X}$ is proper over $\overset{\circ}{\mathbb{P}(N)}$ so:
\[H^k(\overset{\circ}{X}_I/\overset{\circ}{\mathbb{P}(N)})\simeq H^k(\overset{\circ}{X}_I/\overset{\circ}{X}).\]
Finally, the only thing that changes for the $E_{-}$ spectral sequence is the line $q=2$, where there are two non-zero terms, and a differential $d_{-,1}^{-2,2}$ which is computed from Lemma~\ref{lemma: basic case}.
\end{proof}
\begin{proposition}
\label{proposition: computation weight-graded second case}
There is a natural splitting:
\[H^\bullet(U/\overset{\circ}{X})\simeq H^\bullet(U/\overset{\circ}{X})^+ \oplus H^\bullet(U/\overset{\circ}{X})^-.\]
For all $k\leq d$, the pullback in cohomology of the inclusion in projective space is an isomorphism:
\[H^k(\mathbb{P}(N)\setminus \bigcup_{i\in\indexsetbar} H_i/\overset{\circ}{\mathbb{P}(N)})\simeq H^k(U/\overset{\circ}{X})^+.\]
Moreover, for all $0\leq m \leq \frac{d}{2}$ and $m\neq 1$:
\begin{equation}
    \gr^W_{d+2m} H^d(U/\overset{\circ}{X})^- \simeq \bigoplus_{I\subset\indexsetbar,|I|=2m}\chi_{q_I}\left(-\left(d/2+m\right)\right).
\end{equation}
For $m=1$ we have:
\begin{equation}
    \gr^W_{d+2}H^d(U/\overset{\circ}{X})^- \simeq \bigoplus_{I\subset\indexset,|I|=2}\chi_{q_I}\left(-d/2+1\right) \oplus \ker(s)
\end{equation}
where $s$ is the sum of isomorphisms \eqref{equation: isomorphisms chi infty}:
\begin{equation}
    s:\bigoplus_{i\in \indexset}\chi_{q_{i\infty}}\left(-d/2+1\right)\to \chi_{q}\left(-d/2+1\right).
\end{equation}
\end{proposition}
\begin{proof}
The same arguments as in the proof of Proposition~\ref{proposition: computation weight-graded first case} apply. In the $+$ part of the spectral sequence, only the terms corresponding to the degree $0$ cohomology of $\mathbb{P}(N)$ and $H_\infty$ are modified. We obtain:
\[H^k(\mathbb{P}(N)\setminus \bigcup_{i\in\indexsetbar} H_i/\overset{\circ}{\mathbb{P}(N)})\simeq\left\{
    \begin{array}{lll}
    \mathbb{Q}(-1)^{\oplus (n-1)} & \textrm{if $k=1$} \\
       \mathbb{Q}(-k)^{\binom{n}{k}}  & \textrm{if $2\leq k\leq d+1$} \\
        0 & \textrm{else.}
\end{array}\right.\]
For the computation of the weight-graded parts of $H^d(U/\overset{\circ}{X})^-$, use the spectral sequence. There is only one non-zero differentials, which is the surjective (for $n\geq 1$) morphism $s$.
\end{proof}
\begin{remark}
Similar methods should apply to kinematic configuration with some of the masses set to $0$. Indeed, by Lemma~\ref{lemma: dot product u} setting $m_i=0$ is equivalent to making the hyperplane section $X_i$ singular.
\end{remark}
\subsection{Normalised motive}
Similarly to the previous section, to the data $(N,L,q_N,\underline{u})$ we associate the following motive:
\begin{equation}
\label{equation: motive associated to quadric arrangement second case}
\mot(N,L,q_N,\underline{u}):=H^d(U/\overset{\circ}{X})^-\otimes (H^d(X)^-)^\lor.
\end{equation}
Proposition~\ref{proposition: computation weight-graded first case} implies that:
\begin{equation}
    \label{equation: bottom weight normalized motive second case}
W_0 \mot(N,L,q_N,\underline{u})\simeq \mathbb{Q}(0).
\end{equation}
We also make the following definition
\begin{definition}
    \label{definition: chi_I second case}
    For all $I\subset\indexsetbar$ such that $\rank(q_I)$ is even positive we define:
    \begin{equation}
    \label{equation: definition chi second case}
\chi_{I}:=\chi_{q_I}\otimes \chi_{q_N}^{\lor}
\end{equation}
\end{definition}
Then Corollary~\ref{corollary: weight graded normalized motive first case}, equation~\eqref{equation: computation chi I} and Lemma~\ref{lemma: invariance under isometries} adapt straightforwardly to the new setting.
\subsection{Orthogonal sum}
We now turn to the case of an orthogonal sum. Assume as before $\embeddingspace=\embeddingspace'\oplus \embeddingspace''$, $u_i\in \embeddingspace'$ for $i\in \{1,\ldots,n,\infty\}$ and $d''$ even.
Consider $X$, $X'$, $X''$, $I$, $p$, $q$, $X_i$, $X'_i$, $U$, $U'$ as before:
\[
\begin{tikzcd}
 & I \ar[ld,"p"] \ar[rd,"q"] & \\
 X' \times X''& & X
\end{tikzcd}
\]
\[
U:=X\setminus \bigcup_i X_i \ ; \ U':=X'\setminus \bigcup_i X'_i.
\]
What is new is that the subquadrics $X'_\infty$ and $X_\infty$ orthogonal to $u_\infty$ now have a singular point $\spoint$. We set the following notation:
\[\overset{\circ}{X'}:=X'\setminus \{\spoint\}\]
and similarly for $X_\infty$ and $X'_{\infty}$.
As before, we denote by subscripts the fibered products of $I$ over $X$ or $X'$.
For instance: \[I_{U}:=I\times_X U \ ; \ I_{U'}:=I\times_{X'} U'\]

Notice that:
\[
I_U\subset I_{U'} \ ; \ I_{\overset{\circ}{X'}}\subset \ I_{\overset{\circ}{X'}}.
\]
Hence $p$ and $q$ induce morphisms of pairs:
\begin{equation}
\label{equation: map of pairs p}
  (I_{\overset{\circ}{X'}},I_U)\to(\overset{\circ}{X'}\times X'',U'\times X'')
\end{equation}
and
\begin{equation}
    \label{equation: map of pairs q}
(I_{\overset{\circ}{X'}},I_U)\to (X\setminus\{\spoint\},U)
\end{equation}
Note that this last map is cartesian, meaning that $I_U$ is the pullback of $U$.
\begin{proposition}
\label{proposition: orthogonal decompositions quadrics cohomology bis}
The composition of pullback along $p$ and pushforward along $q$ restricts by the Künneth formula to an isomorphism:
\[
H_c^{d'}(U'/\overset{\circ}{X'}\times X'')\otimes H_c^{d''}(X'')^-\simeq H_c^d(U/\overset{\circ}{X})(1)
\]
\end{proposition}
\begin{proof}
The proof is very similar to the previous case.
The pullback along $p$ is the one defined in lemma \ref{lemma: pullback cohomology with compact support in} for the morphism of pairs \eqref{equation: map of pairs p}, which satisfies the properness hypothesis.
The pushforward along $q$ is obtained from Poincaré duality isomorphisms (Lemma~\ref{lemma: appendix Poincare duality relative cohomology and cohomology with compact support in}) from the pushforward in relative homology. We use that morphism \eqref{equation: map of pairs q} is a cartesian morphism of pairs of smooth varieties.
Künneth decomposition is the content of lemma \ref{lemma: Künneth decomposition cohomology with compact support in}. We want to make an induction argument for which we need to check compatibility of the morphism we have defined with residue and pullback morphisms. Our strategy is to use Poincaré duality to reduce to functoriality of relative cohomology. We only treat what differs from Proposition~\ref{proposition: orthogonal decompositions quadrics cohomology}.

Compatibility with residue was proven by defining a morphism of fiber sequences using the commutative diagram \eqref{equation: cd geometry for morphism of fiber sequences}. We reproduce it here in the case $i=\infty$ and enhance its maps to maps of pairs.
\[
\begin{tikzcd}
   \left(\overset{\circ}{X_\infty'}\times X'', U_\infty'\times X'' \ar[r,"i'"]\right) \ar[dr, phantom, "\square"]  & \left(\overset{\circ}{X'}\times X'', U_{\left(\infty\right)}'\times X''\right) & \left(\overset{\circ}{X'}\times X'', U'\times X''\right) \ar[l,"j'"] \\
    \left(I_{\infty,\overset{\circ}{X'}}, I_\infty\right)  \ar[r,"i''"] \ar[u,"p_\infty"] \ar[d,"q_\infty"] & \left(I_{\overset{\circ}{X'}}, I_{\left(\infty\right)}\right) \ar[u,"p_{\left(\infty\right)}"] \ar[d,"q_{\left(\infty\right)}"] \ar[dr, phantom, "\square"]& \left(I_{\overset{\circ}{X'}}, I_U\right) \ar[l,"j''"] \ar[u,"p_U"] \ar[d,"q_U"] \\
    \left(\overset{\circ}{X_\infty}, U_\infty\right) \ar[r,"i"] &  \left(\overset{\circ}{X}, U_{\left(\infty\right)}\right) &\ar[l,"j"]  \left(\overset{\circ}{X}, U\right)
\end{tikzcd}
\]
Here by cartesian square we mean that it is cartesian when restricted to either first or second term of the pairs.  Remark that the maps $p_\infty$, $p_{(\infty)}$ and $p_U$ as well as $j$, $j'$, and $j''$ are proper on the first term of the pairs, hence induce pullback morphisms as in Lemma~\ref{lemma: pullback cohomology with compact support in}. 
They are also smooth, hence by Lemma~\ref{lemma: appendix Poincare duality relative cohomology and cohomology with compact support in} these pullback morphisms can also be obtained by Poincaré duality from pullback morphisms in relative homology (Lemma~\ref{lemma: pushforward relative cohomology}). The maps $q_\infty$, $q_{(\infty)}$ and $q_U$ as well as $i$, $i'$, and $i''$ are cartesian, and have smooth source and target, hence induce pushforward Gysin morphisms under Poincaré duality (Lemma~\ref{lemma: appendix Poincare duality relative cohomology and cohomology with compact support in}).

We get a commutative diagram in $\DMot(\field)$:
\[
\begin{tikzcd}
   H^\bullet\left(U_\infty'\times X''/\overset{\circ}{X_\infty'}\times X''\right) \ar[r,"i'_\Gysin"] \ar[d,"p_\infty^\usual"] &  H^\bullet\left(U_{\left(\infty\right)}'\times X''/\overset{\circ}{X'}\times X''\right) \ar[r,"j'^\usual"] \ar[d,"p_{\left(\infty\right)}^\usual"]  &  H^\bullet\left(U'\times X''/\overset{\circ}{X'}\times X''\right) \ar[d,"p_U^\usual"]\\
     H^\bullet\left(I_\infty/I_{\infty,\overset{\circ}{X'}}\right)  \ar[r,"i''_\Gysin"]  \ar[d,"q_{\infty,\Gysin}"] &  H^\bullet\left(I_{\left(\infty\right)}/I_{\overset{\circ}{X'}}\right) \ar[d,"q_{\left(\infty\right),\Gysin}"] \ar[r,"j''^\usual"]&  H^\bullet\left(I_U/I_{\overset{\circ}{X'}}\right)  \ar[d,"q_{U,\Gysin}"] \\
     H^\bullet\left(U_\infty/\overset{\circ}{X_\infty}\right) \ar[r,"i_\Gysin"] &   H^\bullet\left(U_{\left(\infty\right)}/\overset{\circ}{X}\right) \ar[r,"j^\usual"]&   H^\bullet\left(U/\overset{\circ}{X}\right)
\end{tikzcd}
\]
Commutativity of the upper-right square is Lemma~\ref{lemma: pullback cohomology with compact support in}. The commutativity of the other squares are obtained from the functoriality properties of relative cohomology using the Poincaré duality isomorphisms of Lemma~\ref{lemma: appendix Poincare duality relative cohomology and cohomology with compact support in}. It reduces commutativity of the bottom left square to functoriality of pullback morphisms in relative cohomology (Lemma~\ref{lemma: pullback relative cohomology}). The commutativity of the two remaining squares is obtained from the compatibility between pullback and pushforward in relative cohomology (Lemma~\ref{lemma: pullback/pushforward relative cohomology}). Remark that smoothness of maps $p$ and $j$ is a necessary hypothesis.

Now the upper and bottom line of the diagram are the fiber sequences that define residue morphisms. This proves the commutativity of:
\[
\begin{tikzcd}
H^{d'+d''}\left(U'\times X''/\overset{\circ}{X'}\times X''\right) \ar[r,"\text{Res}"] \ar[d,"q_\usual p^\Gysin"] & H^{d'-1+d''}\left(U_\infty'\times X''/\overset{\circ}{X_\infty'}\times X''\right)(-1) \ar[d,"q_{\infty,\usual} p_\infty^\Gysin"] \\
H^{d}\left(U/\overset{\circ}{X}\right)(1) \ar[r,"\text{Res}"] & H^{d-1}\left(U_\infty/\overset{\circ}{X_\infty}\right)
\end{tikzcd}
\]
\end{proof}.
\subsection{Homology class of real points}
\subsubsection{Class of the sphere}
We adapt the argument of the basic case to our variation. Assume as before that we work over $\mathbb{R}$ and that $(\embeddingspace,q)$ has signature $(d+1,1)$. Assume also that $q(u_i)<0$ for all $1\leq i\leq n$ so that $X_i$ has no real point. The difference is that now $q(u_\infty)=0$, hence $X_\infty(\mathbb{R})=\{\spoint\}$. Because of this, $X(\mathbb{R})$ does not define a homology class in $H_d(U)$. However, it is true that \[\overset{\circ}{X}(\mathbb{R})\subset U(\mathbb{C})\] is proper over $\overset{\circ}{X}(\mathbb{C})$. Hence it defines a locally finite over $\overset{\circ}{X}$ homology class in \[H_d^{lf}(U/\overset{\circ}{X})\] 
\begin{proposition}
\label{proposition: real points belong to minus part second case}
The Betti class $\sigma \in H_d^\betti(U/\overset{\circ}{X})$ of the real locus $U(\mathbb{R})$ lies in $(H^d_{\betti}(U/\overset{\circ}{X})^+)^\perp$.
\end{proposition}
\begin{proof}
If we are removing at most $d+1$ hyperplane sections, then we can use the action of a reflection to conclude that $\sigma$ lies in the $-$ part of the homology. To deduce the general case, use Proposition~\ref{proposition: computation weight-graded second case}. It implies in particular that $H^d_{\betti}\left(U/\overset{\circ}{X}\right)^+$ is generated by the cohomology groups of complements of at most $d+1$ hyperplane sections in the quadric.
\end{proof}
\subsubsection{Orthogonal sum}
We turn to the case of an orthogonal sum and assume as before that we work over $\mathbb{R}$ and that $(\embeddingspace',q')$ and $(\embeddingspace'',q'')$ have signature $(d'+1,1)$ and $(d''+2,0)$ respectively. Hence $(\embeddingspace,q)$ has signature $(d+1,1)$. We would like to show that isomorphism \eqref{equation: isomorphism cohomology open orth sum} maps homology classes of $X(\mathbb{R})$ and $X'(\mathbb{R})$ one to the other.
In this subsection, we adopt the convention that homology is always singular homology and $X$ stands for $X(\mathbb{C})$ if $X$ is an algebraic variety over $\field$.
\begin{lemma}
    \label{lemma: variation orthogonal decomposition real points fundamental class}
    There exists a unique dashed map such that the diagram commutes:
    \begin{equation}
\begin{tikzcd}
H_d^{lf}(\overset{\circ}{X}(\mathbb{R})) \ar[d] \ar[r,dashed,"\sim"] & H_{d'}^{lf}(\overset{\circ}{X'}(\mathbb{R}))\otimes H_{d''}(X'')^- \ar[d] \\
H_d^{lf}(U/\overset{\circ}{X}) \ar[r,"p_\usual q^\Gysin"] & H_{d'+d''}^{lf}(U'\times X''/\overset{\circ}{X'}\times X'')
\end{tikzcd}
    \end{equation}
  It is the same as the map of Lemma~\ref{lemma: orthogonal decomposition real points fundamental class} under the identifications:
  \[H_d^{lf}(\overset{\circ}{X}(\mathbb{R}))\simeq H_d(X(\mathbb{R}))\ ; \  H_{d'}^{lf}(\overset{\circ}{X'}(\mathbb{R}))\simeq H_{d'}(X'(\mathbb{R})).\]
\end{lemma}
\begin{proof}
All the arguments used in the proof of \ref{lemma: orthogonal decomposition real points fundamental class} adapt to our setting. In particular, Lemma~\ref{lemma: strict transform real points blow-up} can be adapted as Lemma~\ref{lemma: strict transform real points blow-up bis} (see next section).
\end{proof}
\subsubsection{Strict transform of real points}
We use the same setting as in the section on strict transforms of real points: $X$ is a smooth variety of dimension $n$ defined over $\mathbb{R}$, $Z$ is a smooth closed subvariety of even codimension $c$ with trivial normal bundle, $p:\tilde{X}\to X$ is the blow-up of $X$ along $Z$, and $E$ is the exceptional divisor. However, we also consider a map $f:X\to Y$, where $Y$ is a variety defined over $\mathbb{R}$. By composition, we also get maps from $Z$, $E$, and $\tilde{X}$ to $Y$. We will consider cohomology with compact support in $Y$, or rather its dual, that we call homology with proper support over $Y$ (we could also say locally finite over $Y$ homology to be closer to the usual terminology for homology). See the appendix for more details. We assume that $X(\mathbb{R})$ is proper over $Y$ and oriented, so that it defines a homology class in $H_n(X/Y)$.

Because $p$ is proper, and $\tilde{X}$, $X$ are smooth, there is a Gysin morphism in homology with proper support over $Y$:
\[p^!:H_n(X/Y) \to H_n(\tilde{X}/Y).\]
Then $p^{-1}(X(\mathbb{R}))$ is the union of two closed components, $\tilde{X}(\mathbb{R})$ and $E\times_{Z}Z(\mathbb{R})$, which intersect on the $n-1$ dimensional manifold $E(\mathbb{R})$. Moreover, everything is proper over $Y$, hence homology with proper support over $Y$ is just locally finite homology. In particular, there is a short exact sequence:
\begin{multline}
        \label{equation: variation SES blow-up preimage real points}
 0\to H_n(p^{-1}(X(\mathbb{R}))/Y) \to H_n((\tilde{X}(\mathbb{R}),E(\mathbb{R}))/Y) \oplus H_n((E\times_{Z}Z(\mathbb{R}),E(\mathbb{R}))/Y)\\  \to H_{n-1}(E(\mathbb{R})/Y) \to 0.
\end{multline}
As before, we define the splitting:
\[s:H_{n-1}(E(\mathbb{R})/Y)\to H_n((E\times_Z Z(\mathbb{R}),E(\mathbb{R}))/Y)\]
of the short-exact sequence \eqref{equation: variation SES blow-up preimage real points}, as well as the morphism:
\[\phi : H_n(X(\mathbb{R}))\to H_n(\tilde{X}).\]
\begin{lemma}
\label{lemma: strict transform real points blow-up bis}
The morphisms $\phi$ and $p^!\circ i_*$ are equal, where $i$ is the inclusion $X(\mathbb{R})\to X$, and $i_*$ is the pushforward in homology. In particular, there is a commutative diagram:
\begin{equation}
\label{equation: commutative diagram exceptional part Gysin real class blow up}
\begin{tikzcd}
H_n(X(\mathbb{R})/Y) \ar[r] \ar[d,"i_*"] & H_{n-c}(Z(\mathbb{R})/Y)\otimes H_{c}(\mathbb{P}^{c-1}(\mathbb{C}),\mathbb{P}^{c-1}(\mathbb{R}))^- \ar[d] \\
H_n(X/Y) \ar[d,"p^!"] & H_n((E,E(\mathbb{R}))/Y) \ar[d,"*"] \\
H_n(\tilde{X}/Y)  \ar[r,"*"] & H_n((\tilde{X},\tilde{X}(\mathbb{R}))/Y)
\end{tikzcd}
\end{equation}
where the upper horizontal map maps the fundamental class $[X(\mathbb{R})]$ to $[Z(\mathbb{R})]\otimes \tau$.
\end{lemma}
\begin{proof}
Same arguments as for Lemma~\ref{lemma: strict transform real points blow-up}. We use the real oriented blow-up of $X$ along $Z$: \[p':X'\to X.\] 
There is a natural map $f:X'\to X$, that is an $S^1$ bundle over the exceptional divisor, and a diffeomorphism outside of it. As before, we get a commutative diagram:
\[
\begin{tikzcd}
 H_n(X(\mathbb{R})) \ar[r,"\phi'"] \ar[rd,"\phi"] & H_n(X'/Y) \ar[d,"f_*"] \\
  & H_n(\tilde{X}/Y)
\end{tikzcd}
\]
Be careful that to show that the diagram below commutes:
\[
\begin{tikzcd}
 & H_{\bullet}(X/Y) \ar[ddr,"p^!"] & \\
 & H_{\bullet}(X\setminus Z/Y) \ar[u] \ar[rd] \ar[ld,"\sim"] & \\
 H_{\bullet}(X'/Y) \ar[ruu,"p'_*"] \ar[rr,"f_*"] & & H_{\bullet}(\tilde{X}/Y)
\end{tikzcd}
\]
we use the fact that
\[H_{\bullet}(X\setminus Z/Y)\to H_{\bullet}(X'/Y)\]
is an isomorphism. It follows from the fact that the inclusion of $X\setminus Z$ in $X'$ is a proper over $Y$ homotopy equivalence. More simply, 6 functors formalism shows that it fits into a natural long exact sequence. Then, local computation of the exceptional pullback on the boundary of $X'$ shows that the third term is $0$.
\end{proof}
\section{Motives for one loop graphs}
\label{section: motives for one loop graphs}
\subsection{Definition over a point}
Let $\Gamma$ be a quotient graph of $\Gamma_n$. We translate definition \eqref{equation: motive associated to quadric arrangement} into the following definition.
\begin{definition}
\label{definition: graph motives pointwise}
We define the reduced motive of $\Gamma$ with kinematics $\momenta,\masses$ as:
\begin{equation}
    \label{equation: definition reduced motive}
\mot'(\Gamma,\momenta,\masses):=\mot(N,L,q_{N},\{l_i,\, e_i\in E_{\Gamma}\}).
\end{equation}
We define the full motive of $(\Gamma,\gamma)$ with kinematics $\momenta,\masses$ as:
\begin{equation}
    \label{equation: definition full motive}
    \mot(\Gamma,\momenta,\masses):=\mot(N,L,q_{N},\{l_{\infty}\}\cup\{l_i,\, e_i\in E_{\Gamma}\}).
\end{equation}
We define the quotient motive of $(\Gamma,\gamma)$ with kinematics $\momenta,\masses$ as:
\begin{equation}
    \label{equation: definition quotient motive}
    \mot''(\Gamma,\momenta,\masses):=\coker(\mot'(\Gamma,\momenta,\masses)\to \mot(\Gamma,\momenta,\masses)).
\end{equation}
\end{definition}
The terminology of reduced and full motive can be found in \cite{brown_feynman_2017-1}. The terminology quotient motive is ad hoc.
\begin{proposition}
\label{proposition: SES reduced full quotient motive}
There is a natural short exact sequence:
\begin{multline}
    \label{equation:SES reduced full quotient motive}
0 \to \mot'(\Gamma,\momenta,\masses) \to \mot(\Gamma,\momenta,\masses)\\ \to H^{d-1}(\overset{\circ}{X}_{\infty}\setminus \bigcup_{i\in E_\Gamma} X_{\{i,\infty\}}/\overset{\circ}{X})^-(-1)\otimes (H^d(X)^-)^\lor \to 0
\end{multline}
In particular:
\[\mot''(\Gamma,\momenta,\masses)\simeq H^{d-1}(\overset{\circ}{X}_{\infty}\setminus \bigcup_{i\in E_\Gamma} X_{\{i,\infty\}}/\overset{\circ}{X})^-(-1)\otimes (H^d(X)^-)^\lor\]
\end{proposition}
\begin{proof}
There is a long exact sequence of the form:
\[\cdots \to H^\bullet(\overset{\circ}{U}/\overset{\circ}{X}) \to H^\bullet(U\setminus U_\infty/\overset{\circ}{X}) \to H^{\bullet-1}(\overset{\circ}{U}_{\infty}/\overset{\circ}{X})(-1) \to \cdots\]
and we can identify
\[H^\bullet(\overset{\circ}{U}/\overset{\circ}{X})\simeq \widetilde{H}^\bullet(U).\]
Then split between $+$ and $-$ parts.
\end{proof}
We also define motives $\mot(\Gamma,\gamma,\momenta,\masses)$, $\mot'(\Gamma,\gamma,\momenta,\masses)$ and $\mot''(\Gamma,\gamma,\momenta,\masses)$ for cut graphs $(\Gamma,\gamma)$. If $\gamma$ is empty then we define them to be the motives \eqref{equation: definition reduced motive}, \eqref{equation: definition full motive} and \eqref{equation: definition quotient motive} respectively. Else we let $r=\# \gamma$ and we set:
\begin{align}
\label{equation: definition motive cut graph}
\mot'(\Gamma,\gamma,\momenta,\masses)&=H^{d-r}(X_\gamma \setminus \bigcup_{i\in E_\Gamma \setminus \gamma} X_{\gamma \cup \{i\}})^-(-r)\otimes (H^d(X)^-)^\lor; \\
\mot(\Gamma,\gamma,\momenta,\masses)&=H^{d-r}(X_\gamma \setminus \bigcup_{i\in E_\Gamma\cup\{\infty\} \setminus \gamma} X_{\gamma \cup \{i\}})^-(-r)\otimes (H^d(X)^-)^\lor; \\
\mot''(\Gamma,\gamma,\momenta,\masses)&=H^{d-r-1}(X_{\gamma\cup \infty} \setminus \bigcup_{i\in E_\Gamma \setminus \gamma} X_{\gamma \cup \{i,\infty\}})^-(-r-1)\otimes (H^d(X)^-)^\lor.
\end{align}
Note that ordinary cohomology suffices, because for $r\geq 1$ no singular quadric such as $X_\infty$ appears.
\begin{proposition}
\label{proposition: SES reduced full quotient motive cut graphs}
There is a natural exact sequence:
\begin{equation}
    \label{equation:SES reduced full quotient motive cut graphs}
0 \to \mot'(\Gamma, \gamma, \momenta,\masses) \to \mot(\Gamma,\gamma,\momenta,\masses) \to \mot''(\Gamma,\gamma,\momenta,\masses) \to 0
\end{equation}
\end{proposition}
\begin{proof}
Same proof.
\end{proof}
\subsection{Definition as motivic local systems on space of kinematic invariants}
Remark that these motives only depend on $d$ and on the kinematic invariants $\kinematics,\masses^2$ modulo scaling. Indeed, if two kinematic configurations $(M,L,\momenta,\masses)$ and $(M',L',\momenta',\masses')$ have same invariants $\kinematics$ (well-defined up to scalar), then, over $\bar{k}$ there exists an isomorphism between the data $(N,L,\tilde{q},\underline{u})$ and $(N',L',\tilde{q}',\underline{u}')$. Moreover, any automorphism of $(N,L,\tilde{q},\underline{u})$ induces the identity on the corresponding motives by Lemma~\ref{lemma: invariance under isometries}. We denote the resulting motives as \[\mot(\Gamma,\kinematics,\masses^2).\] Remark that the process we described can be understood very concretely in terms of matrices, and corresponds essentially to the lemma below.
\begin{lemma}
    \label{lemma: kinematics with same invariants}
Let $U,U'$ be $d$ times $n$ matrices of maximal rank, and let $Q,Q'$ be $d$ times $d$ symmetric matrices of maximal rank, with complex coefficients, such that:
\[U^T Q U=U'^T Q' U'.\]
Then, there exists a $d$ times $d$ invertible matrix $A$ such that $AU=U'$ and $A^T Q' A =Q$.
\end{lemma}
More generally, we have the following.
\begin{proposition}
\label{proposition: motivic local systems}
There is a motivic local system $\mot(\Gamma,d)$ on $\PKinematics_{n,d}^{\textrm{gen}}$ whose fiber at $[\kinematics:\masses^2]$ is $\mot(\Gamma,\kinematics,\masses^2)$. The same holds for the reduced, full and quotient motives of cut graphs.
\end{proposition}
\begin{proof}
We define the motivic local systems by étale descent. It is the same line of argument than the one used to define $\mot(\Gamma,\kinematics,\masses^2)$, except we reason étale locally instead of going to the algebraic closure.
\end{proof}
For large enough $d$, the spaces of kinematics stabilizes, as well as the motive. It is a corollary of Proposition~\ref{proposition: orthogonal decompositions quadrics cohomology}.
\begin{corollary}
\label{corollary: independence of d}
Consider integers $d>d'\geq n \geq 1$, where $d$ and $d'$ are even. Then Proposition~\ref{proposition: orthogonal decompositions quadrics cohomology} yields an isomorphism:
\begin{equation}
    \label{equation: isomorphism large d d'}
\mot(\Gamma_n,d)\simeq \mot(\Gamma_n,d').
\end{equation}
The same applies to reduced, full and quotient motives of quotient graphs with cuts.
\end{corollary}
We will use the notation $\mot(\Gamma_n)$ to denote this motivic local system.
\subsection{Pinches, cuts, merge}
Recall that pinching a subset of edges amounts to forgetting about some of the propagators. Let $(\Gamma,\gamma)$ be a quotient cut graph of $\Gamma_n$ with $k$ edges, and let $e$ be an edge in $E_{\Gamma}\setminus \gamma$.
\begin{definition}
    \label{definition: morphism quotient graph}
We let $p_e(\Gamma,\gamma)$ be the natural pullback morphism:
\[p_e(\Gamma,\gamma):\mot(\Gamma/e,\gamma,d)\to \mot(\Gamma,\gamma,d).\]
We also let $p_\infty(\Gamma,\gamma)$ be the natural pullback morphism:
\[p_\infty(\Gamma,\gamma):\mot'(\Gamma,\gamma,d)\to \mot(\Gamma,\gamma,d).\]
\end{definition}

\begin{definition}
    \label{definition: morphism cut graph}
We let $c_e(\Gamma,\gamma)$ be the natural residue morphism:
\[c_e(\Gamma,\gamma):\mot(\Gamma,\gamma,d)\to \mot(\Gamma,\gamma\cup\{e\},d).\]
We let $c_\infty(\Gamma,\gamma)$ be the natural residue morphism:
\[c_\infty(\Gamma,\gamma):\mot(\Gamma,\gamma,d)\to \mot''(\Gamma,\gamma,d).\]
\end{definition}
The proposition below lists the commutativity relations between pinching and cutting morphisms. We could also include the case where $e'$ is $\infty$. It suffices to replace $p_e$ and $c_e$ by their counterparts $p_e'$, $p_e''$, $c_e'$,$c_e''$ on the reduced or quotient motive when necessary.
\begin{proposition}
\label{proposition: functoriality pinch and cut morphism}
Let $e,e'$ be distinct edges in $E_{\Gamma}\setminus \gamma$. Then:
\begin{align}
    \label{equation: commutativity pinch and cut morphism}
p_{e'}(\Gamma,\gamma)\circ p_{e}(\Gamma/e',\gamma)&=p_{e}(\Gamma,\gamma)\circ p_{e'}(\Gamma/e,\gamma);\\
c_e(\Gamma,\gamma\cup\{e'\})\circ c_{e'}(\Gamma,\gamma)&=-c_{e'}(\Gamma,\gamma\cup\{e\})\circ c_e(\Gamma,\gamma); \\
p_e(\Gamma,\gamma\cup\{e\})\circ c_{e'}(\Gamma/e,\gamma)&=c_{e'}(\Gamma,\gamma)\circ p_{e}(\Gamma,\gamma);\\
c_{e}(\Gamma,\gamma)\circ p_{e}(\Gamma,\gamma)&=0.
\end{align}
In particular, the composition of pinching morphisms is independent of the order of the edges.
\end{proposition}
\begin{proof}
It follows from the usual properties of residue and pullback morphisms in the case of a simple normal crossing divisor.
\end{proof}
Merging external edges that point to the same vertex produces a graph $\Gamma_{\merged}$ that can be identified with the $k$-gon $\Gamma_{k}$. Then, we defined morphism \eqref{equation: edge merging map generic}
\[f_\Gamma:\PKinematics_{n,d}^{\textrm{gen}}\to \PKinematics_{k,d}^{\textrm{gen}}\]
by summing momenta of edges that point to the same vertex.
\begin{proposition}
\label{proposition: smaller parameter space quotient graph}
There is a natural identification:
\[f_\Gamma^*\mot(\Gamma_{k},\gamma,d)\simeq\mot(\Gamma,\gamma,d).\]
In particular, $\mot(\Gamma,\gamma)$ extends to the open subset $f_\Gamma^{-1}(K_k^{\textrm{gen}})$.
\end{proposition}
\subsection{Graphical description of weight-graded parts}
We translate our previous computations of weight-graded parts to the motive of a quotient graph $\Gamma$ of $\Gamma_n$.
\begin{theorem}
\label{theorem: diagrammatic weight-graded computation}
Let $0\leq2k\leq d$ be an even integer. Then the pinching and cutting morphisms induce an isomorphism:
\[\gr^W_{2k}\mot'(\Gamma,d)\simeq \bigoplus_{\gamma\subset E_{\Gamma}, \# \gamma = 2k} \mot'(\Gamma/\gamma^c,\gamma,d ).\]
If $k\neq 1$, then similarly:
\[\gr^W_{2k} \mot''(\Gamma,d)\simeq \bigoplus_{\gamma\subset E_{\Gamma}, \# \gamma = 2k-1} \mot''(\Gamma/\gamma^c,\gamma,d).\]
If $k=1$, then:
\[\gr^W_{2}\mot''(\Gamma,d)\simeq  \Ker(s_\Gamma)\]
where $s_\Gamma$ is the sum morphism:
\[s_\Gamma:\mathbb{Q}(-1)^{E_\Gamma}\to \mathbb{Q}(-1).\]
These are the only non-zero weight-graded parts. Finally:
\[\gr^W_\bullet \mot(\Gamma,d)\simeq \gr^W_\bullet\mot'(\Gamma,d) \oplus \gr^W_\bullet \mot''(\Gamma,d).\]
\end{theorem}
\begin{proof}
It is a corollary of Proposition~\ref{proposition: computation weight-graded first case} and Proposition~\ref{proposition: computation weight-graded second case}.
\end{proof}
\begin{remark}
Proposition~\ref{proposition: computation weight-graded first case} can also be applied to compute the weight-graded parts of the motive of a cut graph.
\end{remark}
\begin{corollary}
\label{corollary: fixed d}
Let $d\leq d'$ be even integers. The pinching morphisms induce an isomorphism:
\begin{align}
    \label{equation: weight d as colimit of quotient graphs}
\colim_{\#\gamma\leq d}(\mot(\Gamma_n/\gamma^c))&\to W_d\mot(\Gamma_n,d')
\end{align}
where the colimit is over the poset of subsets of edges $\gamma$ with at most $d$ edges, and if $\gamma_0\subset \gamma_1$, the associated morphism is the pinching morphism. The same holds for the reduced and quotient motives.
\end{corollary}
\begin{proof}
The morphism from the colimit to the right-hand side is well-defined by Proposition~\ref{proposition: functoriality pinch and cut morphism} and Theorem~\ref{theorem: diagrammatic weight-graded computation}. It suffices to show that the morphism induces isomorphisms on the weight-graded parts. Note that taking weight-graded parts is an exact functor that commutes with colimits. Then apply Theorem~\ref{theorem: diagrammatic weight-graded computation} and commutativity properties of Proposition~\ref{proposition: functoriality pinch and cut morphism}.
\end{proof}
In order to relate the motives for large values of $d$ and smaller ones we make the following definition.
\begin{definition}
The projective space of $d$-generic kinematics is the open subset $\PKinematics_n^{d\textrm{-gen}}$ of $\PKinematics_n$ where for all subsets $I\subset \indexsetbar$ of cardinal at most $d+1$, the function $\mathcal{G}_I$ is invertible.
\end{definition}
The point is that $\PKinematics_n^{d\textrm{-gen}}$ contains both $\PKinematics_{n,d}^{\textrm{gen}}$ and $K_n^{\textrm{gen}}$ as subvarieties.
\begin{corollary}
\label{corollary: extension of Wk mot to bigger open subset}
The motivic local system $W_{d}\mot(\Gamma,\gamma)$ on $\PKinematics_{n}^{\textrm{gen}}$ extends to $\PKinematics_n^{d-{\textrm{gen}}}$. Moreover, its restriction to $\PKinematics_{n,d}^{\textrm{gen}}$ is naturally isomorphic to $\mot(\Gamma,\gamma,d)$. The same holds for the reduced and the quotient motive.
\end{corollary}
\begin{proof}
For the first part, use Corollary~\ref{corollary: fixed d} to write $W_{d}\mot(\Gamma_n)$ as a colimit of motivic local systems that extend to $\PKinematics_n^{d-{\textrm{gen}}}$ by Proposition~\ref{proposition: smaller parameter space quotient graph}. For the second part, use Corollary~\ref{corollary: fixed d} for $d'=d$ and $d'$ large enough. It shows that both $\mot(\Gamma,\gamma,d)$ and $W_{d}\mot(\Gamma,\gamma)$ can be written as the same colimit.
\end{proof}
Therefore, from now on, we will mainly consider the large $d$ case. Smaller values of $d$ can be retrieved from the weight filtration.
\subsection{Tadpole}
The tadpole graph $\Gamma_1$ is a bit special. Notice first that the integral only depends on one parameter, the mass of the edge, hence there are no dimensionless parameters. Moreover, the corresponding integrals with simple poles are divergent in all dimensions $d>0$.

At the level of the motive, we compute that for all even $d>0$:
\[\mot(\Gamma_1,d)\simeq\mot'(\Gamma_1,d)\simeq \mathbb{Q}(0).\]
In particular:\[\mot''(\Gamma_1,d)=0.\]
Moreover,
\[\mot'(\Gamma_1,\{e_1\},d)=0\]
while
\begin{equation}
    \label{equation: cut tadpole motive}
\mot(\Gamma_1,e_1,d)\simeq \mot''(\Gamma_1,e_1,d)\simeq \mathbb{Q}(-1).
\end{equation}
\subsection{Maximal cut motives}
Assume that $\Gamma$ has $k$ edges.
\begin{proposition}
\label{proposition: motives of max cuts}
If $k$ is even, then:
\begin{equation}
    \label{equation: chi first case}
\mot'(\Gamma,E_{\Gamma})\simeq H^0(\mathcal{O}_{\PKinematics_n}[\sqrt{(-1)^{k/2}\G_{E_\Gamma}}])^-.
\end{equation}
If $k$ is odd, then:
\begin{equation}
    \label{equation: chi second case}
\mot''(\Gamma,E_{\Gamma})\simeq H^0(\mathcal{O}_{\PKinematics_n}[\sqrt{(-1)^{(k+1)/2}\G_{E_\Gamma\cup{\infty}}}])^-.
\end{equation}
\end{proposition}
\begin{proof}
It follows from Corollary~\ref{corollary: weight graded normalized motive first case}.
\end{proof}
\begin{remark}
    The case $n=1$ is special, because $\G_{\pm}(u_1,u_\infty)=1$.
\end{remark}
\begin{remark}
On $\PKinematics_n^{\textrm{gen}}$ we are taking square roots of a section of an even power of the tautological line bundle, which is well-defined. It explains why only Gram determinants of an even number of vectors appear.
\end{remark}
\subsection{Basis of de Rham cohomology}
For any motive $M$ of mixed Artin--Tate type, there is a canonical isomorphism
\begin{equation}
\label{equation: isomorphism de rham weight graded}
   M_{\derham}\simeq \left(\gr^W_{\bullet}M\right)_{\derham}
\end{equation}
defined using the Hodge and the weight filtration. Moreover, the de Rham realisation of motives of maximally cut graphs at generic kinematics has a canonical generator by proposition \ref{proposition: motives of max cuts}. Indeed the de Rham realisation of Tate twists is trivial, and for any field $\field\subset \mathbb{C}$ and  $a\in \field^*$, we have that:
\[H^0_{\derham}(k[t]/(t^2-a))^-=k\cdot t.\]
If $\#\gamma$ is even we let $\omega_{\Gamma_n/\gamma^c}$ denote the basis element of $\mot'(\Gamma_n/\gamma^c,\gamma)_{\derham}$. If $\#\gamma$ is odd we let $\omega_{\Gamma_n/\gamma^c}$ denote the basis element of $\mot''(\Gamma_n/\gamma^c,\gamma)_{\derham}$.
\begin{remark}
These de Rham classes are homogeneous in the kinematics $(\kinematics,\masses^2)$, because the functions $\G_{\pm}$ are. The global de Rham realisation of the motives of maximal cut graphs of Proposition~\ref{proposition: motives of max cuts} on $\PKinematics_{n,d}^{\textrm{gen}}$ should be a power of the tautological line bundle.
\end{remark}
By Theorem~\ref{theorem: diagrammatic weight-graded computation} and isomorphism \eqref{equation: isomorphism de rham weight graded}, if $\#\gamma\neq 1$ then we will consider $\omega_{\Gamma_n/\gamma^c}$ as an element of $\mot(\Gamma_n)_{\derham}$. If $\#\gamma=1$, then, for every family $(\lambda_i)_{1\leq i \leq n}$ of rational numbers of sum zero, we will consider the (formal) linear combination
\[\sum_{i=1}^n \lambda_i \omega_{\Gamma/{e_i}^c}\]
as an element of $\mot(\Gamma_n)_{\derham}$. In particular, we set the following notation for $i,j\in \indexset$:
\begin{equation}
    \label{equation: definition omega ij}
\omega_{i,j}:=\omega_{\Gamma/{e_j}^c}-\omega_{\Gamma/{e_i}^c}.
\end{equation}
They satisfy for every $i,j,k\in \indexset$: \[\omega_{i,j}+\omega_{j,k}=\omega_{i,k}.\]
Moreover, the cohomology class $\omega_{i,j}$ comes from $\mot(\Gamma/\{e_i,e_j\}^c)_{\derham}$. We can now state the following corollary to Theorem~\ref{theorem: diagrammatic weight-graded computation} and isomorphism~\eqref{equation: isomorphism de rham weight graded}.
\begin{corollary}
\label{corollary: basis of de rham cohomology}
The cohomology classes $\omega_{\Gamma_n/\gamma^c}$ for subsets $\gamma\subset E_{\Gamma_n}$ of even cardinal form a basis of $\mot'(\Gamma_n)_{\derham}$.

Fix $i\in \indexset$. A basis of $\mot(\Gamma_n)_{\derham}$ is given by the cohomology classes $\omega_{\Gamma_n/\gamma^c}$ for subsets $\gamma\subset E_{\Gamma_n}$ of cardinal different than $1$, together with classes $\omega_{i,j}$ for $j\in \indexset$ .
\end{corollary}
We can be more explicit about the splitting \eqref{equation: isomorphism de rham weight graded}. The Hodge filtration is defined using logarithmic differential forms, and the highest degree is given by global logarithmic differential forms. We encountered such differential forms in the case where the number of edges was equal to the dimension, or to the dimension minus $1$. This yields the following proposition. 
\begin{proposition}
\label{proposition: identification de rham class}
If $\#\gamma$ is different than $1$, then up to a prefactor the class $\omega_{\Gamma_n/\gamma^c}$ is the de Rham cohomology class of the logarithmic differential form in dimension $d=2\lceil \#\gamma/2 \rceil$:
\[\frac{d^{\mathrm{d}}k}{\prod_{i\in \gamma }\propagator_i}.\]
For $i,j\in \indexset$, up to a prefactor the class $\omega_{i,j}$ is the class of the logarithmic differential form in dimension $d=2$:
\[\frac{d^2k}{\propagator_j}-\frac{d^2k}{\propagator_i}.\]
\end{proposition}
\begin{proof}
It follows from the remark we just made.
\end{proof}
We could also ask for an expression of the de Rham cohomology classes in a fixed dimension $d$, rather than making it vary depending on the number of edges. This should be computable from careful examination of propositions~\ref{proposition: orthogonal decompositions quadrics cohomology} and \ref{proposition: orthogonal decompositions quadrics cohomology bis}, or by going through other representations of Feynman integrals.
\subsection{Betti class and motivic periods}
By Proposition~\ref{proposition: real points belong to minus part} and Proposition~\ref{proposition: real points belong to minus part second case}, for Euclidean kinematics, the real point locus $X(\mathbb{R})$ defines a Betti class in the full and reduced motives of each quotient graph. Indeed, if we write the reduced motive as \[H_d^\betti(U)^-\otimes (H_d^\betti(X)^-)^\lor\] then we can define the Betti class $\sigma$ as:
\[[X(\mathbb{R})]\otimes [X(\mathbb{R})]^\lor.\]
It does not depend the orientation of $X(\mathbb{R})$. Moreover, Lemma~\ref{lemma: orthogonal decomposition real points fundamental class} and \ref{lemma: variation orthogonal decomposition real points fundamental class} ensure that the Betti class is well-defined, independently of $d$.
\begin{definition}
\label{definition: motivic period}
Let $\Gamma$ be a quotient graph of $\Gamma_n$ with $k\geq 2$ edges, and let $\kinematics,\underline{m}^2$ be Euclidean generic kinematics. We attach to it the motivic period:
\begin{equation*}
    I^{\fmot}(\Gamma,\kinematics,\underline{m}^2):=[\mot(\Gamma,\kinematics,\masses),\sigma,\omega_{\Gamma}]
\end{equation*}
\end{definition}
If $k$ is even, then $I^{\fmot}(\Gamma,\kinematics,\underline{m}^2)$ is also a motivic period of the reduced motive. Moreover, up to a prefactor, it evaluates to $I(\Gamma,\kinematics,\masses^2,d=2\lceil n/2 \rceil,\underline{\nu}=\underline{1})$. More generally, we have the following proposition.
\begin{proposition}
\label{proposition: I(d,nu) is period of motive}
For all Euclidean generic kinematics $(\kinematics,\masses^2)$ and $\underline{\nu}\in \mathbb{Z}^n$ such that $\nu>d/2$, $I(\Gamma,\kinematics,\masses^2,d,\underline{\nu})$ is a period of $W_d\mot(\Gamma_n,d)_{(\kinematics,\masses^2)}$. 
\end{proposition}
\begin{proof}
It follows from Corollary~\ref{corollary: extension of Wk mot to bigger open subset}, Proposition~\ref{proposition:definition of motivic period} and Proposition~\ref{proposition:definition of motivic period second case}. The only thing to check is that only the ``$-$'' part of the cohomology of $U$ is relevant, in the notations of these two propositions. Proposition~\ref{proposition: real points belong to minus part} and Proposition~\ref{proposition: real points belong to minus part second case} ensure that the Betti class vanishes on the ``$+$'' part of the cohomology of $U$.
\end{proof}
\subsection{Dual de Rham classes and de Rham periods}
Once we have a basis of the de Rham cohomology, we can take the dual basis and define de Rham periods as pairings of elements in the basis and the dual basis. We can then express the de Rham motivic coaction using these de Rham periods.

If $\#\gamma$ is even we let $\varphi_{\Gamma_n,\gamma}$ denote the basis element of $\mot'(\Gamma_n/\gamma^c,\gamma)_{\derham}^\lor$. If $\#\gamma$ is odd we let $\varphi_{\Gamma_n,\gamma}$ denote the basis element of $\mot''(\Gamma_n/\gamma^c,\gamma)_{\derham}^\lor$.
By Theorem~\ref{theorem: diagrammatic weight-graded computation} and isomorphism \eqref{equation: isomorphism de rham weight graded}, we will consider $\varphi_{\Gamma_n,\gamma}$ as an element of $\mot(\Gamma_n)_{\derham}^\lor$. If $\#\gamma=1$, then,
\[\sum_{i=1}^n \varphi_{\Gamma_n,e_i}=0.\]
When $\gamma$ is not a single edge, $\varphi_{\Gamma_n,\gamma}$ corresponds to the dual basis element $\omega_{\Gamma_n/\gamma^c}^\lor$. Moreover:
\[\varphi_{\Gamma_n,e_i}(\omega_{j,k})=\delta_{i,j}-\delta_{i,k}.\]
\begin{definition}
    We define the de Rham period of the cut graph $(\Gamma_n,\gamma)$ as
    \begin{equation}
        \label{equation: definition of de Rham period}
        I(\Gamma_n,\gamma)^{\fderham}=[\mot(\Gamma_n),\omega_{\Gamma_n},\varphi_{\Gamma_n,\gamma}]
    \end{equation}
\end{definition}
We have the relation:
\begin{equation}
    \label{equation: one edge cut relation}
\sum_{i=1}^n I^{\fderham}(\Gamma_n,e_i)=0.
\end{equation}
\begin{corollary}
\label{corollary: de Rham period of M or M'' or 0}
If $n$ is even, then $I^{\fderham}(\Gamma_n,\gamma)$ is a de Rham period of $\mot'(\Gamma_n,\gamma)$.
If $n$ is odd, then $I^{\fderham}(\Gamma_n,\gamma)$ is a de Rham period of $\mot''(\Gamma_n, \gamma)$.
If both conditions are satisfied then $I^{\fderham}(\Gamma_n,\gamma)=0$.
\end{corollary}
\begin{proof}
    It follows from the definition and from the equation $c_\infty\circ p_\infty=0$.
\end{proof}
\subsection{Diagrammatic motivic coaction formula}
Before stating the theorem, we make the following convention. If $\lambda_i$ for $i\in \indexset$ are functions on the space of kinematics of zero sum, then we set:
\begin{equation}
\label{equation: zero sum convention}
\sum_{i=1}^n \lambda_i I^{\mathfrak{m}}(\Gamma_n/\{e_i\}^c):=[\mot(\Gamma_n),\sum_{i=1}^n \lambda_i \omega_{\Gamma_n/\{e_i\}^c},\sigma].
\end{equation}
Therefore, even though each of the terms $I^{\mathfrak{m}}(\Gamma_n/\{e_i\}^c)$ is ill-defined, we can make sense of their linear combinations provided the sum of coefficients is zero. Similarly, if $R$ is a $\mathbb{Q}$ algebra and $\lambda_i$ in $R$ are of sum zero, we choose any $j$ in $\indexset$ and we define:
\begin{equation}
    \label{equation: zero sum tensor convention}
\sum_{i=1}^n I^{\mathfrak{m}}(\Gamma_n/\{e_i\}^c)\otimes_\mathbb{Q}\lambda_i:=
\sum_{i=1}^n(I^{\mathfrak{m}}(\Gamma_n/\{e_i\}^c)-I^{\mathfrak{m}}(\Gamma_n/\{e_j\}^c)) \otimes_\mathbb{Q} \lambda_i \in R\otimes_\mathbb{Q}\motper(\goodK).
\end{equation}
It is clearly independent of $j$.
\begin{theorem}
\label{theorem: motivic coaction formula}
The de Rham motivic coaction is given by the formula:
\[\rhomot I^{\mathfrak{m}}(\Gamma_n)= \sum_{\gamma\subset E_{\Gamma_n}} I^{\mathfrak{m}}(\Gamma_n/\gamma^c) \otimes I^{\mathfrak{dr}}(\Gamma_n,\gamma).\]
If $n$ is even, then the terms with $\#\gamma$ odd cancel. If $n$ is odd, then we use convention \eqref{equation: zero sum tensor convention} for the sum of terms with $\#\gamma=1$.
\end{theorem}
\begin{proof}
Fix $i\in \indexset$ and use the formula for the coaction in the de Rham basis of Corollary~\ref{corollary: basis of de rham cohomology}. The dual de Rham basis is given by the $\varphi_{\Gamma_n,\gamma}$ for $\gamma\subset E_{\Gamma_n}$ of cardinal different than $1$, and by $\varphi_{\Gamma_n,e_j}$ for $j\neq i$. Use Corollary~\ref{corollary: de Rham period of M or M'' or 0} to get the cancellation when $n$ is even.
\end{proof}
\begin{theorem}
\label{theorem: motivic coproduct formula}
The de Rham motivic coproduct is given by the formula:
\[\Deltadr I^{\fderham}(\Gamma_n,\gamma)= \sum_{\gamma\subset \gamma'} I^{\mathfrak{dr}}(\Gamma_n/\gamma'^c,\gamma) \otimes I^{\mathfrak{dr}}(\Gamma_n,\gamma').\]
\end{theorem}
\begin{proof}
The same arguments apply.
\end{proof}
\appendix
\section{Cohomology with support and 6 functors}
\label{section: cohomology and 6 functors}
We work over $S$ a variety over $\mathbb{Q}$ and use the oriented 6 functors formalism, such as satisfied by an oriented motivic triangulated category \cite{cisinski_triangulated_2019}. Oriented here means that for a smooth morphism $f$ of constant relative dimension $d$ there is an isomorphism:
\[f^!\simeq f^*(d)[2d]\]
satisfying natural compatibilities. All schemes will be of finite presentation, separated and reduced over $S$. Consider $p:X\to S$. Then 6 functors formalism make it possible to make the following usual definitions:
\[\begin{array}{cc}
    H^\bullet(X)=p_*p^*1_S & H_c^\bullet(X)=p_!p^*1_S \\
    H_\bullet(X)=p_!p^!1_S & H_\bullet^{\text{lf}}(X)=p_*p^!1_S
\end{array}\]
These are the motives corresponding to the cohomology, the compactly supported cohomology, the homology, and the locally finite homology of $X$.
Now assume that we are also given a closed subscheme $i:Z\to X$, with complementary open $j:U\to X$. We can define:
\[\begin{array}{cc}
    H^\bullet(X,Z)=p_*j_!j^*p^*1_S & H_c^\bullet(U/X)=p_!j_*j^*p^*1_S \\
    H_\bullet(X,Z)=p_!j_*j^!p^!1_S & H_\bullet^{\text{lf}}(U/X)=p_*j_!j^!p^!1_S
\end{array}\]
The left column corresponds to relative cohomology and homology of the pair $(X,Z)$. The right column is less usual, we will call its entries the cohomology of $U$ with compact support in $X$ and the locally finite over $X$ homology of $U$. The notation is \emph{ad hoc}.
Before proceeding, we note that it is not necessary to assume that $j$ is an open immersion. We may drop this hypothesis and define for any $f:X\to Y$ the cohomology of $X$ with compact support in $Y$: \[H_c^\bullet(X/Y)=p_!f_*f^*p^*1_S.\]
We may define the homology of $X$ with proper support over $Y$ similarly. We make the following observations.
\begin{lemma}
\label{lemma: factorisation by proper map isomorphism}
Consider a commutative diagram:
\[\begin{tikzcd}
X \ar[r,"f"] \ar[rd,"h"] & Y \ar[d,"g"] \\
 & Y'
\end{tikzcd}\]
where $g$ is proper. Then $g$ induces an isomorphism:
\[g^*:H^\bullet(X/Y')\simeq H^\bullet(X/Y).\]
\end{lemma}
\begin{proof}
Use that $g_*\simeq g_!$ because $g$ is proper.
\end{proof}
The next corollary shows that we can always reduce to the case of open immersions.
\begin{corollary}
Consider a morphism $f:X\to Y$. Then there exists a factorisation:
\[\begin{tikzcd}
X \ar[r,"j"] \ar[rd,"f"] & \bar{X} \ar[d,"g"] \\
 & Y
\end{tikzcd}\]
with $g$ proper and $j$ an open immersion. Hence there is an isomorphism:
\[g^*:H^\bullet(X/Y)\simeq H^\bullet(X/\bar{X}).\]
\end{corollary}
We now list functoriality properties of these variants of cohomological motives. In the following, $(X,U)$ is always a pair of a scheme with with an open subscheme. Moreover $Z$ is the complement of $U$. We introduce subscripts and superscripts when we need to consider several pairs, and we say that a morphism of pairs is cartesian is the underlying square is cartesian.
\begin{lemma}
\label{lemma: pullback cohomology with compact support in}
Let $f:(X',U') \to (X,U)$ be a morphism of pairs. Assume that $f$ is proper from $X'$ to $X$. Then there are natural pullback and pushforward morphisms
\[f^*:H_c^\bullet(U/X)\to  H_c^\bullet(U'/X')\]
and
\[f_*:H_\bullet^{\text{lf}}(U'/X') \to H_\bullet^{\text{lf}}(U/X).\]
This construction is compatible with composition.
\end{lemma}
\begin{lemma}
\label{lemma: pushforward cohomology with compact support in}
Let $g:(X_1,U_1)\to (X_0,U_0)$ be a cartesian morphism of pairs (i.e. such that $g^-1(U_0)=U_1$). Assume that $g$ restricted to $U_1$ is smooth of pure relative dimension $d_g$. Then there are natural morphisms
\[g_!:H^\bullet_c(U_1/X_1) \to H^{\bullet-2d_g}_c(U_0/X_0)(d_g).\]
and
\[g^!:H_\bullet^{\text{lf}}(U_0/X_0)\to  H_{\bullet+2d_g}^{\text{lf}}(U_1/X_1)(-d_g)\]
This construction is compatible with composition.
\end{lemma}
\begin{remark}
Lemma~\ref{lemma: pushforward cohomology with compact support in} also holds with the weaker hypothesis $g(Z_0)\subset Z_1$. This strengthened lemma is a direct consequence of Lemma~\ref{lemma: pushforward cohomology with compact support in} and Lemma~\ref{lemma: pullback cohomology with compact support in}.
\end{remark}
\begin{lemma}
\label{lemma: pullback relative cohomology}
Let $g:(X_1,U_1)\to (X_0,U_0)$ be a cartesian morphism of pairs. Then, there are natural morphisms
\[g^*:H^\bullet(X_0,Z_0)\to  H^\bullet(X_1,Z_1)\]
and
\[g_*:H_\bullet(X_1,Z_1) \to H_\bullet(X_0,Z_0).\]
This construction is compatible with composition.
\end{lemma}
\begin{remark}
Of course, the lemma also holds with the weaker assumption $g(Z_1)\subset Z_0$.\end{remark}
\begin{lemma}
\label{lemma: pushforward relative cohomology}
Let $f:(X',U')\to (X,U)$ be a morphism of pairs. Assume that $f$ is proper from $X'$ to $X$, and its restriction to $U'$ is smooth of pure relative dimension $d_g$. Then there are natural morphisms
\[f_!:H^\bullet(X',Z') \to H^{\bullet-2d_g}_c(X,Z)(d_g)\]
and
\[f^!:H_\bullet(X,Z)\to  H_{\bullet+2d_g}(X',Z')(-d_g).\]
This construction is compatible with composition.
\end{lemma}
\begin{lemma}
\label{lemma: pullback/pushforward relative cohomology}
The two constructions above are compatible in the following sense. Consider a cartesian square of pairs:
\[
\begin{tikzcd}
(X'_1,U'_1) \ar[r,"f_1"] \ar[d,"g'"] & (X_1,U_1) \ar[d,"g"]\\
(X'_0,U'_0) \ar[r,"f_0"] & (X_0,U_0)
\end{tikzcd}
\]
where $g$ satisfies the hypothesis of Lemma~\ref{lemma: pullback relative cohomology}, and $f_0$ satisfies the hypothesis of Lemma~\ref{lemma: pushforward relative cohomology}. Then the same holds for $g'$ and $f_1$, and there is a natural isomorphism between the morphisms given by the previous lemmas on homology:
\[f_0^!g_*\simeq g'_*f^{!}_1.\]
Similarly for cohomology.
\end{lemma}
\begin{remark}
There is a similar compatibility between Lemma~\ref{lemma: pullback cohomology with compact support in} and Lemma~\ref{lemma: pushforward cohomology with compact support in} which we did not state.
\end{remark}
The lemma below describes the relation between relative homology and cohomology with compact support in $X$, which we refer to as Poincaré duality. 
\begin{lemma}
\label{lemma: appendix Poincare duality relative cohomology and cohomology with compact support in}
Assume that $U$ is smooth of pure dimension $d$. Then there are Poincaré duality isomorphisms
\[\PD: H^\bullet_c(U/X)\simeq H_{2d-\bullet}(X,Z)(-d)\]
and
\[\PD: H_\bullet^{\text{lf}}(U/X)\simeq H^{2d-\bullet}(X,Z)(d).\]
Moreover, our constructions are compatible with each other with respect to this isomorphism. More precisely, if $f:X'\to X$ is proper on $X'$ and smooth on $U'$, $f(U')\subset U$ and $U$ and $U'$ are smooth, then $f^*$ is isomorphic to $f^!$ under Poincaré duality isomorphisms.

Similarly, if $g:X_1\to X_0$ is proper on $X_1$ and smooth on $U_1$, $U_0$ and $U_1$ are smooth and $g(Z_1)\subset Z_0$ then $g_!$ is isomorphic to $g_*$ under Poincaré duality isomorphisms.
\end{lemma}
\begin{remark}
This lemma can be used together with lemmas~\ref{lemma: pushforward relative cohomology}, \ref{lemma: pullback relative cohomology} and \ref{lemma: pullback/pushforward relative cohomology} to define Gysin morphisms for cohomology with compact support in $X$. It is harder to describe these morphisms when they do not reduce to the morphisms of lemmas \ref{lemma: pushforward cohomology with compact support in} and \ref{lemma: pullback cohomology with compact support in}.
\end{remark}
We also have Künneth decomposition in this context.
\begin{lemma}
\label{lemma: Künneth decomposition cohomology with compact support in}
Consider varieties $X/Y$ and $X'/Y'$, and form their product $X\times X'/Y\times Y'$. Then:
\[H^\bullet_c(X/Y)\otimes H^\bullet_c(X'/Y')\simeq H^\bullet_c(X\times X'/Y\times Y').\]
\end{lemma}
Finally, note that cohomology and cohomology with compact support appear as particular cases in our framework.
\begin{lemma}
\label{lemma: appendix X is U}
Assume $U=X$ and $U'=X'$. Then cohomology with compact support in $X$ is naturally isomorphic to cohomology with compact support, and similarly locally finite over $X$ homology is naturally isomorphic to locally finite homology.
\end{lemma}
\begin{lemma}
\label{lemma: appendix X proper}
Assume $X'$ and $X$ to be proper over $S$. Then cohomology with compact support in $X$ is naturally isomorphic to cohomology, and similarly locally finite over $X$ homology is naturally isomorphic to homology.
\end{lemma}
\begin{remark}
\label{remark: unnecessary compactifications}
In Lemma~\ref{lemma: appendix X proper}, rather than taking compactifications we can simply take $X$ and $X'$ to be equal to $S$. The hypothesis that $U\to X$ is an open immersion is actually unnecessary for the lemmas we stated. The only thing to change is that relative cohomology has to be replaced by the corresponding $6$ functors expression. In particular, if $X$ is equal to $S$, then relative cohomology (respectively relative homology) has to be replaced with cohomology with compact support (respectively locally finite homology).
\end{remark}
\newpage
\bibliographystyle{amsalpha}
\bibliography{Ma_bibliotheque_bis}
\end{document}